\RequirePackage{rotating}
\documentclass[]{amsart}
\usepackage{amsmath}
\usepackage[english]{babel}
\usepackage[utf8]{inputenc}
\usepackage{amsfonts}
\usepackage{amsthm}
\usepackage{amssymb}
\usepackage{multicol}
\usepackage{color}
\usepackage{bm}
\usepackage[pagewise]{lineno}
\usepackage{rotating}
\usepackage{tabularx}
\usepackage[labelfont=bf]{caption} % optional
\usepackage{ragged2e}
\usepackage[table,xcdraw]{xcolor}
\usepackage{makecell}
\numberwithin{equation}{section}

\newtheorem{thm}{Theorem}[section]
\newtheorem{defn}{Definition}[section]

\newtheorem{prop}[thm]{Proposition}
\newtheorem{cor}[thm]{Corollary}
\newtheorem{lem}[thm]{Lemma}

\theoremstyle{remark}
\newtheorem{rmk}[thm]{Remark}
\theoremstyle{definition}

\DeclareMathOperator{\E}{\mathbb{E}}

\DeclareMathOperator{\N}{\mathbb{N}}

\DeclareMathOperator{\R}{\mathbb{R}}
\DeclareMathOperator{\cF}{\mathcal{F}}

\DeclareMathOperator{\bP}{\mathbb{P}}

\newcommand{\der}[2]{\frac{d #1}{d #2}}
\newcommand{\dersup}[3]{\frac{d^{#3} #1}{d #2^{#3}}}

\title[Scaling limits for Mittag-Leffler queues]{Scaling limits for some  Mittag-Leffler queues}
\author{Giacomo Ascione$^\dagger$}
\address{$^\dagger$ Scuola Superiore Meridionale, Via Mezzocannone, 4 80134 Napoli - Italy}
\author{Luigia Caputo$^\ast$}
\address{$^\ast$ Dipartimento di Matematica e Applicazioni ``Renato Caccioppoli'', Università degli Studi di Napoli Federico II, 80126 Napoli, Italy}
\begin{document}
\maketitle
\begin{abstract}
In this paper, we consider five models of heavy-tailed queues involving Mittag-Leffler distributions that generalize the classical $M/M/1$ queues. These models are suitable modifications of previously defined models in such a way that the classical $M/M/1$ queue can be recovered by a suitable selection of parameters. We provide the distribution of inter-arrival and service times of both the original and modified queueing models. We then study the scaling limits of all the proposed models and we argue that the behaviour of the limiting processes can be used to characterise the traffic regime of the queues.
\end{abstract}
%\keywords{}\\
%\subclass{}
\section{Introduction}
Queueing theory plays a prominent role in the applications nowadays. Starting from their introduction in the telecommunications field \cite{erlang1909}, queueing systems were adopted to handle several modelling issues, such as in operating systems \cite{silberschatz2013operating}, health-care \cite{mayhew2006using}, epidemiology \cite{hernandez2010application}, production facilities \cite{carlson1992simulation} and economics \cite{radivojevic2014ergodic}, among others. Through time, a quite special attention has been given to the $M/M/1$ queue (in Kendall's notation), that exhibits exponential inter-arrival and service time. For this reason, this kind of model has a purely Markov nature and thus can be studied with the typical arguments of the theory of continuous time Markov Chains (see, for instance, \cite{sharma1990}). Nevertheless, several important results have been achieved even in a more general setting. For $M/G/1$ queues, with the only additional assumption that the service times admit finite moments, Pollaczek and Khintchine (in \cite{pollaczek1930aufgabe,khintchine1932}) obtained a formula for the mean queue length, now known as the Pollaczek-Khintchine formula. An approximation of the mean waiting time for a $G/G/1$ queue has been instead obtained in \cite{kingman1961single}, still under the assumption that the inter-arrival and service times admit finite moments.

Queueing systems with heavy-tailed inter-arrival or service times require, however, a further study. Evidence of heavy-tailed waiting times has been found, to give an example, in finance \cite{sabatelli2002waiting,raberto2002waiting}. Nevertheless, results on the impact of heavy-tailed inter-arrival or service times on the performance of the queueing system have been obtained more recently (see, for instance, \cite{whitt2000impact} and references within). 

Among all possible heavy-tailed distributions, the Mittag-Leffler distribution plays a major role. Indeed, the Mittag-Leffler function represents a quite natural extension of the exponential, that, at the same time, exhibits a power-type asymptotic behaviour on the negative real axis. The related distribution arises as the Laplace transform of the inverse of a stable subordinator (see \cite{bingham1971limit}) and, for such a reason, can be identified as the composition of a stable subordinator with an independent exponential random variable. This identification can be used, for instance, to characterise the inter-jump times of some stepped semi-Markov processes constructed by means of a time-change of an original continuous time Markov chain through the inverse of a stable subordinator. Among these processes, it is important to cite the fractional Poisson process \cite{mainardi2004fractional,laskin2003fractional}, that was originally introduced as the counting process for i.i.d. Mittag-Leffler random variables and then was recognized as a proper time-change of the classical Poisson process \cite{meerschaert2011fractional}. Another important feature of the Mittag-Leffler distribution is its relation with fractional calculus, in particular with the Caputo derivative of which it is an eigenfunction. With this observation in mind, several properties of these semi-Markov processes can be studied through their governing equations by means of Fourier-type spectral decompositions \cite{ascione2021fractional,ascione2021transient}.

A first model Mittag-Leffler queueing system was considered in \cite{cahoy2015transient} and then further studied in \cite{ascione2018fractional,ascione2022skorokhod}. This model, called the fractional $M/M/1$ model, exhibits non-independent Mittag-Leffler inter-arrival and service times and is obtained by means of a time-change procedure of a classical $M/M/1$ queue with the inverse of an $\alpha$-stable subordinator, for $\alpha \in (0,1]$. The relation between the time-change procedure and the solution of fractional abstract Cauchy problems (see \cite{baeumer2001stochastic}) allows for a direct study of its transient behaviour, with an explicit determination of the one-dimensional marginals. Furthermore, this model reduces to the classical $M/M/1$ as $\alpha=1$. Other two models of Mittag-Leffler queueing systems were introduced in \cite{butt2023queuing}, in the form of a renewal process and the reflection of a fractional Skellam process \cite{kerss2014fractional}. In both processes, inter-arrival and service times are related to different Mittag-Leffler distributions, with possibly different order $\alpha,\beta \in (0,1]$. When $\alpha=\beta=1$, these models coincide with a classical $M/M/1$ until it empties. 

All these models share a rather peculiar property: unless $\alpha=1$ or $\beta=1$, inter-arrival and service times have infinite expectation. For such a reason, one cannot compare the mean inter-arrival and mean service time to deduce the \textit{regime} of the queue. Indeed, in the classical $M/M/1$ setting, positive/null recurrence and transience of the states can be identified by comparing the expected inter-arrival and service times, respectively $\lambda^{-1}$ and $\mu^{-1}$. When $\lambda<\mu$, then the queue is in its subcritical regime, i.e. all the states are positive recurrent and the queue length process admits a stationary distribution. When $\lambda>\mu$, i.e. in the supercritical regime, then arrival are much more frequent than services and therefore congestion occurs, so that all the states are transient. When $\lambda=\mu$, we are in the critical regime, in which inter-arrivals and services balance each other. We still lack a stationary distribution, but the states are now null recurrent. This study cannot be carried out in this way for a Mittag-Leffler queue. 

Nevertheless, the distinction between the regimes of the classical $M/M/1$ queue can be also carried out in terms of scaling limits. In the supercritical regime, upon applying a proper space-time scaling, the frequency of the arrivals is highlighted by the fact that the scaling limit of the queue length process is indeed a monotone non-decreasing function. In the critical regime, we need to further \textit{zoom in space} to see some oscillations. Finally, in the subcritical regime, even with this finer space scale, the queue length process converges to $0$.

In this paper, we consider the Mittag-Leffler queueing models introduced in \cite{butt2023queuing} and some suitable modifications, in such a way that these new modified queues actually extend the classical $M/M/1$ queue even after they become empty, while preserving the Mittag-Leffler nature of the inter-arrival and service times. The main aim of the present paper is to determine some scaling limits for such Mittag-Leffler queues (in the style of \cite[Theorem 3.3]{butt2023queuing}) and then use these limits to characterise the regimes of the aforementioned queues. 

The organization of the paper is as follows. In Section~\ref{sec:models} we introduce the considered Mittag-Leffler queueing models. Precisely, we first set, in Section~\ref{subs:ssqs}, the general notation we are going to use for single server queueing models. Several equivalent constructions of the standard $M/M/1$ queue are given in Section~\ref{subs:MM1}: this is done, in particular, to provide a standard reference queueing system. Mittag-Leffler queues will be then constructed by generalizing the different constructions of the $M/M/1$ queue: it will be clear by the distribution of inter-arrival and service times that these constructions, leading to the same queue in the Markovian case, now characterize quite different systems. Sections~\ref{subsmod1} to \ref{sec:Mod5} are dedicated to the construction of the Mittag-Leffler queues and the identification of the inter-arrival and service times distributions, together with some other properties. To streamline the presentation, more technical preliminary results and tools are given in Appendix~\ref{app:A}, while the proofs of the statements of these sections are provided in Appendix~\ref{app:B}. Section~\ref{sec:scaling} is instead devoted to the statements of the main theorems and the regimes of the constructed queueing systems. In Section~\ref{sec:MM1} we provide, for comparison, the statement of the scaling limits of the classical $M/M/1$ queue, together with the relation between the scaling limits and the traffic intensity. Nevertheless, since the Mittag-Leffler queues do not exhibit a proper traffic intensity, we will recognize the regimes of the queues in terms of the scaling limits, in analogy of the aforementioned relation of the $M/M/1$ queue. In Sections~\ref{sec:MM11} to \ref{sec:MM15}, for each Mittag-Leffler queue, we provide the statement of the scaling limits and the interpretation of the related regimes. The proof of such statements are then delayed to Section~\ref{sec:scalingproof}. Precisely, in Section~\ref{sec:MM1proof} we underline some details concerning the proof of the scaling limits of the $M/M/1$ queue, in particular in the subcritical regime $\rho<1$. Indeed, in such a case, one can use the limit distribution to improve the convergence of the queue length process to $0$. We give all the steps in order to compare such a situation with the Mittag-Leffler setting, in particular in the cases in which we are not able to improve the convergence in the subcritical case. Section~\ref{subsecMC} is devoted to some auxiliary (but independently interesting) scaling limits results for some related discrete-time Markov chains. These will be used in the proofs of the main scaling theorems, that are instead given in Sections~\ref{sec:proof1} to \ref{sec:proof5}. While some of the proofs are similar, each proof generally requires a different strategy. All the proofs make explicit extensive use of the continuous mapping theorem, but preliminary convergence results, that could be different for each queueing model, have to be proven.

%The paper is structured as follows: in Section~$2$ we introduce the main notation and provide some preliminaries both on the functional setting and on the involved processes. In Section~$3$ we introduce the Mittag-Leffler queueing models and their main characteristics in terms of distribution of inter-arrival and service times. Finally, Section~$4$ is completely devoted to the scaling limits: the main theorems are proved and the characterization of the regimes of the various queueing models is addressed. Proofs of some preliminary technical results are provided in Appendix.

\section{The Queueing System Models}\label{sec:models}
{In this section we introduce the main queueing system models we are going to study. Before doing that, we recall the main feature of single server queueing systems and of the classical $M/M/1$ queueing system. Then we move to a systematic description of the considered models, with particular attention to the distribution of inter-arrival and service times. In particular, we will focus on the three models of Mittag-Leffler queues introduced in \cite{butt2023queuing,cahoy2015transient} and we will introduce two further models.
%Furthermore, we will recall (or determine, whenever they are not available in the literature) the distribution of inter-arrival and service times for each of these queues. 
%This information will be useful to determine the \textit{regime} of the queueing model, in the next section. 
To recognize each model when comparing them, we will add a preceding subscript to denote the related queue length process, i.e., ${}_jQ$, for $j=1,2,3,4,5$, will denote the queue length process of model $j$, while $Q$ will generally denote the queue length process of a classical $M/M/1$ queue. Furthermore, we will only state the results concerning the distribution of inter-arrival and exit times: their proofs will be given in Appendix \ref{app:B}. Throughout the paper we denote by $\mathbb{D}$ the space of c\'adl\'ag (i.e. right-continuous with left limits) functions $f:[0,+\infty) \to \R$.} {Throughout the paper, we will assume that the queueing system is empty at time zero except if differently specified.}
%, that will be referred as the Mittag-Leffler GI/GI/1 queue
\subsection{Single server queueing systems}\label{subs:ssqs}
Let us recall the definition and some of the main features of single server queueing systems. These systems can be described as follows: we consider two sequences $(T_k)_{k \ge 1}$ and $(S_k)_{k \ge 1}$ of non-negative random variables. The first sequence represents the \textbf{inter-arrival times}, i.e., $T_k$ will be the time that occurs between the arrival of the $k-1$-th and the $k$-th customer. The second sequence represents the service times, i.e., $S_k$ is the time needed to serve the $k$-th customer. We are assuming that there is only one server, which means that the $k$-th customer will have to wait until all the previous customers are served. The arrival mechanism can be described completely by the sequence $(T_k)_{k \ge 1}$, by introducing the \textbf{arrival times} $(A_k)_{k \ge 0}$, where $A_0=0$ and $A_k=A_{k-1}+T_k$ for $k \ge 1$. Concerning the \textit{departures} from the queueing system, one needs to consider the \textbf{cumulative service times} $(CS_k)_{k \ge 0}$, where $CS_0=0$ and $CS_k=CS_{k-1}+S_k$. The respective counting processes, called the \textbf{arrival} and the \textbf{cumulative service counting processes}, are defined as
\begin{equation*}
	N^{\sf a}(t)=\max\{k \ge 0: \ A_k \le t\}, \qquad N^{\sf d}(t)=\max\{k \ge 0: \ CS_k \le t\}.
\end{equation*}
% and the \textbf{cumulative services counting process}
%\begin{equation*}
%	N^{\sf d}(t)=\max\{k \ge 0: \ CS_k \le t\},
%\end{equation*}
Notice that while $N^{\sf a}(t)$ is the number of customers that entered the system up to time $t$, $N^{\sf d}(t)$ does not describe the number of customers served up to time $t$. Indeed, we could still have some \textit{idle periods} in which no customer enters the queue and thus the server does not work: these idle periods are not considered in the evaluation of $N^{\sf d}(t)$. However, we can still use $N^{\sf d}$, up to some modifications, to measure the number of departures. Indeed, if we are able to construct a process $B:=\{B(t), \ t \ge 0\}$ that runs like the time $t$ whenever the queue is not empty but stops when the queue is empty, then $N^{\sf d}(B(t))$ will actually count the number of departures at time $t$. To create such a process, we can consider the following quantities for $t \ge 0$
\begin{equation}\label{eq:CIP}
	C(t)=CS_{N^{\sf a}(t)}, \quad 	X(t)=C(t)-t, \quad L(t)=\Phi(X)(t), \quad B(t)=C(t)-L(t),
\end{equation}
{where $\Phi$ is \textbf{Skorokhod reflection map}, defined as
\begin{equation*}
	\Phi(f)(t)=f(t)+\sup_{0 \le s \le t}\max\{-f(s),0\}, t \ge 0, \ f \in \mathbb{D}.
\end{equation*}
For further details on Skorokhod's reflection map, see Appendix \ref{sec:SRM}.}
The first process in \eqref{eq:CIP} is called the \textbf{cumulative-input process} and represents how much time is needed to serve all the customers that entered the queue up to time $t$. Clearly, since we are already at time $t$, the actual \textit{net} time we need to serve the remaining customers is given by $X$, that is called the \textbf{net-input process}. The net-input process is only allowed to jump upward, while it \textit{moves with continuity} downward. The process $L$ is called the \textbf{workload} of the queue and actually describes when the server is working. Indeed, it is not difficult to check that the server is working only when $L(t)>0$, while it is waiting for a customer to arrive if $L(t)=0$. Indeed, $L(t)$ jumps upward each time a customer enters the queue, with a jump size equal to its service time, then it starts crawling downward with slope $-1$ until it reaches $0$: when this happens, the server has processed all the customers in the queue and will have to wait for a new customer. The process $B$ is called the \textbf{cumulative busy time}: it is a piecewise linear process, whose slope only alternates between $0$ and $1$. Indeed, since $C$ and $L$ have jumps of the same size, they cancel each other. Hence, the process $B$ grows linearly with slope equal to $1$ when $L$ is positive, i.e. when the system is working, independently of the fact that new jumps occur. If instead $L(s)=0$ for $s \in [t_1,t_2]$, then $B$ will be constant in $[t_1,t_2]$. In practice, $B$ is a clock that runs only when the queue is working. We can then evaluate the length of the queue at time $t$ as
\begin{equation*}
	Q(t)=N^{\sf a}(t)-N^{\sf d}(B(t))=N^{\sf a}(t)-D(t),
\end{equation*}
i.e. by counting how many customers entered the queue up to time $t$ and then subtracting the ones whose service took place during the time $B(t)$ in which the server was working in $[0,t]$. The process $Q$ is called the \textbf{queue length process}. Notice that the number of departures is described by the process $D(t)=N^{\sf d}(B(t))$, that is then called the \textbf{departure process}, while its discontinuity points $(D_j)_{j \ge 0}$ (with $D_0=0$) are called the \textbf{departure times}. It is worth noticing that the service times $(S_k)_{k \ge 1}$ do not coincide with the increments of the departure times, since one could find an idle period between two departures. Nevertheless, the following formula holds true:
\begin{equation}\label{eq:servicefromAD}
	S_n=D_n-\max\{D_{n-1},A_n\}.
\end{equation}
The sequence of discontinuity points $(J_n)_{n \ge 0}$ (where $J_0=0$) of $Q$ are called \textbf{event times} and are either arrival or departure times. Their increments $\tau_n=J_n-J_{n-1}$ are called the \textbf{inter-event times}. {We can also consider the \textbf{event counting process}
\begin{equation*}
	G(t)=\max\{k \ge 0: \ J_k \le t\}.
\end{equation*}
It is not difficult to check that the following identity holds
\begin{equation*}
	G(t)=N^{\sf a}+D(t).
\end{equation*}}
If, for some reason, we know that $\bP(\exists k,j \ge 1, \ A_k=D_j)=0$, then the following relations hold for $n \ge 1$:
\begin{align}\label{eq:ADfromQ}
	\begin{split}
		A_n&=\min\{t>A_{n-1}, \ Q(t)=Q(t-)+1\}\\
		D_n&=\min\{t>D_{n-1}, \ Q(t)=Q(t-)-1\}.	
	\end{split}
\end{align}
Let us stress that it is possible that, due to the nature of the experiment, one can clearly observe the queue length process $Q$, while the sequences $(T_n)_{n \ge 1}$ and $(S_n)_{n \ge 1}$ should be determined from it. Furthermore, $\bP(\exists k,j \ge 1, \ A_k=D_j)=0$ can be assumed at the stage of model formulation after some empirical observation of the phenomena we want to describe. Hence, up to considering this property, one can prescribe the queue length process $Q$, then determine arrival and departure times $(A_n)_{n \ge 0}$ and $(D_n)_{n \ge 0}$ by \eqref{eq:ADfromQ}, inter-arrival times as $T_n=A_n-A_{n-1}$ for $n \ge 1$ and service times from \eqref{eq:servicefromAD}. Finally, let us observe that, by definition of reflection and regulation map, we can rewrite
\begin{equation*}
	B(t)=t-\Psi(X)(t):=t-I(t)
\end{equation*}
{where $\Psi$ is the \textbf{Skorokhod regulator map} (see Appendix \ref{sec:SRM}), defined as \begin{equation*}
	\Psi(f)(t)=\sup_{0 \le s \le t}\max\{-f(s),0\}, \ t \ge 0, \ f \in \mathbb{D}.
\end{equation*}
The process} $I(t)=\Psi(X)(t)$ is called the \textbf{cumulative idle time} and describes how much time (up to time $t$) the system has been empty.

% From this, we have that the number of departures from the queue can be described as
%\begin{equation*}
%	D(t)=N^{\sf d}(B(t)),
%\end{equation*}
%where the latter is called the \textbf{departure process}. Discontinuities of the departure process, that we will denote as $(D_k)_{k \ge 0}$, are called the \textbf{departure times}. Finally, the length of the queue at time $t$ can be written by counting the arrivals and then removing the number of departures, i.e.
%\begin{equation*}
%	Q(t)=N^{\sf a}(t)-D(t).
%\end{equation*}
%The process $Q$ is called the \textbf{queue length process}.
%{One can also define the \textbf{cumulative idle time} as 
	%\begin{equation*}
	%	I(t)=\Psi(X(t))
	%\end{equation*}
	%so that we can rewrite the cumulative busy time as
	%\begin{equation*}%
	%	B(t)=t-I(t).
	%\end{equation*}}
In general, a queueing system is said to be a $GI/GI/1$ when $(T_k)_{k \ge 1}$ and $(S_k)_{k \ge 1}$ are i.i.d. random variables and the sequences are independent of each other. If we remove the independence assumption, we get instead a $G/G/1$ queueing system. {For simplicity, throughout the paper we will always assume that the queueing system is empty at the start, i.e. $Q(0)=0$.}
\subsection{The $M/M/1$ queue}\label{subs:MM1}
{Before proceeding with the description of the Mittag-Leffler queueing models, let us first introduce the classical $M/M/1$ queueing systems, that will be our baseline model. Indeed, on the one hand, the Mittag-Leffler queues considered here can be seen as suitable modifications of the $M/M/1$ queue, that will generally appear as a special case; on the other hand, we will compare all the scaling limit results on our Mittag-Leffler queueing systems with the analogous ones for the $M/M/1$ queueing system.}

{It is worth noticing that the $M/M/1$ queueing system can be introduced in three equivalent ways, that will be now discussed separately.}

\subsubsection{The $M/M/1$ as a single server queueing system}
{In the context of single server queueing system, the $M/M/1$ is characterized by means of the sequence of its inter-arrival and service times. In particular, the inter-arrival times $(T_k)_{k \ge 1}$ and the service times $(S_k)_{k \ge 1}$ are i.i.d. exponential random variables with possibly different rates $\lambda,\mu>0$ (that are called then inter-arrival and service rates) and that are independent of each other. Notice that in this case the queue length process is actually a Markov process.}
\subsubsection{The $M/M/1$ as a Markov-renewal process}\label{sec:MRP}
{An alternative, but equivalent, construction of the $M/M/1$ is as follows. Let $X^{(1)}$ and $X^{(2)}$ be two exponential random variables with rates respectively $\lambda,\mu>0$ and let
\begin{equation}\label{mindist}	
p_{\lambda,\mu}=\bP(X^{(1)}<X^{(2)})=\frac{\lambda}{\lambda+\mu}.
\end{equation}
Consider then the Markov chain $q=(q_k)_{k \ge 0}$ on the non-negative integers with $q_0=0$ and transition matrix described, for $i,j \ge 0$, by
\begin{equation}\label{eq:mCren1}
	\bP(q_k=j \mid q_{k-1}=i)=\begin{cases}
		1 & j=1, \ i=0 \\
		p_{\lambda,\mu} & j=i+1 \not = 1 \\
		1-p_{\lambda,\mu} & j=i-1, \ i>0\\
		0 & |j-i|>1 \quad \mbox{ or } j=i \\
	\end{cases}
\end{equation}
Then, we define a sequence of random variables $Y=(Y_k)_{k \ge 1}$ with the property that $Y_k$ is conditionally independent of $(q_h)_{h<k-1}$ given $q_{k-1}$ and such that
\begin{equation}\label{eq:Y1}
	\bP(Y_k>t\mid q_{k-1})=\begin{cases} e^{-\lambda t} & q_{k-1}=0 \\
		e^{-(\lambda+\mu)t} & q_{k-1}>0.
	\end{cases}
\end{equation}
In particular, $(q,Y)_{k \ge 0}$ is a Markov renewal process. Then, if we consider the counting process
\begin{equation}\label{eq:countingproc}
	\widetilde{N}(t):=\max\left\{k \ge 0: \ \sum_{j=1}^{k}Y_k \le t\right\},
\end{equation}
the process $Q(t)=q_{\widetilde{N}(t)}$ is exactly the queue length process of a $M/M/1$ queueing system.}
\subsubsection{The $M/M/1$ as the reflection of a continuous-time random walk}\label{sec:MM1ref}
{A third equivalent construction of the $M/M/1$ is provided as follows. Consider a sequence $(\chi_k)_{k \ge 1}$ of i.i.d. random variables such that
\begin{equation}\label{eq:MCc}
	\bP(\chi_k=j)=\begin{cases} p_{\lambda,\mu} & j=1 \\
		1-p_{\lambda,\mu} & j=-1 \\
		0 & j \not = \pm 1
	\end{cases}
\end{equation}
and the random walk $r=(r_k)_{k \ge 0}$ with
\begin{equation}\label{eq:dtrw}
	r_0=0, \qquad r_k=\sum_{j=1}^{k}\chi_j, \ k \ge 1.
\end{equation}
Let $N$ be an independent Poisson process of rate $\lambda+\mu$ and define the continuous-time random walk
\begin{equation}\label{eq:ctrw}
R(t)=r_{N(t)}, \ t \ge 0.
\end{equation}
Then the process $Q(t)=\Phi(R)(t)$ is still the queue-length process of a $M/M/1$ queue. Since the reflection commutates with the time-change, by Lemma \ref{lem:scaling2}, then one can rewrite the previous process as $Q(t)=\overline{q}_{N(t)}$, where the Markov chain $\overline{q}=(\overline{q}_k)_{k \ge 0}$ over the non-negative integers, defined by means of $\overline{q}_0=0$ and
\begin{equation}\label{eq:Mchain2}
	\bP(\overline{q}_k=j \mid \overline{q}_{k-1}=i)=\begin{cases}
		p_{\lambda,\mu} & j=i+1 \\
		1-p_{\lambda,\mu} & j=i-1 \ge 0 \mbox{ or }j=i=0 \\
		0 & |j-i|>1 \quad \mbox{ or } j=i \not = 0,
	\end{cases}
\end{equation}
is the reflection of the Markov chain $r$.}

\subsubsection{The $M/M/1$ as a reflected difference of competing Poisson processes}\label{sec:comp}
{Finally, let us present a third equivalent way of constructing a $M/M/1$ queue. Let $N^{\sf a}$ and $N^{\sf d}$ be two independent Poisson processes with rates respectively $\lambda,\mu>0$. Then the process
\begin{equation*}
Q=\Phi(N^{\sf a}-N^{\sf d})
\end{equation*}
is again the queue length process of a $M/M/1$ queueing system.}

{We will see that the equivalence of these four constructions strongly rely on the underlying exponential distribution. Indeed, these will correspond respectively with our fifth, second, third and fourth model, that we will observe being extremely different in terms as soon as we drop the exponential time assumption. Once this is clear, we are ready to introduce the main objects of study of this work.}
%A special case of queueing system is given by the $M/M/1$ queue, in which the sequences $(T_k)_{k \ge 1}$ and $(S_k)_{k \ge 1}$ are i.i.d. exponential random variables with rate respectively $\lambda,\mu>0$ and are independent of each other. In this case we have a Markov single server queue length process. {} 

\subsection{Model 1: the fractional M/M/1 queue}\label{subsmod1}
A firsts model of Mittag-Leffler-type queue has been introduced in \cite{cahoy2015transient} {by starting from a $M/M/1$ queueing system and then introducing a random clock in the system}. This model is constructed as follows. {Let $Q$ be the queue length process of a classical $M/M/1$ queue with inter-arrival and service rates $\lambda,\mu>0$. For $\alpha \in (0,1]$, consider an independent $\alpha$-stable subordinator $\sigma_\alpha$ and let $L_\alpha$ be its first-passage time process, i.e.,
\begin{equation*}
L_\alpha(t):=\min\{s \ge 0: \ \sigma_\alpha(s) \ge t\},
\end{equation*}
that we call \textit{inverse $\alpha$-stable subordinator} (see Appendix \ref{sec:stab} for further results).}
%consider the queue length process $Q$ of a classical $M/M/1$ queue with inter-arrival and service rates $\lambda,\mu>0$ and an independent $\alpha$-stable subordinator $L_\alpha$, for some $\alpha \in (0,1]$. 
The \textbf{fractional $M/M/1$ queue length process}, of order $\alpha \in (0,1]$ and {generalized} inter-arrival and service rates $\lambda,\mu>0$, is defined as ${}_1Q:=Q \circ L_\alpha$. The term \textit{fractional} in the name follows from the fact that the transient state probabilities
\begin{equation*}
	p_k(t):=\bP({}_1Q(t)=k)
\end{equation*}
satisfy the fractional Cauchy problem
\begin{equation*}
	\begin{cases}
		\displaystyle \dersup{}{t}{\alpha}{}_1p_0(t)=-\lambda p_0(t)+\mu p_1(t)\\[7pt]
		\displaystyle \dersup{}{t}{\alpha}p_k(t)=-(\lambda+\mu) p_k(t)+\lambda p_{k-1}(t)+\mu p_{k+1}(t)  & k=1,2,\cdots \\[7pt]
		p_0(0)=\delta_{k,0} & k=0,1,2,\cdots 
	\end{cases}
\end{equation*}
that, for $\alpha=1$, reduces to the forward Kolmogorov equation for the distribution of the queue length of the classical $M/M/1$ queue.

Adopting the notation of Section \ref{subs:ssqs}, the distribution of the inter-arrival times $T_k$ and of the serve times $S_k$ has been obtained in \cite[Proposition $7$]{ascione2022skorokhod}. {To state such a result, let us recall the following distribution: a non-negative random variable $X$ is called a \textbf{Mittag-Leffler random variable} of order $\alpha \in (0,1]$ and generalized rate $\lambda>0$ if for $t \ge 0$
\begin{equation*}
	\bP(X>t)=E_\alpha(-\lambda t^\alpha), \quad \mbox{ where }E_\alpha(x)=\sum_{k=0}^{+\infty}\frac{x^k}{\Gamma(\alpha k+1)},
\end{equation*}
where the function $E_\alpha$ is called the \textbf{Mittag-Leffler function}. We denote such a distribution as $X \sim {\sf ML}_\alpha(\lambda)$. For further properties of Mittag-Leffler distribution we refer to Appendix \ref{sec:ML}.}
\begin{prop}
	For a fractional $M/M/1$ queue of order $\alpha \in (0,1]$ and inter-arrival and service rates $\lambda,\mu>0$, the sequences $(T_k)_{k \ge 1}$ and $(S_k)_{k \ge 1}$ are constituted by i.i.d. random variables and it holds
	\begin{equation*}
		T_k \sim {\sf ML}_\alpha(\lambda) \qquad S_k \sim {\sf ML}_\alpha(\mu).
	\end{equation*}
	The two sequences, however, are not independent of each other unless $\alpha=1$.
\end{prop}
Let us stress that the latter statement is due to the fact that, as shown in \cite{cahoy2015transient}, the distribution of the inter-event times $(\tau_{k})_{k \ge 0}$ satisfies, for $k \ge 1$,
\begin{equation}\label{eq:interevent1}
	\bP(\tau_k>t \mid {}_1Q(J_{k-1}))=\begin{cases} E_\alpha(-\lambda t^\alpha) & {}_1Q(J_{k-1})=0 \\
		E_\alpha(-(\lambda+\mu) t^\alpha) & {}_1Q(J_{k-1})>0,
	\end{cases}
\end{equation}
together with the fact that the queue length process ${}_1Q$ is a semi-Markov process (see \cite[Example 2.13]{kaspi1988regenerative}) and that the Mittag-Leffler function does not satisfy a semi-group property unless $\alpha=1$ (see \cite{peng2010note}). For $\alpha=1$, the model reduces to the standard $M/M/1$ queue. 
%. Then, for this reason, we can assume, without loss of generality, that $k=1$ and that we are conditioning on the initial state of the queue. However it is worth noticing that, on the event $\{{}_1Q(0)>0\}$ it holds ${}_1\tau_1=\min\{{}_1S_1,{}_1T_1\}$. If ${}_1S_1$ and ${}_1T_1$ were independent, then, for $k \ge 1$, it should hold (see \cite{peng2010note})
%\begin{multline*}
%\bP({}_1\tau_1>t \mid {}_1Q(0)=k)=\bP(\min\{{}_1S_1,{}_1T_1\}>t)\\
%=E_\alpha(-\lambda t^\alpha)E_\alpha(-\mu t^\alpha) \not = E_\alpha(-(\lambda+\mu)t^\alpha),
%\end{multline*}
%unless $\alpha=1$, that is in contradiction with \eqref{eq:interevent1} whenever $\alpha<1$. 
This process has been extensively studied in \cite{cahoy2015transient,ascione2018fractional,ascione2022skorokhod,butt2023queuing}, in which almost all the characteristics have been determined, and a heavy traffic limit has been achieved in \cite{ascione2022skorokhod}.
\subsection{Model 2: the fast renewal Mittag-Leffler queue}\label{Sec:model2}
In the previous model, the lack of independency of inter-arrival and service times has been verified by checking that the inter-event times are not distributed like the minimum of two independent Mittag-Leffler random variables. Hence, let us consider an alternative model  constructed as follows: consider two i.i.d. sequences $(X^{(1)}_k)_{k \ge 1}$ and $(X^{(2)}_k)_{k \ge 2}$ independent of each other and such that $X^{(1)}_1 \sim {\sf ML}_\alpha(\lambda)$ and $X^{(2)}_1 \sim {\sf ML}_\beta(\mu)$ for some $\lambda,\mu>0$ and $\alpha,\beta \in (0,1]$. We define then the sequence of times $Y=(Y_k)_{k \ge 1}$ where $Y_k=\min\{X^{(1)}_k,X^{(2)}_k\}$. We want to use $(Y_k)_{k \ge 1}$ (or a subsequence of it) as inter-event time for our queue length process.

In \cite{butt2023queuing}, the following model has been proposed, {partially based on the renewal process construction of the $M/M/1$, as in Section \ref{sec:MRP}}. Consider two independent random variables $X^{(1)} \sim {\sf ML}_\alpha(\lambda)$ and $X^{(2)} \sim {\sf ML}_\beta(\lambda)$ {and we let $p_{\lambda,\mu}^{\alpha,\beta}$ as in \eqref{mindist}, where we added the supscripts $\alpha$ and $\beta$ are used to underline the dependence on the orders of the considered Mittag-Leffler random variables. Then we consider the Markov chain $q$ as in \eqref{eq:mCren1} (with $p_{\lambda,\mu}^{\alpha,\beta}$ in place of just $p_{\lambda,\mu}$) and we assume that it is independent of the sequence $Y$ defined before. Once we also consider the counting process $\widetilde{N}(t)$ as in \eqref{eq:countingproc}, we can introduce the desired queueing system by means of the queue length process
\begin{equation*}
	{}_2Q(t)=q_{\widetilde{N}(t)}.
\end{equation*}}
%\begin{equation*}
%	p_{\lambda,\mu}^{\alpha,\beta}=\bP(X^{(1)}<X^{(2)}).
%\end{equation*}
%Then we define the Markov chain $(q_k)_{k \ge 0}$ by setting $q_0=0$ and, for $j,i \ge 0$,
%\begin{equation}\label{eq:mCren1}
%	p_{i,j}=\bP(q_k=j \mid q_{k-1}=i)=\begin{cases}
%		1 & j=1, \ i=0 \\
%		p_{\lambda,\mu}^{\alpha,\beta} & j=i+1 \not = 1 \\
%		0 & |j-i|>1 \quad \mbox{ or } j=i \\
%		1-p_{\lambda,\mu}^{\alpha,\beta} & j=i-1.
%	\end{cases}
%\end{equation}
%We also assume that the Markov chain $(q_k)_{k \ge 0}$ is independent of the sequence $(Y_k)_{k \ge 1}$. Let $\widetilde{N}(t)$ be the counting process of the sequence $(Y_k)_{k \ge 0}$, i.e.
%\begin{equation*}
%	\widetilde{N}(t):=\max\left\{k \ge 0: \ \sum_{j=1}^{k}Y_k \le t\right\}.
%\end{equation*}
%Then the queue length process in this case is defined as
%\begin{equation*}
%	{}_2Q(t)=q_{\widetilde{N}(t)}.
%\end{equation*}
Now we want to determine the distribution of inter-arrival and service times of this queueing model. To do this, let us first introduce some notation.
%a new distribution for non-negative random variables. 
\begin{defn}
	We say that a random variable $X \sim {\sf mML}_{\alpha,\beta}(\lambda,\mu)$
	% is distributed as a \textbf{minimum Mittag-Leffler} of orders 
	with orders $\alpha,\beta \in (0,1]$ and rates $\lambda,\mu>0$ if for $t>0$
	\begin{equation*}
		\bP(X>t)=E_\alpha(-\lambda t^\alpha)E_\beta(-\mu t^\beta)=:1-F_{\alpha,\beta}^{\lambda,\mu}(t),
	\end{equation*}
	i.e. if $X\overset{d}{=}\min\{X_1,X_2\}$, where $X_1={\sf ML}_\alpha(\lambda)$ and $X_2={\sf ML}_\beta(\mu)$.
	%and we denote this as $X \sim {\sf mML}_{\alpha,\beta}(\lambda,\mu)$.
	We say that a random variable $X \sim {\sf mGE}_{\alpha,\beta}(\lambda,\mu;k)$ with orders $\alpha,\beta \in (0,1]$, rates $\lambda,\mu>0$ and shape parameter $k$ if
	\begin{equation*}
		\bP(X \le t)=(F_{\alpha,\beta}^{\lambda,\mu})^{\ast k}(t),
	\end{equation*}
	where
	\begin{equation*}
		(F_{\alpha,\beta}^{\lambda,\mu})^{\ast 1}(t)=F_{\alpha,\beta}^{\lambda,\mu}(t) \qquad (F_{\alpha,\beta}^{\lambda,\mu})^{\ast k}(t)=\int_0^t(F_{\alpha,\beta}^{\lambda,\mu})^{\ast (k-1)}(t-s)dF_{\alpha,\beta}^{\lambda,\mu}(s),
	\end{equation*}
	i.e., if it is distributed as the sum of $k$ i.i.d. random variables $(X_j)_{j=1,\dots,k}$ such that $X_1 \sim {\sf mML}_{\alpha,\beta}(\lambda,\mu)$.
\end{defn}
\noindent Notice that, by definition, $Y_k \sim {\sf mML}_{\alpha,\beta}(\lambda,\mu)$, while $\sum_{j=k_1}^{k_1+\delta-1}Y_j \sim  {\sf mGE}_{\alpha,\beta}(\lambda,\mu;\delta)$. Furthermore, recall that if $\alpha=\beta=1$, then ${\sf ML}_{1,1}(\lambda,\mu)={\sf Exp}(\lambda+\mu)$.
Now we can determine the distribution of the inter-arrival times, conditionally on the number of customers in the queue. Again, we use the notation of Section \ref{subs:ssqs}.
\begin{thm}\label{thm:interarrival2}
	Let $(T_k)_{k \ge 1}$ be the sequence of inter-arrival times of the renewal Mittag-Leffler queue ${}_2Q$. Then $T_1 \sim {\sf mML}_{\alpha,\beta}(\lambda,\mu)$. Furthermore, for any $k \ge 2$ and $n=1,\dots,k-1$ it holds
	\begin{equation*}
		\bP(T_k>t \mid {}_2Q(A_{k-1})=n)
		=\sum_{j=0}^{n}(1-(F_{\alpha,\beta}^{\lambda,\mu})^{\ast (j+1)}(t))(1-p_{\alpha,\beta}^{\lambda,\mu})^jp_{\alpha,\beta}^{\lambda,\mu}.
	\end{equation*}
\end{thm}
Concerning the service times, we have a similar result.
\begin{thm}\label{thm:service2}
	Let $(S_k)_{k \ge 1}$ be the sequence of service times of the renewal Mittag-Leffler queue ${}_2Q$. Then,
	\begin{equation*}
		\bP(S_k>t)
		=\sum_{j=0}^{+\infty}(1-(F_{\alpha,\beta}^{\lambda \mu})^{\ast (j+1)}(t))(p_{\alpha,\beta}^{\lambda,\mu})^j(1-p_{\alpha,\beta}^{\lambda,\mu}).
	\end{equation*}
\end{thm}
{The proofs of the latter results are given in Appendix \ref{sec:ptinter12} and \ref{sec:ptser12}.} 
Other properties concerning the existence of a limit distribution have been considered in \cite{butt2023queuing}. It is worth noticing, however, that this model does not really generalize a classical $M/M/1$ model. Indeed, according to Theorem \ref{thm:interarrival2}, if $\alpha=\beta=1$, then $T_1 \sim{\sf Exp}(\lambda+\mu)$, while we should have $T_1 \sim {\sf Exp}(\lambda)$. Actually, for $\alpha=\beta=1$, we have the following corollary.
\begin{cor}\label{cor:toMM1}
	Let $\alpha=\beta=1$ and consider the sequence of inter-arrival and service times $(T_k)_{k \ge 1}$ and $(S_k)_{k \ge 1}$ of the renewal queue ${}_2Q$. Then it holds
	\begin{enumerate}
		\item $T_1 \sim {\sf Exp}(\lambda+\mu)$
		\item For $n=1,\cdots,k-1$ it holds
		\begin{multline*}
			\bP(T_k>t \mid {}_2Q(A_{k-1})=n)\\=e^{-(\lambda+\mu)t}\sum_{j=0}^{n}\frac{(\lambda+\mu)^jt^j}{j!}\left(\left(\frac{\mu}{\lambda+\mu}\right)^{j}-\left(\frac{\mu}{\lambda+\mu}\right)^{n+1}\right)
		\end{multline*}
		\item $S_k \sim {\sf Exp}(\mu)$ for all $k \in \N$.
	\end{enumerate}
\end{cor}
The proof is omitted as it follows directly from Theorems \ref{thm:interarrival2} and \ref{thm:service2}. The previous corollary tells us that if $\alpha=\beta=1$, ${}_2Q$ does not coincide with a classical $M/M/1$ queue due to a clear discrepancy with the distribution of the inter-arrival times, while we still preserve the distribution of the service times. Actually, we could notice this discrepancy even from the definition of the queueing model. Indeed, the process ${}_2Q$ is defined as the semi-Markov process associated to the Markov renewal process (see, for instance, \cite{pyke1961markov,gikhman2004theory}) $(q,Y)$ where the chain $(q_k)_{k \ge 0}$ is defined in \eqref{eq:mCren1} while $(Y_k)_{k \ge 1}$ are i.i.d. with $Y_k \sim {\sf mML}_{\alpha,\beta}(\lambda,\mu)$, hence leading to $Y_k \sim {\sf Exp}(\lambda+\mu)$ if $\alpha=\beta=1$. On the other hand, the classical $M/M/1$ queue, {the underlying Markov renewal process should satisfy \eqref{eq:Y1}, hence cannot coincide with $(q,Y)$ as introduced here.}
% is the (semi-)Markov process associated to the Markov renewal process $(q,Y^{(1)})$, where $(q_k)_{k \ge 0}$ is still the Markov chain defined in \eqref{eq:mCren1}, but $(Y^{(1)}_k)_{k \ge 1}$ is not independent of $(q_k)_{k \ge 0}$, since for $k \ge 1$
%\begin{equation}\label{eq:Y1}
%	\bP(Y^{(1)}_k>t \mid q_{k-1})=\begin{cases}
%		e^{-\lambda t} & q_{k-1}=0 \\
%		e^{-(\lambda+\mu)t} & q_{k-1}>0.
%	\end{cases}
%\end{equation}
%Hence, since the underlying Markov renewal processes do not coincide, it is clear that ${}_2Q$ for $\alpha=\beta=1$ is not a classical $M/M/1$ queue. 
{In practice, in case $\alpha=\beta=1$, the queue length process ${}_2Q$ coincides with the one of the $M/M/1$ until it reaches $0$. Its idle states are exponentially distributed of rate $\lambda+\mu$ in place of $\lambda$, i.e., they are usually shorter. Then after it leaves the idle state, it keeps behaving like a $M/M/1$ queue. For this reason, we will refer to this model as a \textbf{fast renewal Mittag-Leffler queue}. In order for ${}_2Q$ to be a proper generalization of the $M/M/1$ queue, one would need to prescribe separately the distribution of $Y_k$ when $q_{k-1}=0$.}

\subsection{Model 3: the renewal Mittag-Leffler queue}\label{Sec:model3}
{Instead of just prescribing an arbitrary distribution of $Y_k$ when $q_{k-1}=0$, in such a way that the case $\alpha=\beta=1$ leads trivially to the $M/M/1$ queue, let us introduce a further model based on the construction in Section \ref{sec:MM1ref}. To do this, we consider the random walk $r$ introduced in \eqref{eq:dtrw}, where we use $p_{\lambda,\mu}^{\alpha,\beta}$ in place of $p_{\lambda,\mu}$, and then we time-change it with the counting process $\widetilde{N}$ defined by \eqref{eq:countingproc}, where the sequence $Y=(Y_k)_{k \ge 1}$ is made of i.i.d. random variables such that $Y_k \sim {\sf mML}_{\alpha,\beta}(\lambda,\mu)$, as in the previous section. In this way, we obtain a continuous-time random walk $R(t)=r_{\widetilde{N}(t)}$ and we can consider its reflection
%To correct this behaviour, we consider the following modification of Model $2$ in \cite{butt2023queuing}. Let $(\chi_k)_{k \ge 1}$ be a sequence of i.i.d. random variables, independent of $(Y_j)_{k \ge 1}$, such that
%\begin{equation}\label{eq:MCc}
%	\bP(\chi_k=j)=\begin{cases}
%		p_{\lambda,\mu}^{\alpha,\beta} & j=1 \\
%		1-p_{\lambda,\mu}^{\alpha,\beta} & j=-1.
%	\end{cases}
%\end{equation}
%and consider the partial sum sequence $r=(r_k)_{k \ge 0}$, defined by $r_0=0$ and $r_k=r_{k-1}+\chi_k$ for $k \ge 1$.
%\begin{equation}\label{eq:MCr}
%	r_0=0 \qquad r_k=\sum_{j=1}^{k}\chi_j, \ k \ge 1.
%\end{equation}
%Then the bivariate process $(r_k,Y_k)_{k \ge 1}$ is a Markov renewal process with independent components and we can construct the semi-Markov process $R(t)=r_{\widetilde{N}(t)}$
%\begin{equation}\label{eq:Rt}
%	R(t):=r_{\widetilde{N}(t)}=\sum_{j=1}^{\widetilde{N}(t)}\chi_{j}.
%\end{equation}
%Clearly, $R(t)$ can also assume negative values. So, to construct a proper queueing model, we can consider the reflection of the process $R(t)$, i.e.
\begin{equation*}
	{}_3Q:=\Phi(R).
\end{equation*}}
Let $(T_k)_{k \ge 1}$ and $(S_k)_{k \ge 1}$ be the sequences of inter-arrival and service times of ${}_3Q$, according to the notation in Section \ref{subs:ssqs}. It is not difficult to check that the proof of Theorem \ref{thm:service2} also applies to this case, getting
\begin{equation}\label{eq:serv2Q}
	\bP(S_k>t)
	=\sum_{j=0}^{+\infty}(1-(F_{\alpha,\beta}^{\lambda, \mu})^{\ast (j+1)}(t))(p_{\alpha,\beta}^{\lambda,\mu})^j(1-p_{\alpha,\beta}^{\lambda,\mu}).
\end{equation}
Indeed, as in the previous case, ${S}_k$ is still the first downward jump of ${}_3{Q}$ after the $k$-th upward jump or the $(k-1)$-th downward jump (depending on which one happens after the other) and we only need to consider the possibility that any number of upward jumps happen in between. Differently, ${T}_k$ can be seen as the $k$-th upward jump of $R(t)$, since $R$ and ${}_3{Q}$ share the same upward jumps. However the process $R$ can admit any finite number of downward jumps between two upward jumps. As a consequence, we get the following result, whose proof is identical to the one of Theorem \ref{thm:service2} and thus is omitted.
\begin{thm}\label{thm:interarrival22}
	Let $({T}_k)_{k \ge 1}$ be the sequence of inter-arrival times of the renewal Mittag-Leffler queue ${}_3{Q}$. Then,
	\begin{equation}\label{eq:interarrival2}
		\bP({T}_k>t)
		=\sum_{j=0}^{+\infty}(1-(F_{\alpha,\beta}^{\lambda, \mu})^{\ast (j+1)}(t))(1-p_{\alpha,\beta}^{\lambda,\mu})^jp_{\alpha,\beta}^{\lambda,\mu}.
	\end{equation}
\end{thm}
Differently from ${}_2Q$, this renewal queue ${}_3{Q}$ is a generalization of the classical $M/M/1$ queue. {Indeed, the following result clearly holds, once we observe that if $\alpha=\beta=1$, then $p^{1,1}_{\lambda,\mu}=\frac{\lambda}{\lambda+\mu}$ and $\widetilde{N}$ is a Poisson process.}
\begin{cor}
	If $\alpha=\beta=1$, then ${}_3{Q}$ is a classical $M/M/1$ queue.
\end{cor}
Furthermore, we can deduce whether $p_{\lambda,\mu}^{\alpha,\beta}$ is greater or lower than $1/2$ by comparing the distributions of $T_1$ and $S_1$. Indeed, the following proposition holds true.
\begin{prop}\label{prop:comparison}
	It holds $p_{\lambda,\mu}^{\alpha,\beta} \ge \frac{1}{2}$ if and only if
	\begin{equation}\label{eq:comparsurv}
		\bP({T}_1>t)-(1-F_{\alpha,\beta}^{\lambda,\mu}(t))p_{\alpha,\beta}^{\lambda,\mu} \le \bP({S}_1>t)-(1-F_{\alpha,\beta}^{\lambda,\mu}(t))(1-p_{\alpha,\beta}^{\lambda,\mu}) 
	\end{equation}
	Furthermore, equality holds in \eqref{eq:comparsurv} if and only if $p_{\lambda,\mu}^{\alpha,\beta}=\frac{1}{2}$.
\end{prop}
{The proof is given in Appendix \ref{proofcomp}.}
{
Arguing as in Section \ref{sec:MM1ref}, one can show by Lemma \ref{lem:scaling2} that 
%Notice that the queue length process ${}_3{Q}$ can be equivalently defined, thanks to Lemma \ref{lem:scaling2}, as
		\begin{equation*}
			{}_3{Q}(t)=\overline{q}_{\widetilde{N}(t)},
		\end{equation*}
		where $\overline{q}=(\overline{q}_k)_{k \ge 0}$ is the Markov chain independent of $\widetilde{N}$ defined in \eqref{eq:Mchain2}, with $p_{\lambda,\mu}^{\alpha,\beta}$ in place of $p_{\lambda,\mu}$.
}

% Before showing this, let us notice that the same arguments as in the proof of Item $(3)$ in Corollary \ref{cor:toMM1} lead to the following result.
%\begin{prop}
%	If $\alpha=\beta=1$, then ${}_3{T}_k \sim {\sf Exp}(\lambda)$ and ${}_3{S}_k \sim {\sf Exp}(\mu)$ for all $k \ge 1$.
%\end{prop}
%To prove that for $\alpha=\beta=1$ our model ${}_3{Q}$ coincide with a classical $M/M/1$ queue, we will first provide a representation of ${}_3{Q}$ in terms of a suitable Markov renewal process.
{Before proceeding, let us observe that we can rewrite ${}_3Q$ as a semi-Markov process with jump chain $q$ by means of a suitable modification of the distribution of $Y_k$ whenever $q_{k-1}=0$.}
\begin{prop}\label{prop:27}
	Consider the Markov chain $(q_k)_{k \ge 0}$ defined in \eqref{eq:mCren1} with $p_{\lambda,\mu}^{\alpha,\beta}$ in place of $p_{\lambda,\mu}$ and let $(\widetilde{Y}_k)_{k \ge 1}$ be a sequence of non-negative random variables such that
	\begin{equation}\label{eq:tildeY}
		\bP(\widetilde{Y}_k>t \mid q_{k-1})=\begin{cases}
			1-F_{\alpha,\beta}^{\lambda,\mu}(t) & q_{k-1}>0 \\
			\sum_{j=0}^{+\infty}(1-(F_{\alpha,\beta}^{\lambda,\mu})^{\ast(j+1)})(1-p_{\alpha,\beta}^{\lambda,\mu})^jp_{\alpha,\beta}^{\lambda,\mu} & q_{k-1}=0.
		\end{cases}
	\end{equation}
	Let also $\overline{N}(t)$ be the counting process, defined as in \eqref{eq:countingproc} with $(\widetilde{Y}_k)_{k \ge 0}$ in place of $(Y_k)_{k \ge 0}$. Then 
	\begin{equation*}
		{}_3Q(t)=q_{\overline{N}(t)}, \qquad t \ge 0.
	\end{equation*}
%	${}_3{Q}$ is the semi-Markov process induced by the Markov renewal process $(q,\widetilde{Y})$, i.e., up to equality in distribution,
%	\begin{equation*}
%		{}_3{Q}(t)=q_k \quad \mbox{ for }\quad  t \in \left[\sum_{j=1}^{k}\widetilde{Y}_j,\sum_{j=1}^{k+1}\widetilde{Y}_j\right).
%	\end{equation*}
\end{prop}
{The proof is given in Appendix \ref{proofprop27}. In general, we refer to this model as the \textbf{renewal Mittag-Leffler queue}.}
%Notice that if $\alpha=\beta=1$, then, applying the same proof of Item $(3)$ in Corollary \ref{cor:toMM1}, \eqref{eq:tildeY} coincides with \eqref{eq:Y1}. Hence, we get the following straightforward result.

\subsection{Model 4: the \textit{restless} Mittag-Leffler queue}
Let us now introduce a further model, according to Model 3 in \cite{butt2023queuing}. {Before proceeding let us recall that a fractional Poisson process $N$ of order $\alpha \in (0,1)$ and rate $\lambda>0$ is defined as the counting process given in \eqref{eq:countingproc} where $Y=(Y_k)_{k \ge 1}$ is a sequence of i.i.d. random variables such that $Y_k \sim {\sf ML}_\alpha(\lambda)$. Then we consider a generalization of the $M/M/1$ queue based on the approach presented in Section \ref{sec:comp}. Precisely, we consider $\alpha,\beta \in (0,1]$ and $\lambda,\mu>0$ we define the queue length process as
\begin{equation*}
	{}_4Q=\Phi(N^{\sf a}-N^{\sf d}),
\end{equation*}
where $N^{\sf a}$ and $N^{\sf d}$ are two independent fractional Poisson processes of order $\alpha,\beta$ and rate $\lambda,\mu$ respectively. To better understand how this process works, notice that the following dynamics hold: the queue length process ${}_4Q$ increases by $1$ at each jump of $N^{\sf a}$; if a jump of $N^{\sf d}$ occurs when ${}_4Q$ is not $0$, then the latter decreases by $1$, otherwise it is kept at $0$.}

Here, we will refer to this model as the \textbf{\textit{restless} Mittag-Leffler queue}: the motivation will be clear in the following. Before stating the next result, let us recall that we are again using the notation given in Section \ref{subs:ssqs}.
% $(A_n)_{n \ge 0}$ and $(D_n)_{n \ge 0}$ are respectively the sequences of arrival and departure times.
\begin{prop}\label{prop:distserv3}
	Let ${}_4Q$ be a restless Mittag-Leffler queue. Then
	\begin{equation}\label{eq:service1}
		\bP(S_{n}>t \mid A_n , \ D_{n-1})=\begin{cases} E_\beta(-\mu t^\beta) & A_n<D_{n-1} \\
			1-F_t(A_n-D_{n-1}) & A_n \ge D_{n-1}
		\end{cases}
	\end{equation}
	where
	\begin{align*}
		F_t(T)=\int_{[0,T]}\frac{E_\beta(-\mu(T-s+t)^\beta)}{E_\beta(-\mu(T-s)^\beta)}dF_{\sf LJ}(s;\beta,\mu,T)
	\end{align*}
	and $F_{\sf LJ}(s;\beta,\mu,T)$ is defined, for $s \le T$, as
	\begin{align}\label{eq:lastjumptime}
		F_{\sf LJ}(s;\beta,\mu,T)=E_\beta(-\mu T^\alpha)+\sum_{n=1}^{+\infty}\beta n \mu^n \int_0^s E_\beta(-\mu (T-z)^\beta)z^{n\beta-1}E_{\beta,n\beta+1}^{n+1}(-\mu z^\beta)\, dz,
	\end{align}
	while $F_{\sf LJ}(s;\beta,\mu,T)=1$ for $s \ge T$.
	%in \eqref{eq:lastjumptime}.
	%	\begin{equation*}
		%		{\sf LJ}_\beta(s;T)=E_\beta(-\mu T^\beta)+\sum_{n=1}^{+\infty}\beta n \mu^n \int_0^{T+t-s}E_\beta(-\mu(T-z)^\beta)z^{n\beta-1}E_{\beta,n\beta+1}^{n+1}(-\mu z^\beta)\, dz.
		%	\end{equation*}	
\end{prop}
{The proof is given in Appendix \ref{sec:restlessproof1}}
In practice, if a customer enters the queue after it became empty for the first time, its service time will be not Mittag-Leffler distributed, but its distribution takes in consideration how much time passed from the arrival of such customer and the last departure: in a certain sense, the server continues working while waiting for a new customer and each \textit{finished service} while the queue is empty is \textit{thrown away}. For the inter-arrival times, since ${}_4Q$ has the same upward jumps as $N_{\sf a}$, hence we get immediately that
\begin{equation*}
	T_k \sim {\sf ML}_\alpha(\lambda).
\end{equation*}
Concerning the case $\alpha=\beta=1$, it is not difficult to check that $T_n \sim {\sf Exp}(\lambda)$ and $S_n \sim {\sf Exp}(\mu)$ independently of the values of $A_n$ and $D_{n-1}$. Furthermore, in such a case, we get that the sequences of inter-arrival and service times are independent. Indeed, it is clear by definition that $S_n$ is independent of $T_k$ for $k > n$. Furthermore, with the very same proof as in Proposition \ref{prop:distserv3}, we can substitute $A_n$ in \eqref{eq:service1} with the sequence $(T_j)_{j \le n}$, without changing the right-hand side of the equality. Hence, in case $\beta=1$, $\bP(S_n>t \mid (T_j)_{j \le n}, D_{n-1})=e^{-\mu t}$, where the right-hand side is deterministic. With this in mind, we have, by the tower property of the conditional expectation, $\bP(S_n>t \mid T_j)=e^{-\mu t}$ for any $j=1,\dots,n$, i.e. $S_n$ is independent of $(T_j)_{j=1,\dots,n}$. Hence, for $\alpha=\beta=1$, the model ${}_4Q$ reduces to a classical $M/M/1$ queue, as observed in \cite{butt2023queuing}.
\subsection{Model 5: the Mittag-Leffler GI/GI/1 queue}\label{sec:Mod5}
In the previous model, we were able to preserve the Mittag-Leffler distribution of the inter-arrival times and of the service times for a non-empty queue, but the structure of the queue length process ${}_4Q$, thanks to \eqref{eq:service1}, tells us that the service times are in general not Mittag-Leffler distributed and they depend on the value of the last departure time and the arrival time of the considered customer.

Let us now construct a proper Mittag-Leffler GI/GI/1 queueing model starting from the actual distribution of the inter-arrival and service times. Fix $\alpha,\beta \in (0,1]$ and $\lambda,\mu>0$. We want to construct a queueing system with inter-arrival and service times, respectively $(T_k)_{k \ge 1}$ and $(S_k)_{k \ge 1}$, that satisfy the following properties:
\begin{itemize}
	\item $(T_k)_{k \ge 1}$ are i.i.d. with $T_k \sim {\sf ML}_\alpha(\lambda)$;
	\item $(S_k)_{k \ge 1}$ are i.i.d. with $T_k \sim {\sf ML}_\beta(\mu)$;
	\item The two sequences are independent of each other. 
\end{itemize}
We now proceed as in Section \ref{subs:ssqs} to construct the queue length process. First, consider the arrival and cumulative service times, respectively $(A_k)_{k \ge 0}$ and $(CS_k)_{k \ge 0}$ such that, for $k \ge 1$,
\begin{equation*}
	A_0=CS_0=0 \qquad A_k=A_{k-1}+T_k \qquad CS_k=CS_{k-1}+S_k.
\end{equation*}
Then we define the arrival and cumulative services counting processes as
\begin{equation*}
	N^{\sf a}(t)=\max\{k \ge 0: \ A_k \le t\} \qquad N^{\sf d}(t)=\max\{k \ge 0: \ CS_k \le t\}.
\end{equation*}
With the cumulative service times and the arrival counting process we construct the cumulative-input process, the net-input process and then the workload
\begin{equation*}
	C(t)=CS_{N^{\sf a}(t)} \qquad X(t)={}_5C(t)-t \qquad L(t)=\Phi(X(t)).
\end{equation*}
From the latter we can define the cumulative busy time and then we combine it with the cumulative services counting process to get the departure process:
\begin{equation*}
	B(t)=C(t)-L(t) \qquad D(t)=N^{\sf d}(B(t)).
\end{equation*}
Once we have all the pieces, we can finally define the queue length process
\begin{equation*}
	{}_5Q(t)=N^{\sf a}(t)-D(t).
\end{equation*}
It is worth noticing that this model and the restless Mittag-Leffler queue share the same inter-arrival times distribution and the exact same behaviour as soon as the queue is non-empty. When the queue is empty, the cumulative busy time stops and then $D(t)$ becomes constant: as a consequence, differently from the previous model, this time the server actually stops working and restarts only if a customer enters the queue. As a consequence, we get that the service time is always Mittag-Leffler distributed and that inter-arrival and service times are actually independent, as assumed from the beginning of the construction.

\section{Scaling limits of Mittag-Leffler queues: Main Results}\label{sec:scaling}
Now that we have introduced all the models and we described the distribution of their inter-arrival and service times, we want to focus on the scaling limits of these queues. Let us first underline a quite important difference between the framework of Mittag-Leffler queues and the classical $M/M/1$ queue. We notice that in Models $1$, $4$ and $5$ the inter-arrival and service times are heavy-tailed. The same applies to the inter-event times in both ${}_2Q$ and ${}_3{Q}$, provided that $\alpha+\beta\le 1$. This clearly implies that the parameters $\lambda$ and $\mu$ cannot play the role of \textit{average frequencies of arrivals and services}, since inter-arrival and service times have infinite expected value. As a consequence, it is clear that one needs to change approach with respect to the classical case. 

{In this section, for the ease of the reader, we will only state the main scaling limit results, while leaving the proofs in the next section. Concerning the structure of the results, we will always recognize three different scaling limits. Let us anticipate that we will use this as a criterion to distinguish among different \textit{regimes} of the queues. Furthermore, we will notice that the difference between the scaling limits will always be dictated by a suitable parameter, that is different for each model. At the end of the section, we will give a table with the summary of the characteristic of the queue. For readability, we will refer to the aforementioned parameter as \textbf{Regime Discrimination Parameter} (RDP). Concerning the way this discrimination works, in practice we will discriminate among the three scaling limits according to the fact that the RDP is greater, equal or smaller than a value, that we call the \textbf{Critical Value} (CV).}

For this reason, let us first recall the scaling limits of the classical $M/M/1$ queue. Furthermore, we will make frequently use of the continuous mapping theorem (see \cite[Theorem 3.4.4]{whitt2002stochastic}). During the proof, for the ease of the reader, we will only specify the theorem that guarantees that the function we are using is actually continuous in the limit process, while implying the use of the continuous mapping theorem. 

\subsection{The scaling limits of the classical $M/M/1$ queue: a summary}\label{sec:MM1}
In the classical $M/M/1$ case, a criterion to determine the \textit{regime} of the queue is dictated by the traffic intensity $\rho=\frac{\lambda}{\mu}$, by means of the critical value $1$: if $\rho<1$, then the queue admits a limit distribution (since it is a positive recurrent continuous-time Markov chain), if $\rho \ge 1$ this fails. Furthermore, for $\rho=1$, the queue is still recurrent, but it is only null-recurrent, while for $\rho>1$ it is transient. One can refer to these three regimes as subcritical, critical and supercritical. Furthermore, these three regimes are strictly related to the scaling limits of the queue length process. Indeed, let $Q$ be the queue length process of a classical $M/M/1$ queue with inter-arrival and service rates $\lambda$ and $\mu$, and let us denote by $\mathbf{Q}_n$ its space-time scaling, i.e.
\begin{equation}\label{eq:rescaling}
	\mathbf{Q}_n(t)=\frac{Q(nt)}{n^\delta}
\end{equation}
for some $\delta>0$, that will be called a \textbf{scaling exponent}. Then the following limit theorem holds. Since this is a classical result, we omit the proof.
\begin{thm}\label{thm:classscal}
	Let $Q$ be the queue length process of a classical $M/M/1$ queue with inter-arrival and service rates $\lambda$ and $\mu$ and $\mathbf{Q}_n$ be its rescaling, defined in \eqref{eq:rescaling}. Let also $\rho=\frac{\lambda}{\mu}$ be the traffic intensity and $\mathbf{W}$ be a standard Brownian motion. Then:
	\begin{itemize}
		\item[$(i)$] If $\rho>1$ then for $\delta=1$ it holds $\mathbf{Q}_n \overset{\sf U}{\Rightarrow} (\lambda-\mu)\iota$;
		\item[$(ii)$] If $\rho=1$ then for $\delta=\frac{1}{2}$ it holds $\mathbf{Q}_n \overset{\sf U}{\Rightarrow} \sqrt{2\lambda}\Phi(\mathbf{W})$;
		%\item[$(iii)$] If $\rho<1$ then for $\delta=\frac{1}{2}$ it holds $\mathbf{Q}_n \overset{\sf U}{\Rightarrow} 0$.
		{\item[$(iii)$] If $\rho<1$ then for any $\delta>0$ it holds 
		%	$\mathbf{Q}_n \overset{\sf U}{\Rightarrow} 0$. In particular,
		\begin{equation*}
			\lim_{n \to +\infty}\sup_{t \ge 0}\mathbf{Q}_n(t)=0
			\end{equation*}
		almost surely.}
		\end{itemize}
\end{thm}
{Let us first stress that, clearly, item (iii) implies the convergence in distribution. Furthermore, despite these results are classical, we will briefly discuss them in the next section. Thanks to the previous theorem, we can recognize the different regimes by means of both the scaling exponent and the scaling limit we obtain. However, in the following, we will make use for simplicity of some non-linear scaling of time. In general, we will consider processes of the form
\begin{equation*}
	\mathbf{Q}_n(t)=\frac{Q(nm_n t)}{n^\delta},
\end{equation*}
and we will refer to $\delta$ as the \textbf{space scaling exponent}. If, furthermore, $m_n=n^{\gamma-1}$, we will refer to $\gamma$ as the \textbf{time scaling exponent}. In this case, thanks to the self-similarity of the queue length process $Q$, one could either fix $\gamma=1$, as in the statement of Theorem \ref{thm:classscal} or leave any time scaling exponent $\gamma>0$ and substitute the space scaling exponent $\delta$ with $\delta-\gamma+1$. In the following models, this will be not always possible. With this in mind, let us fix $\gamma=1$ and let us distinguish the three different regimes depending on $\delta$ and the scaling limit:}
%In the following, we will consider, for simplicity, some non-linear scaling of time. Nevertheless, the idea behind the scaling exponent does not change. Indeed, if we define $\mathbf{Q}_n$ differently as
%\begin{equation*}
%	\mathbf{Q}_n(t)=\frac{Q_n(n^\gamma t)}{n^\delta}
%\end{equation*}
%then the scaling exponent can be defined as the ratio $\kappa=\frac{\delta}{\gamma}$. From the scaling limit theorem for the classical $M/M/1$ queue, we notice that:
\begin{itemize}
	\item In the supercritical regime, the space scaling exponent is $\delta=1$ and the limit is proportional to the identity, hence, it is a non-decreasing process;
	\item As soon as we reach the critical regime, $\delta$ decreases to $1/2$ and we observe some oscillations: this also implies that if we keep $\delta=1$, then the limit process is $0$;
	%\item In the subcritical regime, keeping $\delta=1/2$ still leads to a $0$ process.
	{\item In the subcritical regime, any space scaling exponent $\delta>0$ leads to a $0$ process.}
\end{itemize}
Hence, we could recognize the three regimes as follows:
\begin{itemize}
	\item In the supercritical regime, the arrival processes leads the dynamics, generating at the scaling limit a non-decreasing process;
	\item In the critical regime, the two processes balance out and, with a finer scaling if necessary, we can capture their oscillations;
	\item In the subcritical regime, the service process leads the dynamics, keeping the queue length process at $0$. 
\end{itemize}
The previous scheme will be the rationale behind the classification of the regimes of the Mittag-Leffler queues. 
{The use of a finer scale could be dictated not only by reducing the space scaling exponent, but also by possibly changing the time scaling exponent, as it will be evident in the Mittag-Leffler GI/GI/1 ${}_5Q$.}

%It will be evident that the quantity $\rho=\frac{\lambda}{\mu}$, that does not play the role of traffic intensity in all the queueing models, will not be crucial in some of these models. For such a reason, we do not refer to the next results as heavy-traffic limits, but directly as scaling limits: indeed, the actual heavy-traffic limit only appears in the fractional M/M/1 queue, while in all the others the value of the traffic intensity will not play any prominent role.
%\begin{rmk}
%	While it will not play any role in the scaling limits we are going to present, the traffic intensity is quite important as one tries to determine transience or recurrence of the original renewal Mittag-Leffler queue ${}_2Q$, see \cite[Proposition 3.1]{butt2023queuing}.
%\end{rmk}
%Throughout this section, for any $\delta>0$ we will use $\iota_\delta$ to denote the function $\iota_\delta(t)=\delta t$ and we set $\iota:=\iota_1$. Let us also state that we will make extensive use of the continuous mapping theorem \cite[Theorem 3.4.3]{whitt2002stochastic}.

%\section{The Heavy-Traffic limit of the\textit{restless} Mittag-Leffler queue}
\subsection{The scaling limits of the fractional $M/M/1$ queue ${}_1Q$}\label{sec:MM11}
We can now proceed with the analysis of the scaling limits (and thus the regimes) of the previously introduced Mittag-Leffler queues. We start here with the fractional $M/M/1$ queue ${}_1Q$, that is scaled as follows:
\begin{equation}\label{eq:rescalingmod1}
	{}_1\mathbf{Q}_n(t)=\frac{{}_1Q(n^{\frac{1}{\alpha}}t)}{n^\delta}.
\end{equation}
 
\begin{thm}\label{thm:model1scaling}
	Let ${}_1Q$ be the queue length process of a fractional $M/M/1$ queue with order $\alpha \in (0,1)$ and generalized inter-arrival and service rates $\lambda$ and $\mu$. Define ${}_1\mathbf{Q}_n$ as in \eqref{eq:rescalingmod1}. Let also $\rho=\frac{\lambda}{\mu}$ be the generalized traffic intensity, $\mathbf{W}$ be a Brownian motion and $L_\alpha$ an inverse $\alpha$-stable subordinator independent of it. The following scaling limits hold true:
	\begin{itemize}
		\item[$(i)$] If $\rho>1$ then for $\delta=1$ it holds 
			${}_1\mathbf{Q}_n \overset{\sf U}{\Rightarrow} (\lambda-\mu)L_\alpha$;
		\item[$(ii)$] If $\rho=1$ then for $\delta=\frac{1}{2}$ it holds ${}_1\mathbf{Q}_n \overset{\sf U}{\Rightarrow} \sqrt{2\lambda}\Phi(\mathbf{W}) \circ L_\alpha$;
		%\item[$(iii)$] If $\rho<1$ then for $\delta=\frac{1}{2}$ it holds ${}_1\mathbf{Q}_n \overset{\sf U}{\Rightarrow} 0$.
		{\item[$(iii)$] If $\rho<1$ then for any $\delta>0$ it holds 
		\begin{equation*}
			\lim_{n \to +\infty}\sup_{t>0}{}_1\mathbf{Q}_n(t)=0.
		\end{equation*}	
			}
	\end{itemize}
\end{thm}
From the previous theorem, we notice that the generalized traffic intensity $\rho$ still plays a prominent role in determining the regime of the queueing system{, while the time scaling exponent $\gamma=\frac{1}{\alpha}$}. Indeed, if we differentiate the regimes as discussed in Section \ref{sec:MM1}, we have the following scheme:
\begin{itemize}
	\item For $\rho>1$ we are in the \textbf{supercritical regime}: in such a case, with a space scaling exponent $\delta=1$, the rescaled queue length process ${}_1\mathbf{Q}_n$ converges in distribution towards a non-decreasing stochastic process $(\lambda-\mu)L_\alpha$, indicating that the arrival counting process is driving the dynamics of the queue. 
	\item For $\rho=1$ we are in the \textbf{critical regime}: in this case, with a space scaling exponent $\delta=\frac{1}{2}$, the rescaled queue length process ${}_1\mathbf{Q}_n$ converges in distribution towards a non-monotone stochastic process $\sqrt{\lambda}\Phi(\mathbf{W})\circ L_\alpha$, that is a delayed reflected Brownian motion. This is due to the fact that arrival and service processes tend to balance, as inter-arrival and service times share the same distribution.
	\item For $\rho<1$ we are in the \textbf{subcritical regime}: in this case the rescaled queue length process ${}_1\mathbf{Q}_n$ converges almost surely towards $0$ independent of the choice of the space scaling exponent, hence even on extremely fine scale. We cannot say directly that this is due to the fact that mean service times are shorter than mean inter-arrival times, as both of their expected values are infinite. However, it is still true that for a given constant $M>0$, the probability of having a service time longer than $t>M$ is lower than the one of having a inter-arrival time longer than $t$.
\end{itemize}
{Notice further that since $L_\alpha$ is $\alpha$-self-similar, even in this case we could consider a generic time scaling exponent $\gamma>0$ and then substitute $\delta$ with $\delta-\alpha\gamma+1$.}

{Concerning the interpretation of the generalized traffic intensity $\rho$, that plays the role of our regime discrimination parameter, let} us notice that, unless $\alpha=1$, we have 
\begin{equation*}
\lim_{t \to +\infty}\frac{\bP(S_1>t)}{\bP(T_1>t)}=\rho,
\end{equation*}
thus the generalized traffic intensity, in this case, is recovered from the asymptotic comparison of the survival functions of the service and the inter-arrival times. Nevertheless, in general, the scaling behaviour, and thus the regimes, of Mittag-Leffler queueing models is not necessarily identified by means of the very same asymptotic comparison, as it will be evident for ${}_4Q$ and ${}_5Q$.

Finally, let us stress that the process $\Phi(\mathbf{W})\circ L_\alpha$ that appears as the scaling limit of ${}_1Q$ in the critical regime is a delayed reflected Brownian motion, that, in view of Lemma \ref{lem:scaling}, is equal to $\Phi(\mathbf{W}\circ L_\alpha)$. This process, together with a slightly modified drifted version, naturally arise when studying the \textit{heavy traffic limit} of ${}_1Q$, i.e when we consider the scaling ${}_1\mathbf{Q}_n=\frac{{}_1Q_n(n^{\frac{1}{\alpha}}t)}{\sqrt{n}}$ where the fractional $M/M/1$ queues admit varying generalized traffic intensities $\rho_n<1$ such that $\rho_n \uparrow 1$ and $\sqrt{n}(1-\rho_n) \to \zeta>0$. Such an heavy traffic limit, together with the properties of the reflected delayed Brownian motion, has been studied in \cite{ascione2022skorokhod}.

\subsection{The scaling limits of the fast renewal Mittag-Leffler queue ${}_2{Q}$}
Next, let us consider the fast renewal Mittag-Leffler queue ${}_2Q$. Here, we do not have a proper time-exponent, as it will be clear in case $\alpha+\beta=1$. Indeed, let us consider the scaled queue length process as follows:
\begin{equation}\label{eq:rescalingmod2}
	{}_2\mathbf{Q}_n(t)=\frac{{}_2Q(nm_nt)}{n^\delta},
\end{equation}
where, recalling the notation in Section \ref{Sec:model2}, we set
\begin{equation}\label{eq:mndef}
	m_n=\begin{cases}
		n^{\gamma-1} & \alpha+\beta \not = 1\\
		\frac{n \sin(\pi \alpha)}{2\lambda \mu}\E\left[\sin\left(\frac{2\lambda \mu Y_1}{n\sin(\pi \alpha)}\right)\right] & \alpha+\beta=1,
	\end{cases}\qquad \gamma=\max\{1,(\alpha+\beta)^{-1}\}.
\end{equation}
We have the following result.
\begin{thm}\label{thm:fastren}
	Let ${}_2Q$ be the queue length process of a fast renewal Mittag-Leffler queue with orders $\alpha,\beta \in (0,1)$ and rates $\lambda,\mu>0$. Let ${}_2\mathbf{Q}_n$ be its rescaling defined in \eqref{eq:rescalingmod2}. Denote by $\mathbf{W}$ a standard Brownian motion and let $L_{\alpha+\beta}$ be an independent inverse $(\alpha+\beta)$-stable subordinator if $\alpha+\beta<1$, while $L_{\alpha+\beta}\equiv \iota$ if $\alpha+\beta \ge 1$. Define also
	\begin{equation}\label{eq:giantconstant}
		\mathcal{K}_{\lambda,\mu}^{\alpha,\beta}=\begin{cases} \left(\frac{\lambda \mu \Gamma(1-\alpha)\Gamma(1-\beta)}{\Gamma(1-\alpha-\beta)}\right)^{\frac{1}{\alpha+\beta}} & \alpha+\beta<1 \\
			1 & \alpha+\beta=1 \\
			\frac{1}{\E[Y_1]} & \alpha+\beta>1.
		\end{cases}
	\end{equation}
	The following scaling limits hold true:
	\begin{itemize}
		\item[(i)] If $p_{\lambda,\mu}^{\alpha,\beta}>\frac{1}{2}$ then for $\delta=1$ it holds ${}_2\mathbf{Q}_n \overset{\sf U}{\Rightarrow} \left( 2p_{\lambda,\mu}^{\alpha,\beta}-1\right)L_{\alpha+\beta}\circ \iota_{\mathcal{K}_{\lambda,\mu}^{\alpha,\beta}}$
		\item[(ii)] If $p_{\lambda,\mu}^{\alpha,\beta}=\frac{1}{2}$ then for $\delta=1/2$ it holds ${}_2\mathbf{Q}_n \overset{\sf U}{\Rightarrow} \Phi(\mathbf{W})\circ L_{\alpha+\beta}\circ \iota_{\mathcal{K}_{\lambda,\mu}^{\alpha,\beta}}$
		\item[(iii)] If $p_{\lambda,\mu}^{\alpha,\beta}<\frac{1}{2}$ then for any $\delta>0$ it holds
		\begin{equation*}
			\lim_{n \to +\infty}\sup_{t \ge 0}{}_2\mathbf{Q}_n(t)=0.
		\end{equation*}
	\end{itemize}
\end{thm}
The previous theorem allows us to recognize the three different regimes of the fast renewal Mittag-Leffler queue ${}_2{Q}$. Before proceeding with the classification, let us stress that due to the non-power scaling we have in case $\alpha+\beta=1$, we cannot discuss in general about time scaling exponents. 
\begin{itemize}
	\item For $p_{\lambda,\mu}^{\alpha,\beta}>\frac{1}{2}$ we are in the \textbf{supercritical regime}: in such a case, the rescaled queue length process ${}_2{\mathbf{Q}}_n$ converges in distribution towards a non-decreasing stochastic (degenerate if $\alpha+\beta \ge 1$) process $(2p_{\lambda,\mu}^{\alpha,\beta}-1)L_{\alpha+\beta} \circ \iota_{\mathcal{K}_{\lambda,\mu}^{\alpha,\beta}}$, hence the arrival counting process is driving the dynamics of the queue with a space scale dictated by the space scaling exponent $\delta=1$ and a (possibly, non-power) time scale $nm_n$ depending on $\alpha+\beta$. 
	\item For $p_{\lambda,\mu}^{\alpha,\beta}=\frac{1}{2}$ we are in the \textbf{critical regime}: if we keep the same space and time scales as in the supercritical regime, the queue converges towards $0$. Nevertheless, if we reduce the space scale to $\delta=\frac{1}{2}$, we see some oscillations appearing as the rescaled queue length process ${}_2{\mathbf{Q}}_n$ converges in distribution towards a non-monotone stochastic process $\Phi(\mathbf{W})\circ L_{\alpha+\beta}\circ \iota_{\mathcal{K}_{\lambda,\mu}^{\alpha,\beta}}$, that is again a delayed reflected Brownian motion. 
	\item For $p_{\lambda,\mu}^{\alpha,\beta}<\frac{1}{2}$ we are in the \textbf{subcritical regime}: in this case, the rescaled queue length process ${}_2\mathbf{Q}_n$ converges almost surely towards $0$ independently of the choice of the space scaling exponent.
\end{itemize}
\subsection{The scaling limits of the renewal Mittag-Leffler queue ${}_3{Q}$}
Concerning the renewal Mittag-Leffler queue ${}_3{Q}$, we will see that, despite the different behaviour in terms of idle periods, it shares the scaling limits with the fast renewal Mittag-Leffler queue ${}_2Q$. Indeed, let us consider the scaled process
\begin{equation}\label{eq:2Qresc}
	{}_3{\mathbf{Q}}_n(t)=\frac{{}_3{Q}(nm_n t)}{n^\delta},
\end{equation}
where $m_n$ is defined in \eqref{eq:mndef} with $\gamma=\max\{1,(\alpha+\beta)^{-1}\}$. The following result holds.
\begin{thm}\label{thm:scalingmodel3}
	Let ${}_3Q$ be the queue length process of a renewal Mittag-Leffler queue with orders $\alpha,\beta \in (0,1)$ and rates $\lambda,\mu>0$. Let ${}_3\mathbf{Q}_n$ be its rescaling defined in \eqref{eq:2Qresc}. Denote by $\mathbf{W}$ a standard Brownian motion and let $L_{\alpha+\beta}$ be an independent inverse $(\alpha+\beta)$-stable subordinator if $\alpha+\beta<1$, while $L_{\alpha+\beta}\equiv \iota$ if $\alpha+\beta \ge 1$. Let also $\mathcal{K}_{\lambda,\mu}^{\alpha,\beta}$ as in \eqref{eq:giantconstant}.
	%Let $\alpha,\beta \in (0,1]$, $\lambda,\mu>0$, $\gamma=\max\{1,(\alpha+\beta)^{-1}\}$, $m_n$ as in \eqref{eq:mndef} and $\mathcal{K}_{\lambda,\mu}^{\alpha,\beta}$ as in \eqref{eq:giantconstant}. Let ${}_3{Q}$ be the queue length process of a modified renewal Mittag-Leffler queue and ${}_3{\mathbf{Q}}_n$ be its rescaling defined in \eqref{eq:2Qresc}. Denote by $\mathbf{W}$ a Brownian motion and let $L_{\alpha+\beta}$ be an inverse $(\alpha+\beta)$-stable subordinator independent of it if $\alpha+\beta<1$, while $L_{\alpha+\beta}\equiv \iota$ if $\alpha+\beta \ge 1$. 
	The following scaling limits hold true:
	\begin{itemize}
		\item[$(i)$] If $p_{\lambda,\mu}^{\alpha,\beta}>\frac{1}{2}$ then for $\delta=1$ it holds
		${}_3{\mathbf{Q}}_n \overset{\sf U}{\Rightarrow}(2p_{\lambda,\mu}^{\alpha,\beta}-1)L_{\alpha+\beta}\circ \iota_{\mathcal{K}_{\lambda,\mu}^{\alpha,\beta}}$.
		\item[$(ii)$] If $p_{\lambda,\mu}^{\alpha,\beta}=\frac{1}{2}$ then for $\delta=\frac{1}{2}$ it holds ${}_3{\mathbf{Q}}_n \overset{\sf U}{\Rightarrow} \Phi(\mathbf{W})\circ L_{\alpha+\beta}\circ \iota_{\mathcal{K}_{\lambda,\mu}^{\alpha,\beta}}$;
		\item[$(iii)$] If $p_{\lambda,\mu}^{\alpha,\beta}<\frac{1}{2}$, then for any $\delta>0$ it holds 
		\begin{equation*}
			\lim_{n \to +\infty}\sup_{t \ge 0}{}_3\mathbf{Q}_n(t)=0 \quad \mbox{a.s.}
		\end{equation*}
	\end{itemize}
\end{thm}

The three different regimes of the renewal Mittag-Leffler queue ${}_3Q$ can be then recognize analogously as we already did for ${}_2Q$:
%The previous theorem allows us to recognize the three different regimes of the renewal Mittag-Leffler queue ${}_3{Q}$. Before proceeding with the classification, let us stress that due to the non-power scaling we have in case $\alpha+\beta=1$, we cannot discuss in general about scaling exponents. 
\begin{itemize}
	\item For $p_{\lambda,\mu}^{\alpha,\beta}>\frac{1}{2}$ we are in the \textbf{supercritical regime}: in such a case, the rescaled queue length process ${}_3{\mathbf{Q}}_n$ converges in distribution towards a non-decreasing stochastic (degenerate if $\alpha+\beta \ge 1$) process $(2p_{\lambda,\mu}^{\alpha,\beta}-1)L_{\alpha+\beta} \circ \iota_{\mathcal{K}_{\lambda,\mu}^{\alpha,\beta}}$, hence the arrival counting process is driving the dynamics of the queue with a space scale dictated by the space scaling exponent $\delta=1$ and a (non-power) time scale $nm_n$ depending on $\alpha+\beta$. 
	\item For $p_{\lambda,\mu}^{\alpha,\beta}=\frac{1}{2}$ we are in the \textbf{critical regime}: again, the choice of a finer space scaling exponent, i.e., $\delta=\frac{1}{2}$, allows us to underline the oscillations appearing in the limit, since the rescaled queue length process ${}_3{\mathbf{Q}}_n$ converges in distribution towards a non-monotone stochastic process $\Phi(\mathbf{W})\circ L_{\alpha+\beta}\circ \iota_{\mathcal{K}_{\lambda,\mu}^{\alpha,\beta}}$. 
	\item For $p_{\lambda,\mu}^{\alpha,\beta}<\frac{1}{2}$ we are in the \textbf{subcritical regime}: in this case, independently of the space scaling exponent $\delta>0$, the rescaled queue length process ${}_3\mathbf{Q}_n$ always converges a.s. towards $0$.
\end{itemize}
Notice that one can recognize the regime by a direct comparison of the survival functions of ${S}_1$ and ${T}_1$, according to Proposition \ref{prop:comparison}. Let us also stress that in the critical regime the scaling limit of both the fast renewal and the renewal Mittag-Leffler queue is described by means of a delayed reflected Brownian motion (if $\alpha+\beta<1$), as in the case of the fractional $M/M/1$ queue. Despite such a similarity seems to suggest that this behaviour could be due to the sole fact that the inter-event times of the process are heavy-tailed, this is not the case, as we will see for both the restless Mittag-Leffler queue and the $GI/GI/1$ Mittag-Leffler queue, in which the critical regime exhibits oscillations due to more \textit{exotic} processes.

\subsection{The scaling limits of the restless Mittag-Leffler queue ${}_4Q$}
The scaling limits for the restless Mittag-Leffler queue ${}_4Q$ have been already obtained in \cite[Theorem 3.3]{butt2023queuing}, with a fixed space scale $\delta=1$. Here we only refine the result in the subcritical case. We set
\begin{equation}\label{eq:3Qresc}
	{}_4\mathbf{Q}_n(t):=\frac{{}_4Q(n^\gamma t)}{n^\delta}
\end{equation}
where $\gamma=\frac{1}{\max\{\alpha,\beta\}}$ and $\delta>0$. Then the following result holds.
%
%
% We recall here the results, for comparison with the other models. Precisely, let us consider
%for some $\gamma>0$ that will be chosen depending on the orders $\alpha$ and $\beta$ of the involved Mittag-Leffler random variables. Let us in particular recall that
%\begin{equation*}
%	Q(t)=\Phi(N^{\sf a}-N^{\sf d})(t)
%\end{equation*}
%where $N^{\sf a}$ and $N^{\sf d}$ are two independent fractional Poisson processes respectively of order $\alpha$ and $\beta$ and with rate $\lambda$ and $\mu$. We resume here \cite[Theorem 3.3]{butt2023queuing}.
\begin{thm}\label{thm:limits3Q}
Let ${}_4Q$ be the queue length process of a restless Mittag-Leffler queue of order $\alpha,\beta \in (0,1]$ with $\min\{\alpha,\beta\}<1$ and generalized rates $\lambda,\mu>0$. Let also ${}_4\mathbf{Q}_n$ be its rescaling defined in \eqref{eq:3Qresc}. Denote by $L_\alpha^{\sf a}$ and $L^{\sf d}_{\alpha}$ are two independent inverse $\alpha$-stable subordinators. The following scaling limits hold true:
\begin{itemize}
	\item[$(i)$] If $\alpha>\beta$ then for $\delta=1$ it holds ${}_4\mathbf{Q}_n \overset{\sf U}{\Rightarrow} \lambda L^{\sf a}_\alpha$
	\item[$(ii)$] If $\alpha=\beta$ then for $\delta=1$ it holds ${}_4\mathbf{Q}_n \overset{\sf U}{\Rightarrow} \Phi(\lambda L^{\sf a}_\alpha-\mu L^{\sf d}_\alpha)$;	
	\item[$(iii)$] If $\alpha<\beta$ then for $\delta=\max \left\{\alpha\beta^{-1},2^{-1}\right\}$, it holds ${}_4\mathbf{Q}_n \overset{\sf U}{\Rightarrow} \mathbf{0}$.
\end{itemize}	
\end{thm}
\begin{rmk}
We notice that the case $\alpha=\beta=1$ is covered by the scaling limits of the classical $M/M/1$ queue, as in Theorem \ref{thm:classscal}. 	
\end{rmk}
Concerning the regimes, we have the following classification:
\begin{itemize}
	\item For $\alpha>\beta$ we are in the \textbf{supercritical regime}: in such a case, the rescaled queue length process ${}_4\mathbf{Q}_n$ converges towards a non-decreasing process $\lambda L_\alpha$ and the arrival counting process drives the limit dynamics;
	\item For $\alpha=\beta \not = 1$ we are in the \textbf{critical regime}: in this case we see oscillations even in the same scale as in the supercritical regime. This is due to the fact that despite sharing the same scaling, the arrival and departure processes $N^{\sf a}$ and $N^{\sf d}$ converge towards two independent non-degenerate stochastic processes (that are equal in law up to the multiplicative constants) that do not cancel out even if $\lambda=\mu$. It is also worth noticing that $\lambda L^{\sf a}-\mu L^{\sf d}$ is not even monotone, as $L^{\sf a}$ can increase on a constancy interval of $L^{\sf d}$, hence $\Phi(\lambda L^{\sf a}-\mu L^{\sf d})$ is not identically $0$.
	\item For $\alpha<\beta$ we are in the \textbf{subcritical regime} and the scaled queue length converges uniformly to $0$ even in a finer space-scale: notice that in such a case the space scaling still depends on $\alpha$ and $\beta$. In this case we are not able to prove the almost surely convergence of the queue length process to $0$. Actually, the queue length process ${}_4Q$ is not regenerative on its cycles (i.e., the constancy intervals of the cumulative idle time), hence we could not use classical extreme value theory even if we knew the distribution of the cycle maximum, as in \cite{asmussen1998extreme}.
\end{itemize}
\subsection{The scaling limits of the Mittag-Leffler GI/GI/1 queue ${}_5Q$}\label{sec:MM15}
We finally discuss the scaling limits of a Mittag-Leffler GI/GI/1 queue. To better handle this, we will consider a fixed space scaling exponent $\delta=1$, while we change the time scaling exponent. Indeed, let us consider the scaled queue length process
\begin{equation}\label{eq:4Qresc}
	{}_5\mathbf{Q}_n(t)=\frac{{}_5Q(n^\gamma t)}{n},
\end{equation}
where $\gamma=\min\{\alpha^{-1},\beta^{-1}\}$. In this case, we have the following result.
\begin{thm}\label{thm:scalingmodel5}
	Let ${}_5Q$ be the queue length process of a $GI/GI/1$ Mittag-Leffler queue of orders $\alpha,\beta \in (0,1]$ with $\min\{\alpha,\beta\}<1$ and generalized inter-arrival and service rates $\lambda,\mu>0$. Let ${}_5\mathbf{Q}_n$ be its rescaling defined in \eqref{eq:4Qresc}.
	%Let $\alpha,\beta \in (0,1]$ with $\min\{\alpha,\beta\}<1$ and $\lambda,\mu>0$. Let ${}_4Q$ be the queue length process of a $GI/GI/1$ Mittag-Leffler queue and ${}_4\mathbf{Q}_n$ be its rescaling defined in \eqref{eq:4Qresc}. 
	Denote by $\sigma_\alpha^{\sf a}$ and $\sigma_\beta^{\sf d}$ two independent $\alpha$-stable and $\beta$-stable subordinators with inverses $L_\alpha^{\sf a}$ and $L_\beta^{\sf d}$ (where we set $\sigma_1^{\sf a}=\iota$ and $\sigma_1^{\sf d}=\iota$). The following scaling limits hold true:
	\begin{itemize}
		\item[$(i)$] If $\alpha>\beta$ it holds ${}_5\mathbf{Q}_n \overset{\sf U}{\Rightarrow} L_\alpha^{\sf a} \circ \iota_{\lambda^{\frac{1}{\alpha}}}$;
		\item[$(ii)$] If $\alpha=\beta$ it holds
		\begin{equation*}
			{}_5\mathbf{Q}_n \overset{\sf M_1}{\Rightarrow} L_\alpha^{\sf a} \circ \iota_{\lambda^{\frac{1}{\alpha}}}-L_\alpha^{\sf d}\circ \iota_{\mu^{\frac{1}{\alpha}}}\circ \mathbf{B},
		\end{equation*}
		where
		\begin{equation}\label{eq:Bproc}
			\mathbf{B}:=(\iota-\Psi(\mu^{-\frac{1}{\alpha}}\sigma_\alpha^{\sf d}\circ (L_\alpha^{\sf a} \circ \iota_{\lambda^{\frac{1}{\alpha}}}-\iota)))
		\end{equation}
		\item[$(iii)$] If $\alpha<\beta$ it holds ${}_5\mathbf{Q}_n \overset{\sf U}{\Rightarrow} \mathbf{0}$.
	\end{itemize}
\end{thm}
With this result in mind, we can recognize the different traffic regimes of the $GI/GI/1$ Mittag-Leffler queue as follows:
\begin{itemize}
	\item For $\alpha>\beta$ we are in the \textbf{supercritical regime}: in such a case, the rescaled queue length process ${}_5\mathbf{Q}_n$ converges towards a non-decreasing process $L_\alpha \circ \iota_{\lambda^{\frac{1}{\alpha}}}\overset{d}{=} \lambda L_\alpha$ and the arrival counting process drives the limit dynamics. Notice that this is exactly the same scaling limit as in the restless Mittag-Leffler queue: this is due to the fact that the arrival mechanism is left unaltered and once the queue is non-empty also the service times are the same.
	\item For $\alpha=\beta \not = 1$ we are in the \textbf{critical regime}: in this case we see oscillations even in the same scale as in the supercritical regime. While the motivation of this scaling is the same as for the restless Mittag-Leffler queue, the limit process is extremely different, due to the fact that in the $GI/GI/1$ model the departure process only moves when the cumulative busy time increases. 	
%	 This is due to the fact that despite sharing the same scaling, the arrival and departure processes $N^{\sf a}$ and $N^{\sf d}$ converge towards two independent non-degenerate stochastic processes (that are equal in law up to the multiplicative constants) that do not cancel out even if $\lambda=\mu$. It is also worth noticing that $\lambda L^{\sf a}-\mu L^{\sf d}$ is not even monotone, as $L^{\sf a}$ can increase on a constancy interval of $L^{\sf d}$, hence $\Phi(\lambda L^{\sf a}-\mu L^{\sf d})$ is not identically $0$.
	\item For $\alpha<\beta$ we can still recognize a \textbf{subcritical regime}, despite we are not able to improve the scaling. 
	% Furthermore, due to the techniques adopted here, we only have ${\sf M_1}$ convergence to $0$, in place of the uniform one.
\end{itemize}
%If we notice that for $\alpha \le \beta$ we could write
%\begin{align*}
%	{}_4\mathbf{Q}_n&=\mathbf{N}_n^{\sf d}-\mathbf{N}_n^{\sf d}\circ \mathbf{B}_n+c_n(\iota-\mathbf{B}_n)-c_n(\iota-\mathbf{B}_n)\\
%	&=\Phi(\mathbf{N}_n^{\sf d}-\mathbf{N}_n^{\sf d} \circ \mathbf{B}_n-c_n(\iota-\mathbf{B}_n)),
%\end{align*}
%for any sequence $c_n \to +\infty$, then we could try to argue as in Theorem \ref{thm:limits3Q} to get the uniform convergence in the subcritical case. Nevertheless, since we are using explicitly Theorem \ref{thm:contcomp}, that requires the ${\sf M}_1$ topology, we are not able to improve the convergence of the vector $(\mathbf{N}_n^{\sf a},\mathbf{N}_n^{\sf d}, \mathbf{B}_n)$ to the ${\sf J}_1$ or ${\sf U}$ topology, that is instead required to use Theorem \ref{thm:contcent2}. 
{Notice that we can guarantee the convergence in distribution with respect to the locally uniform topology only in the supercritical and subcritical regimes, since in the critical one the limit is not a.s. continuous, while in the proof we explicitly use the continuity of the composition operator in the ${\sf M_1}$ convergence. Furthermore, it is worth recalling that, since we are considering a proper GI/GI/1 queueing system, the process ${}_5Q$ is regenerative, where the regeneration cycles are given by the intervals of non-constancy of the cumulative idle time. Nevertheless, the cycle length is bigger than at least one service time, hence if $\beta<1$ it has infinite mean. As a consequence, we cannot use the approach provided in \cite{asmussen1998extreme} to deduce the asymptotic distribution of the supremum of the queue by means of extreme value theory. A deeper investigation in this direction is indeed required.} 
%Hence, to handle these models, a deeper study on the relation between non-linear centering and reflection of c\'adl\'ag functions is required.

% Please add the following required packages to your document preamble:
% \usepackage[table,xcdraw]{xcolor}
% Beamer presentation requires \usepackage{colortbl} instead of \usepackage[table,xcdraw]{xcolor}
\begin{sidewaystable}
	\centering
	\begin{tabular}{|l|l|l|l|l|l|l|}
		\hline
		\rowcolor[HTML]{C0C0C0} 
		Queueing Model & IAD & SD & RDP & CV & Scaling & Scaling Limits\\ \hline
		$M/M/1$ &  ${\sf Exp}(\lambda)$ &  ${\sf Exp}(\mu)$ & $\rho=\frac{\lambda}{\mu}$ & $1$ &  \makecell{$n^{-\delta} Q_n(nt)$ \\[2pt] $\delta=\begin{cases} 1 \\ 1/2 \\ >0 \end{cases}$} & \makecell{$\begin{cases} (\lambda-\mu)\iota \\ \sqrt{2\lambda}\Phi(\mathbf{W}) \\ 0 \ \mbox{(a.s.)} \end{cases}$} \\ \hline
		Fractional $M/M/1$ & ${\sf ML}_\alpha(\lambda)$ & ${\sf ML}_\alpha(\mu)$  & $\rho=\frac{\lambda}{\mu}$ & $1$ & \makecell{$n^{-\delta}{}_1Q_n(n^{\gamma}t)$ \\[2pt] $\delta=\begin{cases} 1 \\ 1/2 \\ >0 \end{cases}$ \\[2pt] $\gamma^{-1}=\alpha$} & \makecell{$\begin{cases} (\lambda-\mu)L_\alpha \\ \sqrt{2\lambda}\Phi(\mathbf{W})\circ L_\alpha \\ 0 \ \mbox{(a.s.)} \end{cases}$} \\ \hline
		Fast Renewal ML Queue &  Theorem \ref{thm:interarrival2} & Theorem \ref{thm:service2} & $p_{\lambda,\mu}^{\alpha,\beta}$ & $\frac{1}{2}$ & \makecell{$n^{-\delta}{}_2{Q}_n(nm_nt)$ \\[2pt]
			$m_n$ in \eqref{eq:mndef}\\[2pt]
			$\delta=\begin{cases} 1 \\ 1/2 \\ >0 \end{cases}$} & \makecell{$\begin{cases} (2p_{\lambda,\mu}^{\alpha,\beta}-1)L_{\alpha+\beta}\circ \iota_{K_{\lambda,\mu}^{\alpha,\beta}} \\
				\Phi(\mathbf{W})\circ L_{\alpha+\beta}\circ \iota_{\mathcal{K}_{\lambda,\mu}^{\alpha,\beta}} \\ 0 \ \mbox{(a.s.)} \end{cases}$\\ $\mathcal{K}_{\lambda,\mu}^{\alpha,\beta}$ in \eqref{eq:giantconstant}} \\ \hline
		Renewal ML Queue & Theorem \ref{thm:interarrival22} & Theorem \ref{thm:service2} & $p^{\alpha,\beta}_{\lambda,\mu}$ &  $\frac{1}{2}$ & \makecell{$n^{-\delta}{}_3{Q}_n(nm_nt)$ \\[2pt]
			$m_n$ in \eqref{eq:mndef}\\[2pt]
			$\delta=\begin{cases} 1 \\ 1/2 \\ >0 \end{cases}$} & \makecell{$\begin{cases} (2p_{\lambda,\mu}^{\alpha,\beta}-1)L_{\alpha+\beta}\circ \iota_{K_{\lambda,\mu}^{\alpha,\beta}} \\
			\Phi(\mathbf{W})\circ L_{\alpha+\beta}\circ \iota_{\mathcal{K}_{\lambda,\mu}^{\alpha,\beta}} \\ 0 \ \mbox{(a.s.)}\end{cases}$\\ $\mathcal{K}_{\lambda,\mu}^{\alpha,\beta}$ in \eqref{eq:giantconstant}} \\ \hline
		Restless ML Queue & ${\sf ML}_\alpha(\lambda)$ & Proposition \ref{prop:distserv3} & $\frac{\alpha}{\beta}$ & $1$ & \makecell{$n^{-\delta}{}_4Q\left(n^{\gamma}t\right)$ \\[2pt] $\delta=\begin{cases} 1 \\ 1 \\ \max\left\{\frac{\alpha}{\beta},\frac{1}{2}\right\}\end{cases}$ \\[2pt] $\gamma^{-1}=\max\{\alpha,\beta\}$} & $\begin{cases} \lambda L_\alpha^{\sf a} \\ \Phi(\lambda L_\alpha^{\sf a}-\mu L_{\alpha}^{\sf d})\\
		0\end{cases}$ \\ \hline
		ML $GI/GI/1$ Queue & ${\sf ML}_\alpha(\lambda)$ & ${\sf ML}_{\beta}(\mu)$ & $\frac{\alpha}{\beta}$ & $1$ & \makecell{$n^{-1}{}_5Q(n^{\gamma}t)$ \\[2pt] $\gamma^{-1}=\max\{\alpha,\beta\}$} & \makecell{$\begin{cases}
			\lambda L_{\alpha}^{\sf a} \\
			L_{\alpha}^{\sf a} \circ \iota_{\lambda^{\frac{1}{\alpha}}}-L_{\alpha}^{\sf a} \circ \iota_{\mu^{\frac{1}{\alpha}}} \circ \mathbf{B}\\
			0
		\end{cases}$\\[2pt]
	$\mathbf{B}$ in \eqref{eq:Bproc}} \\ \hline
	\end{tabular}
	\caption{Summary of the characteristics of the considered queueing models. Here IAD is the distribution of inter-arrival times, SD is the distribution of service times, RDP is the regime discrimination parameter, CV is the critical value. Furthermore, whenever three options are proposed, they are disposed, from above to below, in the following order: supercritical case, critical case, subcritical case.}
\end{sidewaystable}
\section{Scaling limits of Mittag-Leffler queues: proofs of the main results}\label{sec:scalingproof}
\subsection{A brief discussion of Theorem \ref{thm:classscal}}\label{sec:MM1proof}
As we already stated, the results contained in Theorem \ref{thm:classscal} are classical and well-known. Nevertheless, let us briefly discuss some details of their proofs. First of all, items (i) and (ii) can be actually proved in several different ways. Indeed, it will be clear that the proof strategies we will adopt for Theorems \ref{thm:scalingmodel3} and \ref{thm:scalingmodel5}, together with the one used in \cite{butt2023queuing} for Theorem \ref{thm:limits3Q}. Concerning item (iii), the aforementioned arguments only guarantee for the desired convergence as $\delta\ge \frac{1}{2}$. To prove the second part of item (iii), notice that, if we define
\begin{equation*}
	\tau=\inf\{t \ge 0: \ Q(t)=0\}
\end{equation*}
then it holds for any $k=1,2,\dots$
\begin{equation*}
	\bP(\sup_{t \in [0,\tau]}Q(t) \ge k\mid Q(0)=1)=\frac{(1-\rho)\rho^{k-1}}{1-\rho^{k}} \sim \frac{1-\rho}{\rho}\rho^{k},
\end{equation*}
see \cite[Proposition 3.1]{asmussen1998extreme}. With this in mind, by usual extreme value theory arguments one can show that for any $x \ge 0$
\begin{equation}\label{eq:Gumbel}
	\lim_{t \to +\infty}\bP\left(\left(-\log(\rho)\sup_{s \in [0,t]}Q(s)-\log\left(\frac{ t}{\mu-\lambda}\right)-\log\left(\frac{1-\rho}{\rho}\right)\right) \le x\right)=e^{-e^{-x}},
\end{equation}
see \cite[Example 1.1]{asmussen1998extreme}.
Once this has been established, we notice that
\begin{equation*}
	\bP\left(\limsup_{n \to +\infty}\sup_{t \ge 0}\mathbf{Q}_n(t)>0\right)=\bP\left(\bigcup_{p \in \N}\bigcap_{n \ge p}\bigcup_{m \ge n}\sup_{t \ge 0}\mathbf{Q}_m(t)>p^{-1}\right).
\end{equation*}
Observe then that, by the monotone convergence theorem,
%, since $\sup_{t \ge 0}Q(mt)=\sup_{t \ge 0}Q(t)$ for any $m \in \N$,
\begin{align}\label{eq:limitrep}
	\begin{split}
	\bP\left(\sup_{t \ge 0}\mathbf{Q}_m(t)>p^{-1}\right)&=\bP\left(\sup_{t \ge 0}Q(mt)>\frac{m^{\delta}}{p}\right)=\lim_{h \to +\infty}\bP\left(\sup_{t \in [0,h]}Q(mt)>\frac{m^{\delta}}{p}\right)\\
	&=\lim_{h \to +\infty}\bP\left(\sup_{t \in [0,mh]}Q(t)>\frac{m^{\delta}}{p}\right).		
	\end{split}
\end{align}
Now assume that $h$ is big enough to have $\frac{mh}{\mu-\lambda}>1$. Then
\begin{align*}
	\bP&\left(\sup_{t \in [0,mh]}Q(t)>\frac{m^{\delta}}{p}\right) \le \bP\left(\sup_{t \in [0,mh]}Q(t)+\frac{1}{\log(\rho)}\log\left(\frac{mh}{\mu-\lambda}\right)>\frac{m^{\delta}}{p}\right)\\
	&=\bP\left(-\log(\rho)\sup_{t \in [0,mh]}Q(t)-\log\left(\frac{mh}{\mu-\lambda}\right)-\log\left(\frac{1-\rho}{\rho}\right)>K_{\rho,p,\delta}(m)\right),
\end{align*}
where
\begin{equation*}
	K_{\rho,p,\delta}(m)=\frac{m^{\delta}|\log(\rho)|}{p}-\log\left(\frac{1-\rho}{\rho}\right)
\end{equation*}
that combined with \eqref{eq:limitrep} and \eqref{eq:Gumbel} implies
\begin{equation*}
	\bP\left(\sup_{t \ge 0}\mathbf{Q}_m(t)>p^{-1}\right) \le 1-\exp\left(\exp\left(-K_{\rho,p,\delta}(m)\right)\right).
\end{equation*}
Recalling that for any $x \ge 0$ it holds $1-e^{-x} \le x$, we get
\begin{equation*}
	\bP\left(\sup_{t \ge 0}\mathbf{Q}_m(t)>p^{-1}\right) \le \exp\left(-K_{\rho,p,\delta}(m)\right),
\end{equation*}
where the right hand side satisfies for fixed $\rho$, $p$ and $\delta$
\begin{equation*}
	\sum_{m=1}^{+\infty}\exp\left(-K_{\rho,p,\delta}(m)\right)<\infty.
\end{equation*}
Hence the statement follows by the Borel-Cantelli Lemma.
\subsection{Preliminary auxiliary results on some Markov chains}\label{subsecMC}
To handle some of the models, we will also need to determine the scaling limits of the Markov chains defined in \eqref{eq:mCren1}, \eqref{eq:dtrw} and \eqref{eq:Mchain2}, with $p_{\lambda,\mu}^{\alpha,\beta}$. Actually, we only need to argue for $\alpha=\beta=1$. Indeed, for any $\lambda,\mu>0$ and $\alpha,\beta \in (0,1]$ we have that if we select
\begin{equation}\label{eq:lambdabar}
	\overline{\lambda}=\frac{\mu p_{\lambda,\mu}^{\alpha,\beta}}{1-p_{\lambda,\mu}^{\alpha,\beta}},
\end{equation}
then $p_{\overline{\lambda},\mu}^{1,1}=p_{\lambda,\mu}^{\alpha,\beta}$. For this reason, we will only state the results of this section in case $\alpha=\beta=1$ and we will omit the superscript. Let us start with $(r_k)_{k \ge 0}$ as in \eqref{eq:dtrw}. We consider the following scaled processes
\begin{equation*}
	\overline{\mathbf{r}}_n(t):=\frac{1}{\sqrt{n}}(r_{\lfloor nt \rfloor}-n(1-2p_{\lambda,\mu})t), \qquad \mathbf{r}_n(t):=\frac{r_{\lfloor nt \rfloor}}{n}.
\end{equation*}
Notice in particular that if $\delta=1$ then
\begin{equation*}
	\overline{\mathbf{r}}_n=\sqrt{n}(\mathbf{r}_n-(1-2p_{\lambda,\mu})\iota).
\end{equation*}
Hence, with a direct application of Donsker's Theorem \ref{thm:DTa} and observing that $\sqrt{n} \to +\infty$ we get the following result.
\begin{lem}\label{lem:rn}
	Let $\alpha,\beta \in (0,1]$ and $\lambda,\mu>0$. Then
	\begin{equation*}
		\overline{\mathbf{r}}_n \overset{\sf U}{\Rightarrow} 	2\sqrt{p_{\lambda,\mu}(1-p_{\lambda,\mu})}\mathbf{W}, \qquad \mathbf{r}_n \overset{\sf U}{\Rightarrow} (2p_{\lambda,\mu}-1)\iota,
	\end{equation*}
	where $\mathbf{W}$ is a Brownian motion.
\end{lem}
Concerning the Markov chain $(\overline{q}_k)_{k \ge 0}$ as in \eqref{eq:Mchain2}, we consider the scaling
\begin{equation}\label{eq:scalingoq}
	\mathbf{\overline{q}}_n(t):=\frac{\overline{q}_{\lfloor nt\rfloor}}{n^\delta}=n^{1-\delta}\Phi(\mathbf{r}_n).
\end{equation}
The following lemma is a direct consequence of Lemma \ref{lem:rn} and Theorem \ref{thm:classscal}.
\begin{lem}\label{lem:ovqk}
	Let $(\overline{q}_k)_{k \ge 0}$ be a Markov chain defined as in \eqref{eq:Mchain2} and $\overline{q}_n$ be as in \eqref{eq:scalingoq}. Let also $\mathbf{W}$ be a Brownian motion. The following scaling limits hold true:
	\begin{itemize}
		\item[(i)] If $p_{\lambda,\mu}>\frac{1}{2}$ then for $\delta=1$ it holds $\overline{\mathbf{q}}_n \overset{\sf U}{\Rightarrow} (2p_{\lambda,\mu}-1)\iota$;
		\item[(ii)] If $p_{\lambda,\mu}=\frac{1}{2}$ then for $\delta=\frac{1}{2}$ it holds $\overline{\mathbf{q}}_n\overset{\sf U}{\Rightarrow}\Phi(\mathbf{W})$;
		\item[(iii)] If $p_{\lambda,\mu}<\frac{1}{2}$ then for any $\delta>0$ it holds
		\begin{equation*}
			\lim_{n \to +\infty}\sup_{t>0}\mathbf{\overline{q}}_n(t)=0.
		\end{equation*}
	\end{itemize}
\end{lem}
\begin{proof}
	Items (i) and (ii) follow by Lemma \ref{lem:rn} and the continuity of the reflection map (Theorem \ref{thm:continuity}), once we recall that for $p_{\lambda,\mu}=\frac{1}{2}$ it holds
	\begin{equation*}
		\overline{\mathbf{q}}_n=\Phi(\overline{\mathbf{r}}_n).
	\end{equation*}
	To prove item (iii), notice that if $N$ is a Poisson process with rate $\lambda+\mu$, we have that $Q(t)=\overline{q}_{N(t)}$ is the queue length process of a M/M/1 queue with intensity traffic $\overline{\rho}=\frac{\overline{\lambda}}{\mu}<1$. Hence we get
	\begin{equation*}
		\sup_{t \ge 0}\overline{\mathbf{q}}_n(t)=n^{-\delta}\sup_{k \ge 0}\overline{q}_k=n^{-\delta}\sup_{t \ge 0}Q(t)=\sup_{t \ge 0}\mathbf{Q}_n(t) \to 0.
	\end{equation*} 
\end{proof}
Before proceeding, let us underline that $p_{\lambda,\mu} \le \frac{1}{2}$ if and only if $\rho=\frac{\lambda}{\mu} \le 1$, with equality occurring if and only if $\rho=1$. Now let us consider the Markov chain $(q_k)_{k \ge 0}$ defined as in \eqref{eq:mCren1}. We consider the following scaling
\begin{equation}\label{eq:scalingqn}
	\mathbf{q}_n(t)=\frac{q_{\lfloor nt \rfloor}}{n^\delta}.
\end{equation}
We will now derive the scaling limits for $\mathbf{q}_n$ from Theorem \ref{thm:classscal}.
\begin{lem}\label{lem:scalingMC2}
	Let $(q_k)_{k \ge 0}$ be a Markov chain as in \eqref{eq:mCren1}, $\mathbf{q}_n$ as in \eqref{eq:scalingqn} and $\mathbf{W}$ be a standard Brownian motion. Then
	\begin{itemize}
		\item[(i)] If $p_{\lambda,\mu}>\frac{1}{2}$ then for $\delta=1$ it holds $\displaystyle \mathbf{q}_n \overset{\sf U}{\Rightarrow}(2p_{\lambda,\mu}-1)\iota$
		\item[(ii)] If $p_{\lambda,\mu}=\frac{1}{2}$ then for $\delta=1/2$ it holds $\displaystyle \mathbf{q}_n \overset{\sf U}{\Rightarrow}\Phi(\mathbf{W})$
		\item[(iii)] If $p_{\lambda,\mu}<\frac{1}{2}$ then for any $\delta>0$ it holds 
		\begin{equation*}
		\lim_{n \to +\infty}\sup_{t \ge 0}\mathbf{q}_n(t)=0.
		\end{equation*}
	\end{itemize}
\end{lem}
\begin{proof}
	Let $(Y_k)_{k \ge 0}$ be as in \eqref{eq:Y1}. Let also $\widetilde{N}$ be the counting process of the $(Y_k)$, as in \eqref{eq:countingproc}. Then the process $Q(t)=q_{\widetilde{N}(t)}$ is the queue length process of a classical $M/M/1$ with traffic intensity $\rho=\frac{\lambda}{\mu}$. Let us immediately prove item (iii): since $\widetilde{N}:[0,+\infty) \to \N$ is surjective and $\rho<1$, we have
	\begin{equation*}
		\sup_{t \ge 0}\mathbf{q}_n(t)=n^{-\delta}\sup_{k \ge 0}q_k=n^{-\delta}\sup_{t \ge 0}Q(t)=\sup_{t \ge 0}\mathbf{Q}_n(t) \to 0.
	\end{equation*}
	To prove items (i) and (ii), with the notation of Section \ref{subs:ssqs}, let
	\begin{equation*}
		\mathbf{N}_n^{\sf a}(t)=n^{-1}N^{\sf a}(nt), \quad \mathbf{N}_n^{\sf d}(t)=n^{-1}N^{\sf d}(nt), \quad \mathbf{D}_n(t)=n^{-1}D(nt)
	\end{equation*}
	\begin{equation*}
		\mathbf{B}_n(t)=n^{-1}B(nt), \quad \mathbf{I}_n(t)=n^{-1}I(nt), \quad \mathbf{X}_n(t)=n^{-1}X(nt),
	\end{equation*}
	\begin{equation*}
	\mathbf{C}_n(t)=n^{-1}C(nt), \quad \mathbf{CS}_n(t)=n^{-1}CS_{\lfloor nt \rfloor}, \quad  \mathbf{G}_n(t)=n^{-1}G(t)
	\end{equation*}
	Then the following relations hold
	\begin{equation*}
		\mathbf{B}_n=\iota-\mathbf{I}_n, \quad \mathbf{I}_n=\Psi(\mathbf{X}_n), \quad \mathbf{X}_n=\mathbf{C}_n-\iota
	\end{equation*}	
	\begin{equation*}
		\mathbf{C}_n=\mathbf{CS}_n \circ \mathbf{N}_n^{\sf a}, \quad \mathbf{D}_n=\mathbf{N}^{\sf d}\circ \mathbf{B}_n, \quad \mathbf{G}_n=\mathbf{N}^{\sf a}_n+\mathbf{D}_n.
	\end{equation*}
	Once we recall the limit
	\begin{equation*}
		(\mathbf{N}_n^{\sf a}, \mathbf{N}_n^{\sf d}, \mathbf{CS}_n) \overset{\sf M_1}{\Rightarrow} (\lambda \iota, \mu \iota, \mu^{-1}\iota),
	\end{equation*}
	that clearly follows by Theorems \ref{thm:DTa} and \ref{thm:contcounting}, we immediately get
	\begin{equation*}
		\mathbf{C}_n \overset{\sf M_1}{\Rightarrow} \rho \iota, \quad \mathbf{X}_n \overset{\sf M_1}{\Rightarrow} (\rho-1) \iota.
	\end{equation*}
	Notice that we are only arguing for $\rho \ge 1$, hence $\rho-1 \ge 0$ and we have, by continuity of the regulator map as in Theorem \ref{thm:continuity},
	\begin{equation*}
		\mathbf{I}_n \overset{\sf M_1}{\Rightarrow} \mathbf{0}, \quad \mathbf{B}_n \overset{\sf M_1}{\Rightarrow} \iota, \quad \mathbf{D}_n \overset{\sf M_1}{\Rightarrow} \mu\iota, \quad 
		\mathbf{G}_n \overset{\sf M_1}{\Rightarrow} (\lambda+\mu) \iota.
	\end{equation*}
	Now let also
	\begin{equation*}
		\mathbf{J}_n(t)=n^{-1}J_{\lfloor nt \rfloor}
	\end{equation*}
	and use Theorem \ref{thm:contcounting} to guarantee that
	\begin{equation*}
		\mathbf{J}_n \overset{\sf M_1}{\Rightarrow} (\lambda+\mu)^{-1}\iota.
	\end{equation*}
	Now it remains to observe that $q_k=Q(J_k)$ for all $k \ge 0$, hence
	\begin{equation*}
		\mathbf{q}_n(t)=n^{-\delta}Q(J_{\lfloor nt \rfloor})=n^{-\delta}Q(n\mathbf{J}_n(t))=\mathbf{Q}_n \circ \mathbf{J}_n.
	\end{equation*}
	Since the limit of $\mathbf{Q}_n$ is continuous and the one of $\mathbf{J}_n$ is strictly increasing, we can use the continuous mapping theorem together with the fact that the composition is continuous in such a case (Theorem \ref{thm:contcomp}) to finally get the statement from Theorem~\ref{thm:classscal}.
\end{proof}
Now we can proceed with the proofs of the scaling limits.

% the following lemma is an immediate consequence of Donsker's theorem \ref{thm:DTa}.

\subsection{Proof of Theorem \ref{thm:model1scaling}}\label{sec:proof1}
The proof of items (i) and (ii) of Theorem \ref{thm:model1scaling} is an immediate application of the continuous mapping theorem. Indeed, by the $\alpha$-self-similarity of $L_\alpha$, we have that ${}_1\mathbf{Q}_n\overset{d}{=}\mathbf{Q}_n \circ L_\alpha$, where $Q$ is a classical M/M/1 queue and $\mathbf{Q}_n$ is defined as in \eqref{eq:rescaling}. Let $\mathbf{Q}_\infty$ be the limit in distribution of $\mathbf{Q}_n$ (independently of the value of $\rho$), as given in Theorem \ref{thm:classscal}. Then, since $L_\alpha$ is constant with respect to $n$ and independent of $\mathbf{Q}_n$, we clearly have $(\mathbf{Q}_n,L_\alpha) \overset{\sf M_1}{\Rightarrow} (\mathbf{Q}_\infty,L_\alpha)$. In particular, $\mathbf{Q}_\infty$ is continuous and $L_\alpha$ is increasing, thus $(\mathbf{Q}_\infty,L_\alpha)$ is a continuity point of the composition operator (see Theorem \ref{thm:contcomp}). Hence, by the continuous mapping theorem and the fact that the limits are continuous, we get the desired statement for $\rho \ge 1$. To prove item (iii) of Theorem \ref{thm:model1scaling} instead, let $Q$ be the M/M/1 queue such that ${}_1Q(t)=Q(L_\alpha(t))$ and let $\mathbf{Q}_n$ be defined as in \eqref{eq:rescaling}. Then, since $L_\alpha$ is continuous,
\begin{equation*}
\sup_{t \ge 0} {}_1\mathbf{Q}_n(t)=n^{-\delta}\sup_{t \ge 0}{}_1Q(mt)=n^{-\delta}\sup_{t \ge 0}Q(t)=\sup_{t \ge 0}\mathbf{Q}_n(t),
\end{equation*}
hence the statement follows by item (iii) in Theorem \ref{thm:classscal}.
\hfill \qed
\subsection{Proof of Theorem \ref{thm:fastren}}
Let us immediately observe that, with the notation of Sections \ref{Sec:model2} and \ref{subsecMC}, we have
\begin{equation*}
	{}_2\mathbf{Q}_n=\mathbf{q}_n \circ \widetilde{\mathbf{N}}_n, \quad \mbox{ where } \quad \widetilde{\mathbf{N}}_n(t)=\frac{\widetilde{N}(nm_nt)}{n}. 
\end{equation*}
Up to substituting $\lambda$ with $\overline{\lambda}$ in \eqref{eq:lambdabar}, we already know the limits of $\mathbf{q}_n$ thanks to Lemma \ref{lem:scalingMC2}. In particular, this immediately gives item (iii), since $\widetilde{N}:[0,+\infty) \to \N$ is surjective and then
\begin{equation*}
	\lim_{n \to +\infty}\sup_{t \ge 0}{}_2\mathbf{Q}_n(t)=\lim_{n \to +\infty}\sup_{t \ge 0}\mathbf{q}_n(t)=0.
\end{equation*}
Now we focus on items (i) and (ii). To prove them, we only need to identify the limits of $\widetilde{\mathbf{N}}_n$, so that the statement will follow by applying the continuous mapping theorem together with the continuity of the composition map as in Theorem \ref{thm:contcomp}. To do this, let us first define the non-decreasing processes
\begin{equation*}
	\mathbf{Y}_n(t)=\frac{1}{n^\gamma}\sum_{j=1}^{\lfloor nt \rfloor}Y_j
\end{equation*}
of which $\widetilde{\mathbf{N}}_n$ is the scaled counting process. Concerning functional limit theorems for $\mathbb{Y}_n$, we notice that, if $\max\{\alpha,\beta\}<1$, then
\begin{equation}\label{eq:asymptotic}
	\bP(Y_k>t)=E_\alpha(-\lambda t^\alpha)E_\beta(-\lambda t^\beta) \sim \frac{t^{-\alpha-\beta}}{\lambda \mu \Gamma(1-\alpha)\Gamma(1-\beta)},
\end{equation}
that implies that $Y_k$ belongs to the normal domain of attraction of a $(\alpha+\beta)$-stable law. On the other hand, if $\alpha<\beta=1$ we have
\begin{equation*}
	\bP(Y_k>t)=E_\alpha(-\lambda t^\alpha)e^{-\mu t} \sim \frac{t^{-\alpha}}{\lambda\Gamma(1-\alpha)}e^{-\mu t},
\end{equation*}
thus it belongs to the normal domain of attraction of a normal random variable. The latter also holds if $\beta<\alpha=1$ or $\beta=\alpha=1$.

Next, observe that for $\alpha+\beta \not = 1$ we clearly have $nm_n \to +\infty$. Furthermore, if $\alpha+\beta<1$, then $m_n \to +\infty$ while if $\alpha+\beta=1$ we have $m_n \equiv 1$. For $\alpha+\beta=1$, we need to prove the following lemma.
\begin{lem}
	It holds $\lim_{n \to \infty}m_n=+\infty$
\end{lem}
\begin{proof}
	It is sufficient to show that $\liminf_{n \to \infty}m_n=+\infty$. To do this, let us consider a non-relabelled $m_n$ converging towards the limit inferior. In general, for $\alpha+\beta=1$, by Theorem \ref{thm:DTa} we have
	\begin{equation}\label{eq:Stable1}
		\mathbf{Y}_n-m_n \iota \overset{\sf J_1}{\Rightarrow}\frac{\sin(\pi \alpha)}{2\mu \lambda}\mathbf{S},
	\end{equation}
	where $\mathbf{S} \sim \mathbf{S}_1(1,1)$. If $\lim_{n \to \infty}m_n=m_\infty \in \R$, then since $\iota$ is continuous, we can use Theorem \ref{thm:contsum} to get $\mathbf{Y}_n \overset{\sf M_1}{\Rightarrow}  \frac{\sin(\pi \alpha)}{2\mu \lambda}\mathbf{S}+m_\infty\iota$. Since $\mathbf{Y}_n$ is a sequence of non-decreasing processes, by Proposition \ref{prop:premon} we have that $\frac{\sin(\pi \alpha)}{2\mu \lambda}\mathbf{S}+m_\infty\iota$ must be non-decreasing as well, which is absurd. Furthermore, if $m_\infty=-\infty$, we get that $|m_n|^{-1}\mathbf{Y}_n\overset{\sf M_1}{\Rightarrow}-\iota$, that is again a contradiction since $|m_n|^{-1}\mathbf{Y}_n$ are non-decreasing and $-\iota$ is non-increasing.
\end{proof}
With this in mind, we can easily identify the limits of $\widetilde{\mathbf{N}}_n$.
\begin{lem}\label{prop:tildeN}
	Let $\alpha,\beta \in (0,1]$ and $\lambda,\mu>0$. Denote by $L_{\alpha+\beta}$ an inverse $(\alpha+\beta)$-stable subordinator if $\alpha+\beta<1$, while $L_{\alpha+\beta}=\iota$ if $\alpha+\beta \ge 1$. Then
	\begin{equation}\label{eq:limittildeN}
		\widetilde{\mathbf{N}}_n\overset{\sf U}{\Rightarrow} L_{\alpha+\beta} \circ \iota_{\mathcal{K}_{\alpha,\beta}^{\lambda,\mu}}.
	\end{equation}
\end{lem}
\begin{proof}
	In case $\alpha+\beta<1$, the statement is an immediate consequence of Theorems \ref{thm:DTa} and \ref{thm:contcounting}. Let us now assume that $\alpha+\beta>1$ with $\max\{\alpha,\beta\}<1$. In this case, we introduce the sequence of processes
	\begin{equation*}
		\overline{\mathbf{Y}}_n=n^{1-\frac{1}{\alpha+\beta}}\left(\mathbf{Y}_n-\mathcal{K}^{\alpha,\beta}_{\lambda,\mu}\iota\right).	
	\end{equation*} 
	By Theorem \ref{thm:DTa} we know that $\overline{\mathbf{Y}}_n \overset{\sf J_1}{\Rightarrow}C \mathbf{S}_{\alpha+\beta}$ where $\mathbf{S}_{\alpha+\beta} \sim \mathbf{S}_{\alpha+\beta}(1,1)$ and $C$ is a suitable constant depending on $\alpha,\beta,\lambda,\mu$. However, since $\frac{1}{\alpha+\beta}<1$ and then $n^{1-\frac{1}{\alpha+\beta}} \to +\infty$, the previous limit also implies
	\begin{equation}\label{eq:scalinglimYn}
		\mathbf{Y}_n \overset{\sf J_1}{\Rightarrow} \mathcal{K}^{\alpha,\beta}_{\lambda,\mu}\iota
	\end{equation}
	and then \eqref{eq:limittildeN} follows by Theorem \ref{thm:contcounting}. If instead $\max\{\alpha,\beta\}=1$, we can introduce
	\begin{equation*}
		\overline{\mathbf{Y}}_n=n^{\frac{1}{2}}\left(\mathbf{Y}_n-\mathcal{K}^{\alpha,\beta}_{\lambda,\mu}\iota\right)	
	\end{equation*} 
	and then we get, by Theorem \ref{thm:DTa}, that $\overline{\mathbf{Y}}_n \overset{\sf J_1}{\Rightarrow}{\sf Var}[Y_1] \overline{\mathbf{W}}$, where $\overline{\mathbf{W}}$ is a Brownian motion. Even in this case we have \eqref{eq:scalinglimYn} and thus \eqref{eq:limittildeN} follows by Theorem \ref{thm:contcounting}. Finally, if $\alpha+\beta=1$, then we get the statement by Theorem \ref{thm:contcounting} once we use \eqref{eq:Stable1} to deduce that $m_n^{-1}\mathbf{Y}_n \overset{\sf J_1}{\Rightarrow} \iota$.
\end{proof}
Once this is clear, to prove items (i) and (ii) it is sufficient to observe that the limits of $\mathbf{q}_n$ are continuous, while the ones of $\widetilde{\mathbf{N}}_n$ are non-decreasing, hence we can use the continuous mapping theorem together with the continuity of the composition operator as in Theorem \ref{thm:contcomp}.
\hfill \qed
\subsection{Proof of Theorem \ref{thm:scalingmodel3}}
With the notation of Sections \ref{Sec:model3} and \ref{subsecMC}, let us consider the process $R(t)=r_{\widetilde{N}(t)}$ and its scaling
\begin{equation*}
	\mathbf{R}_n(t)=\frac{R(nm_nt)}{n^\delta}=n^{1-\delta}\mathbf{r}_n \circ \widetilde{\mathbf{N}}_n
\end{equation*}
so that, by Lemma \ref{lem:scaling2}, we have
\begin{equation*}
{}_3\mathbf{Q}_n=\Phi(\mathbf{R}_n)=n^{1-\delta}\Phi(\mathbf{r}_n)\circ \widetilde{\mathbf{N}}_n=\overline{\mathbf{q}}_n\circ \widetilde{\mathbf{N}}_n.
\end{equation*}
Then items (i) and (ii) follow by Lemmas \ref{lem:ovqk}, \ref{prop:tildeN} and the continuous mapping theorem, since the limits of $\overline{\mathbf{q}}_n$ are continuous, while the ones of $\widetilde{\mathbf{N}}_n$ are increasing and thus we can use Theorem \ref{thm:contcomp}. Concerning item (iii), since $\mathbf{\widetilde{N}}_n$ is surjective, we have
\begin{equation*}
	\sup_{t \ge 0}{}_3\mathbf{Q}_n(t)=\sup_{t \ge 0}\overline{\mathbf{q}}_n(t) \to 0,
\end{equation*}
where we used item (iii) of Lemma \ref{lem:ovqk}.
\hfill \qed
%theorem from a straightforward combination of  and Theorem \ref{thm:contcomp}.
\subsection{Proof of Theorem \ref{thm:limits3Q}}
Items $(i)$ and $(ii)$ are given in \cite[Theorem 3.3]{butt2023queuing}. Let us prove item $(iii)$. Set $\mathbf{N}^{\sf a}_n(t)=n^{-\delta}N^{\sf a}(n^{\frac{1}{\beta}}t)$ and $\mathbf{N}^{\sf d}_n(t)=n^{-\delta}N^{\sf d}(n^{\frac{1}{\beta}}t)$, so that
\begin{equation*}
	{}_4\mathbf{Q}_n=\Phi(\mathbf{N}^{\sf a}_n-\mathbf{N}^{\sf d}_n)
\end{equation*}
By \cite[Theorem 1]{leonenko2019limit} we have, in general,
\begin{equation}\label{eq:Lalpha}
	n^{\delta-\frac{\alpha}{\beta}}\mathbf{N}^{\sf a}_n \overset{\sf J_1}{\Rightarrow} L_\alpha,
\end{equation}
where $L_\alpha$ is an inverse $\alpha$-stable subordinator. On the other hand, by \cite[Proposition 6]{leonenko2019limit} and recalling that $n^{\delta}\mathbf{N}_n^{\sf d}$ is a fractional Poisson process with order $\beta$ and generalized rate $n\mu$, we have
\begin{equation}\label{eq:Lbeta}
	\left(\mathbf{N}_n^{\sf d}-n^{1-\delta}\mu L_\beta\right)n^{\delta-\frac{1}{2}} \overset{\sf J_1}{\Rightarrow} \mathbf{W} \circ L_\beta,
\end{equation}
where $\mathbf{W}$ is a standard Brownian motion and $L_\beta$ is an independent inverse $\beta$-stable subordinator. Combining \eqref{eq:Lalpha} and \eqref{eq:Lbeta}, we achieve, by Theorem \ref{thm:contsum},
\begin{equation*}
	\mathbf{N}_n^{\sf a}-\mathbf{N}_n^{\sf d}+n^{1-\delta}\mu L_\beta \overset{\sf J_1}{\Rightarrow} \widehat{\mathbf{W}},
\end{equation*}
where $\widehat{\mathbf{W}}=\mathbf{W}\circ L_\beta$ if $\beta>2\alpha$, $\widehat{\mathbf{W}}=\mathbf{W}\circ L_\beta+\lambda L_\alpha$ if $\beta=2\alpha$ and $\widehat{\mathbf{W}}=\lambda L_\alpha$ if $\beta<2\alpha$. Since in general $1-\delta>0$, we can use Theorem \ref{thm:contcent2} to get the statement.
%
%
%
%
%
%Let us distinguish among three cases. If $\beta>2\alpha$, then $\delta=\frac{\alpha}{\beta}>\frac{1}{2}$. As a consequence, $\mathbf{N}^{\sf a}_n \overset{\sf J_1}{\Rightarrow} \lambda L_\alpha$ by \cite[Theorem 1]{leonenko2019limit}, where $L_\alpha$ is an inverse $\alpha$-stable subordinator, and, by \cite[Proposition 6]{leonenko2019limit}
%
%
%
%
%
%If $\delta=\frac{\alpha}{\beta}=\frac{1}{2}$, then the exact same argument leads to
%\begin{equation*}
%	\mathbf{N}_n^{\sf a}-\mathbf{N}_n^{\sf d}+n^{1-\delta}\mu L_\beta \overset{\sf J_1}{\Rightarrow} \lambda L_\alpha-\mathbf{W} \circ L_\beta.
%\end{equation*}
%Next, if $\delta=\frac{1}{2}>\frac{\alpha}{\beta}$, then \eqref{eq:Lbeta} still holds true, while, since $n^{\frac{\alpha}{\beta}-\delta}\mathbf{N}_n^{\sf a}$
%
%
%
%
\hfill \qed
\subsection{Proof of Theorem \ref{thm:scalingmodel5}}\label{sec:proof5}
	First of all, let us stress that despite several (but not all) convergences can be improved from the ${\sf M_1}$ topology to the ${\sf J_1}$ one, we will always consider the ${\sf M_1}$ convergence for the ease of the reader.
	
	With the notation of Sections \ref{subs:ssqs} and \ref{sec:Mod5}, let 
	\begin{equation*}
		\mathbf{N}^{\sf a}_n(t)=\frac{N^{\sf a}(n^\gamma t)}{n}, \quad \mathbf{D}_n(t)=\frac{D(n^\gamma t)}{n}, \quad \mathbf{N}^{\sf d}_n(t)=\frac{N^{\sf d}(n^{\frac{1}{\beta}} t)}{n} 
	\end{equation*}
	\begin{equation*}
		\mathbf{I}_n(t)=\frac{I(n^\gamma t)}{n^{\frac{1}{\beta}}}, \quad \mathbf{X}_n(t)=\frac{X(n^{\gamma }t)}{n^{\frac{1}{\beta}}}, \quad \mathbf{C}_n(t)=\frac{C(n^\gamma t)}{n^{\frac{1}{\beta}}}, \quad \mathbf{CS}_n(t)=n^{-\frac{1}{\beta}}CS_{\lfloor nt \rfloor} 
	\end{equation*}
	so that
	\begin{equation*}
		{}_5\mathbf{Q}_n=\mathbf{N}^{\sf a}_n-\mathbf{D}_n, \quad \mathbf{D}_n=\mathbf{N}^{\sf d}_n \circ \mathbf{B}_n, \quad \mathbf{B}_n=n^{\gamma-\frac{1}{\beta}}\iota-\mathbf{I}_n 
	\end{equation*}
	\begin{equation*}
		\mathbf{I}_n=\Psi(\mathbf{X}_n), \quad \mathbf{X}_n=\mathbf{C}_n-n^{\gamma-\frac{1}{\beta}}\iota, \quad \mathbf{C}_n=\mathbf{CS}_n \circ \mathbf{N}_n^{\sf a}.
	\end{equation*}
	As a consequence of Theorem \ref{thm:DTa}, together with Theorems \ref{thm:contcounting} and \ref{thm:prodconv}, we have $(\mathbf{N}^{\sf d}_n,\mathbf{ CS}_n) \overset{\sf M_1}{\Rightarrow} \left(L_\beta^{\sf d}\circ \iota_{\mu^{\frac{1}{\beta}}},\mu^{-\frac{1}{\beta}}\sigma_\beta^{\sf d}\right)$ for a $\beta$-stable subordinator $\sigma_\beta^{\sf d}$ with inverse $L_\beta^{\sf d}$. Now let us distinguish among two cases. If $\alpha \ge \beta$, then we can use again Theorems \ref{thm:contcounting} and \ref{thm:prodconv} to get 
	\begin{equation*}
		(\mathbf{N}^{\sf a}_n,\mathbf{N}^{\sf d}_n,\mathbf{ CS}_n) \overset{\sf M_1}{\Rightarrow} \left(L_\beta^{\sf a}\circ \iota_{\lambda^{\frac{1}{\alpha}}},L_\beta^{\sf d}\circ \iota_{\mu^{\frac{1}{\beta}}},\mu^{-\frac{1}{\beta}}\sigma_\beta^{\sf d}\right)
	\end{equation*}
	where $L_\alpha^{\sf a}$ is the inverse of an $\alpha$-stable subordinator $\sigma_\alpha^{\sf a}$ independent of $\sigma_\beta^{\sf d}$. Now, to obtain the limit of $\mathbf{C}_n$, we need to show that the joint limit of $\mathbf{CS}_n$ and $\mathbf{N}_n^{\sf a}$ is a ${\sf M_1}$-continuity point of the composition operator, via Theorem \ref{thm:contcomp}. Let us first notice that $L_{\alpha}^{\sf a}\circ \iota_{\lambda^{\frac{1}{\alpha}}}$ is continuous hence item $(ii)$ of Theorem \ref{thm:contcomp} is trivially verified. To verify item $(i)$ of the same theorem, notice that $L_\alpha^{\sf a} \circ \iota_{\lambda^{\frac{1}{\alpha}}}$ is not strictly increasing in $t$ if and only if $L_\alpha^{\sf a}(\lambda^{\frac{1}{\alpha}}t) \in {\sf Disc}(\lambda^{-\frac{1}{\alpha}}\sigma_\alpha^{\sf a})$. Since, however, $\lambda^{-\frac{1}{\alpha}}\sigma_\alpha^{\sf a}$ and $\mu^{-\frac{1}{\alpha}}\sigma_\alpha^{\sf d}$ are independent, then by Proposition \ref{prop:noshareddisc}, the cannot share discontinuities, hence if $L_\alpha^{\sf a} \circ \iota_{\lambda^{\frac{1}{\alpha}}}(t) \in {\sf Disc}(\mu^{-\frac{1}{\beta}}\sigma_\beta^{\sf d})$, then it is strictly increasing in $t$ with probability $1$. Thus, also item $(i)$ of Theorem \ref{thm:contcomp} is verified and we can use it to get $\mathbf{C}_n \overset{\sf M_1}{\Rightarrow} \mu^{-\frac{1}{\beta}}\sigma_\beta^{\sf d} \circ L_\alpha^{\sf a} \circ \iota_{\lambda^{\frac{1}{\alpha}}}$. If $\alpha>\beta$, then $\gamma-\frac{1}{\beta}<0$ and then, by using Theorems \ref{thm:contsum},\ref{thm:continuity} and \ref{thm:prodconv}, together with the fact that $\mu^{-\frac{1}{\beta}}\sigma_\beta^{\sf d} \circ L_\alpha^{\sf a} \circ \iota_{\lambda^{\frac{1}{\alpha}}}$ is non-negative, we get, for $\alpha>\beta$,
	\begin{equation*}
		(\mathbf{N}^{\sf a}_n,\mathbf{N}^{\sf d}_n,\mathbf{B}_n) \overset{\sf M_1}{\Rightarrow} \left(L_\alpha^{\sf a}\circ \iota_{\lambda^{\frac{1}{\alpha}}},L_\beta^{\sf d}\circ \iota_{\mu^{\frac{1}{\beta}}},\mathbf{0}\right)
	\end{equation*}
	that, by a further application of Theorems \ref{thm:contcomp} and \ref{thm:contsum}, leads to the statement in item (i). If instead $\alpha=\beta$, then $\gamma=\frac{1}{\beta}$ and thus we get
	\begin{equation*}
		(\mathbf{N}^{\sf a}_n,\mathbf{N}^{\sf d}_n,\mathbf{B}_n) \overset{\sf M_1}{\Rightarrow}\left(L_\alpha^{\sf a}\circ \iota_{\lambda^{\frac{1}{\alpha}}},L_\alpha^{\sf d}\circ \iota_{\mu^{\frac{1}{\alpha}}},\iota-\Psi(\mu^{-\frac{1}{\alpha}}\sigma_\alpha^{\sf d}\circ (L_\alpha^{\sf a} \circ \iota_{\lambda^{\frac{1}{\alpha}}}-\iota))\right).
	\end{equation*}
	Notice that the third component on the right-hand side is the ${\sf M_1}$-limit of a sequence $\mathbf{B}_n$ of non-decreasing processes, hence by Proposition \ref{prop:premon} it is also non-decreasing. Hence, we can again use Theorems \ref{thm:contcomp} and \ref{thm:contsum} to get item (ii).
	
	If $\alpha<\beta$, then $\gamma=\frac{1}{\beta}$ and thus we have
	\begin{equation*}
		(\mathbf{N}^{\sf a}_n,\mathbf{N}^{\sf d}_n,\mathbf{ CS}_n) \overset{\sf M_1}{\Rightarrow} \left(\mathbf{0},L_\beta^{\sf d}\circ \iota_{\mu^{\frac{1}{\beta}}},\mu^{-\frac{1}{\beta}}\sigma_\beta^{\sf d}\right).
	\end{equation*}
	By Theorem \ref{thm:contcomp} we have $\mathbf{C}_n \overset{\sf M_1}{\Rightarrow} \mathbf{0}$ which in turn implies by $\mathbf{X}_n \overset{\sf M_1}{\Rightarrow} -\iota$. By Theorem \ref{thm:continuity} we then get $\mathbf{I}_n \overset{\sf M_1}{\Rightarrow} \iota$ and thus, still by Theorem \ref{thm:prodconv}
	\begin{equation*}
		(\mathbf{N}^{\sf a}_n,\mathbf{N}^{\sf d}_n,\mathbf{B}_n) \overset{\sf M_1}{\Rightarrow} \left(\mathbf{0},L_\beta^{\sf d}\circ \iota_{\mu^{\frac{1}{\beta}}},\mathbf{0}\right).
	\end{equation*}
	Finally, by Theorems \ref{thm:contcomp} and \ref{thm:contsum} we get Item (iii).
\hfill \qed

\appendix
\section{Preliminary materials}\label{app:A}
In this section, we give some preliminary definitions and properties concerning some of the tools we will use in the main results. 
\subsection{The space of c\'adl\'ag functions}
Let us recall that a function $f:[0,+\infty) \to \R^k$ is said to be \textbf{c\'adl\'ag} if it is right continuous and for all $t >0$ it holds
\begin{equation*}
	\lim_{s \to t-}f(s)=f(t-)\in \R^k.
\end{equation*}
We denote by $\mathbb{D}^k$ the space of c\'adl\'ag functions $f:[0,+\infty) \to \R^k$ and we omit the superscript $k$ if $k=1$. Furthermore, for $T>0$, we denote by $\mathbb{D}^k_T$ the space of c\'adl\'ag functions $f:[0,T] \to \R^k$. Together with the usual uniform topology ${\sf U}$ on $\mathbb{D}_T^k$, i.e., the one induced by the metric
\begin{equation*}
	d^{\sf U}_T(f_1,f_2)=\sup_{t \in [0,T]}|f_1(t)-f_2(t)|
\end{equation*} 
one can consider several metrics (and then topologies) that take into account the possible discontinuities of the involved functions. To do this, let us denote by $\Lambda_T$ the set of strictly increasing homeomorphisms of $[0,T]$ into itself. Clearly the identity function $\iota(t)=t$ belongs to $\Lambda_T$. The \textbf{Skorokhod ${\sf J_1}$ metric} on $\mathbb{D}_T^k$ can be then introduced as follows:
\begin{equation*}
	d^{\sf J_1}_T(f_1,f_2)=\inf_{g \in \Lambda}\max\{\sup_{t \in [0,T]}|f_1(g(t))-f_2(t)|, \sup_{t \in [0,T]}|g(t)-\iota(t)|\}.
\end{equation*} 
We will also make use of another topology. To describe it, let us introduce the \textbf{segments}
\begin{equation*}
	[x_1,x_2]:=\{tx_1+(1-t)x_2, \ t \in [0,1]\}, \ x_1,x_2 \in \R^k
\end{equation*}
and then the \textbf{thin graph} of a c\'adl\'ag function $f$ as
\begin{equation*}
	\Gamma_f:=\{(t,z) \in [0,T]\times \R^k: \ z \in [f(t-),f(t)]\}.
\end{equation*}
In practice, $\Gamma_f$ coincides with the graph of $f$ with the addition of segments connecting the discontinuity points. Now we need to understand how we can parametrize the thin graph $\Gamma_f$. If $f$ is continuous, it is clear that the couple $t \in [0,T] \mapsto (t,f(t)) \in [0,T]\times \R^k$ is a parametrization of $\Gamma_f$ that is coherent with the order in $[0,T]$. However, if $f$ is discontinuous, then we cannot \textit{transfer} directly the order of $[0,T]$ into $\Gamma_f$, due to the presence of the connecting segments. It is worth noticing that for a fixed value of $t$ we can introduce an order on $[f(t-),f(t)]$ by simply comparing the distance from one of the extrema. If we combine this order with the one on $[0,T]$ in a \textit{lexicographic-type way}, we get the following total order relation on $\Gamma_x$:
\begin{equation*}
	(t_1,z_1) <_{\Gamma_x} (t_2,z_2) \Leftrightarrow t_1<t_2 \mbox{ or } (t_1=t_2 \mbox{ and } |f(t_1-)-z_1|<|f(t_1-)-z_2|).
\end{equation*}
With this lexicographic order in mind, we can define a \textbf{strong parametrization} of $f \in \mathbb{D}$ as a non-decreasing continuous function $(u,r):[0,1] \mapsto \Gamma_f$. The set of all strong parametrizations of $f$ will be denoted as $\Pi(f)$.

Once we have a complete description of the thin graph of a c\'adl\'ag function in terms of strong parametrizations, we can use the distance of the parametrized graph to characterise a different topology on $\mathbb{D}^k_T$. The \textbf{Skorokhod ${\sf M_1}$ metric} on $\mathbb{D}^k_T$ is defined as
\begin{equation*}
	d^{\sf M_1}_T(f_1,f_2)=\inf_{\substack{(u_1,r_1) \in \Pi(f_1)\\ (u_2,r_2) \in \Pi(f_2)}}\max\{\sup_{t \in [0,1]}|u_1(t)-u_2(t)|, \sup_{t \in [0,1]}|r_1(t)-r_2(t)|\}.
\end{equation*}  
It is worth noticing that (see \cite[Theorem 12.3.2]{whitt2002stochastic})
\begin{equation*}
	d^{\sf M_1}_T(f_1,f_2)\le d^{\sf J_1}_T(f_1,f_2) \le d^{\sf U}_T(f_1,f_2).
\end{equation*}
Furthermore, recall that despite $\mathbb{D}^k_T$ can be represented as the Cartesian product of $k$ copies of $\mathbb{D}_T$, the ${\sf M_1}$ metric on $\mathbb{D}^k_T$ does not generate the product topology of $\mathbb{D}_T$. A sufficient (but not necessary) condition that guarantees ${\sf M_1}$ convergence in $\mathbb{D}^k_T$ provided that the components are ${\sf M_1}$ convergent in $\mathbb{D}_T$ is the absence of common discontinuities among them (see \cite[Theorem 12.6.1]{whitt2002stochastic}).

The very same topologies can be introduced in $\mathbb{D}^k$, by saying that a sequence $f_n$ converges towards $f$ in any of the topologies ${\sf U},{\sf J_1}$ or ${\sf M_1}$ if the restriction of $f_n$ on $[0,T]$ converge in the same topology to $f$ for any $T \not \in {\sf Disc}(f)$, where ${\sf Disc}(f)$ is the set of discontinuity points of $f$.
%there exists a sequence $T_n \to \infty$ such that all the restrictions on $[0,T_n]$ converge in the same topology in $\mathbb{D}^k_T$. 
%One can also actually extend the metric, by setting for $f_1,f_2 \in \mathbb{D}^k$
%\begin{equation*}
%	d^{{\sf T}}(f_1,f_2)=\int_0^{+\infty}e^{-t}\min\{d^j_t(f_1,f_2), \ 1\}\, dt \quad \mbox{ where } {\sf T}={\sf U}, {\sf J_1}, {\sf M_1}.
%\end{equation*}
Throughout the paper, we will denote by $f_n \overset{{\sf T}}{\to} f_n$ the convergence of a sequence of functions $(f_n)_{n \in \N} \subset \mathbb{D}^k$ to $f \in \mathbb{D}^k$ with respect to the topology ${\sf T}$, that can be either ${\sf U}$, ${\sf J}_1$ or ${\sf M}_1$. {Furthermore, notice that, according to \cite[Theorem 12.4.1, Lemma 12.4.2, Theorem 12.5.1]{whitt2002stochastic}, the three modes of convergence coincide when the limit is continuous.} For any $c \in \R$, we will denote $\iota_c(t)=ct$, where clearly $\iota=\iota_1$. Moreover, by $\mathbb{D}_{\uparrow}$ we will denote the subset of $\mathbb{D}$ composed of non-decreasing functions. 
%Finally, for $f \in \mathbb{D}^k$, we denote by ${\sf Disc}(f)$ the set of its discontinuity points.

We will use several properties of the ${\sf M}_1$ convergence on $\mathbb{D}^k$. First, let us state that ${\sf M}_1$-convergence preserves monotonicity when passing to the limit. This is indeed a consequence of \cite[Corollary 12.5.1]{whitt2002stochastic}.
%The proof of this proposition is given in Appendix \ref{app:A}.
\begin{prop}\label{prop:premon}
	Assume that $(f_n)_{n \ge 1} \subset \mathbb{D}$ is a sequence of non-decreasing functions such that $f_n \overset{\sf M_1}{\to}f \in \mathbb{D}$. Then $f$ is non-decreasing.
\end{prop}
%\begin{rmk}
%	Actually, during the proof given in Appendix \ref{app:A}, we showed the following more general statement: if $f \in \mathbb{D}_T$ is monotone on a dense subset ${\sf S}\subset [0,T]$ such that $0,T \in {\sf S}$, then $f$ is monotone on $[0,T]$. 	
%\end{rmk}
Furthermore, if $f_n \overset{{\sf M}_1}{\to} f$ in $\mathbb{D}^k$ and $g_n \overset{{\sf M}_1}{\to} g$ in $\mathbb{D}^h$, then it is not necessarily true that $(f_n,g_n) \overset{{\sf M}_1}{\to} (f,g)$ in $\mathbb{D}^{k+h}$. However, the following result (see \cite[Theorem 12.6.1]{whitt2002stochastic}) holds true. 
\begin{thm}\label{thm:prodconv}
	Let $(f_n)_{n \ge 1} \subset \mathbb{D}^k$ and $(g_n)_{n \ge 1}\subset \mathbb{D}^{h}$ such that $f_n \overset{\sf M_1}{\to}f$ and $g_n \overset{\sf M_1}{\to} g$. If ${\sf Disc}(f)\cap {\sf Disc}(g)=\emptyset$, then $(f_n,g_n)\overset{\sf M_1}{\to}(f,g) \in \mathbb{D}^{k+h}$.
\end{thm}
%For further details on the space of c\'adl\'ag functions we refer to \cite{whitt2002stochastic} and references therein.
We will also need to work with operators and functionals on $\mathbb{D}$. The most common operator is the sum, which could be not continuous in the ${\sf M}_1$-topology. A quite general continuity condition for the sum is given in \cite[Theorem 2.7.3]{whitt2002stochastic}
\begin{thm}\label{thm:contsum}
	Let $(f_n)_{n \ge 1}, (g_n)_{n \ge 1}\subset \mathbb{D}^k$ and assume that $f_n {\sf M_1}{\to} f$ and $g_n {\sf M_1}{\to} g$ in $\mathbb{D}^k$. If the following condition holds
	\begin{equation}\label{eq:samejump}
		(f(t)-f(t-))(g(t)-g(t-)) \ge 0, \ \forall t>0
	\end{equation}
	then $f_n+g_n \overset{{\sf M}_1}{\to} f+g$ in $\mathbb{D}^k$.
\end{thm}
\begin{rmk}
	Condition \eqref{eq:samejump} holds if ${\sf Disc}(f)\cap {\sf Disc}(g)=\emptyset$.
\end{rmk}
Another operator we will frequently use is the composition, that will be denoted by means of the operator $\circ:\mathbb{D}^k \times \mathbb{D}_{\uparrow} \to \mathbb{D}^k$ (see \cite[Lemma 13.2.4]{whitt2002stochastic} to guarantee that it is well-defined), as
%\begin{equation*}
%	(f\circ g)(t):=\circ(f,g)(t)=f(g(t)).
%\end{equation*}
\begin{equation*}
	(f\circ g)(t):=f(g(t)), \quad \mbox{ for } t \ge 0.
\end{equation*}
Concerning the continuity of $\circ$ with respect to the ${\sf M_1}$-convergence, we recall here \cite[Theorem 13.2.4]{whitt2002stochastic}.
\begin{thm}\label{thm:contcomp}
	Let $(f_n)_{n \ge 1}\subset \mathbb{D}^k$ and $(g_n)_{n \ge 1}\subset \mathbb{D}_{\uparrow}$ with $(f_n,g_n) \overset{\sf M_1}{\to} (f,g) \in \mathbb{D}^{k+1}$. Assume that:
	\begin{itemize}
		\item[$(i)$] If $g(t) \in {\sf Disc}(f)$, then $t \not \in {\sf Disc}(g)$ and $g$ is strictly increasing in $t$;
		\item[$(ii)$] If $t \in {\sf Disc}(g)$, then $g(t),g(t-)\not \in {\sf Disc}(f)$ and the coordinates of $f$ are monotone in $[g(t-),g(t)]$. 
	\end{itemize}
	Then $f_n \circ g_n \overset{\sf M_1}{\to}f \circ g \in \mathbb{D}^k$.
\end{thm}
We will also work with counting functions. Precisely, let $(x_n)_{n \ge 1}\subset [0,+\infty)$ be a sequence of non-negative real numbers and consider the partial sum sequence $(s_n)_{n \ge 0}$ defined as $s_0=0$ and $s_n=s_{n-1}+x_n$ for $n \ge 1$.
%\begin{equation*}
%	s_0=0, \qquad s_n=\sum_{k=1}^{n}x_k.
%\end{equation*}
Assume that $s_n \to +\infty$ and let $y(t)=s_{\lfloor t \rfloor}$, where $\lfloor t \rfloor$ is the integer part of $t$. Then $y \in \mathbb{D}_{\uparrow}$ as well as its generalized inverse,
\begin{equation*}
	y^{-1}(t)=\inf\{s \ge 0: \ y(s)>t\}.
\end{equation*}
We can then define the counting function
\begin{equation*}
	c(t)=\max\{k \ge 0: \ s_k \le t\}, \ t \ge 0
\end{equation*}
that also belongs to $\mathbb{D}_{\uparrow}$. The ${\sf M}_1$-convergence of a sequence of processes $(y_n)_{n \ge 1} \subset \mathbb{D}_{\uparrow}$, constructed as described before, is related with the ${\sf M}_1$-convergence of the sequence $(c_n)_{n \ge 1}\subset \mathbb{D}_{\uparrow}$ of the related counting processes (see \cite[Theorem 13.8.1]{whitt2002stochastic}).
\begin{thm}\label{thm:contcounting}
	Let $(x_k^n)_{n,k \ge 1},(a_n)_{n \ge 1}\subset [0,+\infty)$ be such that $a_n>0$, $a_n \to +\infty$ and, defining the partial sums $(s_k^n)_{k \ge 0}$ as $s_0^n=0$ and $s_k^n=s_{k-1}^n+x_k^n$ for $k \ge 1$,
	%	\begin{equation*}
		%		s_0^n=0, \quad, s_k^n=\sum_{j=1}^{k}x_j^n,
		%	\end{equation*}
	it holds $s_k^n \to +\infty$ as $k \to +\infty$ for fixed $n \in \N$. Define $y_n(t)=s_{\lfloor t \rfloor}^n$ and $c_n$ the respective counting function. With $\mathbf{y}_n(t)=\frac{y_n(nt)}{a_n}$ and $\mathbf{c}_n(t)=\frac{c_n(a_nt)}{n}$, then the following statements are equivalent:
	\setlength{\columnsep}{3pt}
	\begin{multicols}{4}
		\begin{itemize}
			\item[$(i)$] $\mathbf{y}_n \overset{\sf M_1}{\to} \mathbf{y}$;
			\item[$(ii)$] $\mathbf{y}_n^{-1} \overset{\sf M_1}{\to} \mathbf{y}^{-1}$;
			\item[$(iii)$] $\mathbf{c}_n \overset{\sf M_1}{\to} \mathbf{y}^{-1}$;
			\item[$(iv)$] $\mathbf{c}_n^{-1}\overset{\sf M_1}{\to}\mathbf{y}$.
		\end{itemize}
	\end{multicols}
	\setlength{\columnsep}{10pt}
\end{thm}
The same relation can be explored when the processes are subject to centering.
\begin{thm}\label{thm:countingcenter}
	Let $(x_k^n)_{n,k \ge 1},(a_n)_{n \ge 1},(m_n)_{n \ge 1}\subset [0,+\infty)$ be such that $a_n>0$, $a_n \to +\infty$, $\frac{n}{a_n}\to +\infty$, $m_n \to m \in (0,+\infty)$ and, defining the partial sums $(s_k^n)_{k \ge 0}$ as $s_0^n=0$ and $s_k^n=s_{k-1}^n+x_k^n$ for $k \ge 1$,
	%	\begin{equation*}
		%		s_0^n=0, \quad, s_k^n=\sum_{j=1}^{k}x_j^n,
		%	\end{equation*}
	it holds $s_k^n \to +\infty$ as $k \to +\infty$ for fixed $n \in \N$. Define $y_n(t)=s_{\lfloor t \rfloor}^n$ and $c_n$ the respective counting function and set
	\begin{equation*}
		\overline{\mathbf{y}}_n(t)=\frac{y_n(nt)-m_nnt}{\delta_n}, \qquad \overline{\mathbf{c}}_n(t)=\frac{c_n(nt)-m_n^{-1}nt}{\delta_n}, \qquad \mbox{ for } t \ge 0.
	\end{equation*}
	Assume that $\overline{\mathbf{y}} \in \mathbb{D}$ with $\overline{\mathbf{y}}(0)=0$. Then $\overline{\mathbf{y}}_n \overset{\sf M_1}{\to} \overline{\mathbf{y}}$ if and only if $\overline{\mathbf{c}}_n \overset{\sf M_1}{\to} -m^{-1}\overline{\mathbf{y}} \circ \iota_{m^{-1}}$.
\end{thm}

\subsection{The Skorokhod reflection map}\label{sec:SRM}
Now we want to discuss the solutions of one-dimensional and one-sided reflection problems for c\'adl\'ag paths. Precisely, for a function $f \in \mathbb{D}$, a couple $(g,a) \in \mathbb{D} \times \mathbb{D}$ is a solution of the \textbf{Skorokhod reflection problem} for the path $f$ if $g(t)\ge 0$ for all $t \ge 0$, $a$ is non-decreasing, $a(0)=0$, $\int_{[0,+\infty)}g(t)da(t)=0$, and $g(t)=f(t)+a(t)$ for all $t \ge 0$.

The aim is to construct a non-negative c\'adl\'ag path $g$, that is called the \textbf{reflected} (or \textbf{regulated}) path, by adding a \textbf{regulator} $a$, that is allowed to increase only when $g$ touches $0$, to the original path $f$. Concerning existence and uniqueness of the solution of the Skorokhod reflection problem, we have the following result (see, for instance, \cite{tanaka1979stochastic}).
\begin{thm}
	For any function $f \in \mathbb{D}$ there exists a unique solution $(g,a) \in \mathbb{D} \times \mathbb{D}$ to the Skorokhod reflection problem for the path $f$. Precisely, for all $t \ge 0$,
	\begin{equation*}
		a(t)=\sup_{0 \le s \le t}\max\{-f(s),0\}, \quad g(t)=f(t)+\sup_{0 \le s \le t}\max\{-f(s),0\}.
	\end{equation*}
\end{thm}
With this theorem in mind, we can define the \textbf{Skorokhod regulator and reflection maps}, $\Psi$ and $\Phi$, respectively, acting on $f \in \mathbb{D}$ as
\begin{equation*}
	\Psi(f)(t)=\sup_{0 \le s \le t}\max\{-f(s),0\} \qquad \Phi(f)(t)=f(t)+\sup_{0 \le s \le t}\max\{-f(s),0\},
\end{equation*}
i.e., in such a way that $(\Phi(f),\Psi(f)) \in \mathbb{D} \times \mathbb{D}$ is the solution of the Skorokhod reflection problem for the path $f$.

We are interested in the continuity properties of $\Psi$ and $\Phi$ with respect to the different topologies in $\mathbb{D}$. These properties are summarized in the following statement (see \cite[Lemmas 13.4.1, 13.5.1 and Theorems 13.4.1, 13.5.1]{whitt2002stochastic}).
\begin{thm}\label{thm:continuity}
	For all $f_1,f_2 \in \mathbb{D}$, for ${\sf T}={\sf U}$, ${\sf J_1}$ or ${\sf M_1}$ it holds
	\begin{equation*}
		d^{{\sf T}}(\Psi(f_1),\Psi(f_2)) \le d^{\sf T}(f_1,f_2), \qquad  d^{\sf T}(\Phi(f_1),\Phi(f_2)) \le 2d^{\sf T}(f_1,f_2).
	\end{equation*}
\end{thm}
We will also need some commutation properties of the reflection map with the composition with strictly monotone functions and multiplication by positive constants. The next lemma is immediate from the definition of both the reflection and the regulator map.
\begin{lem}\label{lem:scaling}
	Let $f \in \mathbb{D}$, $\lambda>0$ and $g$ be a strictly increasing homeomorphism of $[0,+\infty)$ into itself. Then
	\begin{equation*}
		\lambda\Phi(f)\circ g=\Phi(\lambda f \circ g) \qquad \lambda \Psi(f)\circ g=\Psi(\lambda f \circ g).
	\end{equation*}
\end{lem}
The previous lemma can be further extended to the case in which $g$ is not continuous while $f$ is a stepped function as follows. This is stated in the next lemma whose proof is immediate and is omitted.
\begin{lem}\label{lem:scaling2}
	Let $(a_n)_{n \ge 1} \subset \R$, $\lambda>0$ and $g \in \mathbb{D}$ be increasing, such that $\lambda^{-1}\mathbb{N} \subset g([0,+\infty))$. Set also $\mathbf{a}^\lambda(t)=a_{\lfloor \lambda t \rfloor}$. Then
	\begin{equation*}
		\Phi(\mathbf{a}^\lambda)\circ g=\Phi(\mathbf{a}^{\lambda} \circ g) \qquad \Psi(\mathbf{a}^\lambda)\circ g=\Psi(\mathbf{a}^\lambda \circ g).
	\end{equation*}
\end{lem}

Finally, let us recall the following result, as given in \cite[Theorem 13.5.2]{whitt2002stochastic}.
\begin{thm}\label{thm:contcent}
	Let $(c_n)_{n \ge 1}\subset \R$ and $(f_n)_{n \ge 1} \subset \mathbb{D}$ such that $f_n-\iota_{c_n} \overset{\sf M_1}{\to} f$. Then
	\begin{itemize}
		\item[$(i)$] If $c_n \to +\infty$, then $\Phi(f_n)-\iota_{c_n} \overset{\sf M_1}{\to} f-\min\{f(0),0\}$;
		\item[$(ii)$] If $c_n \to -\infty$, $f(0) \le 0$ and $f$ has no positive jump, then $\Phi(f_n) \overset{\sf M_1}{\to} 0$.
	\end{itemize}
\end{thm}
By a simple modification of the proof given in \cite[Theorem 6.2, Item $(ii)$]{whitt1980some}, we can also provide the following result for non-linear centering.
\begin{thm}\label{thm:contcent2}
	Let $(c_n)_{n \ge 1}\subset \R$ such that $c_n \to -\infty$, $(f_n)_{n \ge 1} \subset \mathbb{D}$ and $g \in \mathbb{D}_{\uparrow}$ with $g(t)=0$ if and only if $t=0$. If $f_n-c_ng \overset{\sf J_1}{\to} f$ for some function $f \in \mathbb{D}$, then $\Phi(f_n) \overset{\sf U}{\to} 0$.
\end{thm}
For further properties of both the reflection and regulator maps we refer to \cite{whitt2002stochastic} and references therein.

\subsection{Stable processes and a generalized functional central limit theorem}\label{sec:stab}
We fix now a probability space $(\Omega, \cF, \bP)$ on which all the involved random variables and processes will be defined, while $\E$ will denote the expectation operator. Furthermore, the symbol $\overset{d}{=}$ is used to denote equality in distribution. In the following we will make extensive use of stable distributions and L\'evy stable motions. Hence, let us recall some definitions concerning these distributions. We will mostly use the notation in \cite{whitt2002stochastic}.
\begin{defn}
	A random variable $X$ is said to be \textbf{stable} if for any $(a_1,a_2) \in \R^2$ there exist two constants $b$ and $c$ such that for any independent copies $X_1$ and $X_2$ of $X$ it holds
	\begin{equation*}
		a_1X_1+a_2X_2 \overset{d}{=}bX+c.
	\end{equation*}
\end{defn}
It can be shown (see \cite[Theorem 1.1.2]{samorodnitsky1994non}) that if $X$ is stable there exists a constant $\alpha \in (0,2]$ such that
\begin{equation*}
	b^\alpha=a_1^\alpha+a_2^\alpha.
\end{equation*}	
For $\alpha=2$, we are identifying Gaussian random variables and these are the only stable random variables with finite variance. In general, we call $\alpha$ the \textbf{stability index}. The characteristic function of an $\alpha$-stable random variable can be identified as follows (see \cite[Definition 1.1.6]{samorodnitsky1994non}):
\begin{equation*}
	-\log(\E[e^{i\theta X}])=\begin{cases} \displaystyle \gamma^\alpha|\theta|^\alpha\left(1-i\beta {\sf sign}(\theta)\tan\left(\frac{\pi \alpha}{2}\right)\right)+i\mu \theta & \alpha \not = 1 \\[7pt]
		\displaystyle 
		\gamma |\theta|\left(1+i\beta \frac{2}{\pi}{\sf sign}(\theta)\log(|\theta|)\right)+i\mu \theta & \alpha=1\end{cases}		
\end{equation*}
for three parameters $\beta \in [-1,1]$, $\gamma>0$ and $\mu \in \R$, called respectively the \textbf{skewness}, \textbf{scale} and \textbf{location} parameter. We use the notation $X \sim S_\alpha(\gamma,\beta,\mu)$ to underline that $X$ is an $\alpha$-stable random variable with given skewness, scale and location. Let us recall some properties of stable random variables (see \cite[Section 1.2]{samorodnitsky1994non}).
\begin{prop}
	The following properties hold true:
	\begin{itemize}
		\item If $X \sim S_\alpha(\gamma,\beta,\mu)$, then $X-\mu \sim S_\alpha(\gamma,\beta,0)$. 
		\item If $X \sim S_\alpha(\gamma,\beta,0)$ and $a \in \R \setminus \{0\}$, then $aX \sim S_\alpha(|a|\gamma,{\sf sign}(a)\beta,0)$.
		\item If $\alpha<1$ and $X \sim S_\alpha(\gamma,\pm 1, 0)$, then $\bP({\sf sign}(\beta)X<0)=0$.
		\item If either $\alpha \in [1,2]$ or $\alpha<1$ and $\beta \in (-1,1)$, then for all intervals $[a,b] \subset \R$ it holds $\bP(X \in [a,b])>0$.
		\item If $X \sim S_\alpha(\gamma,\beta,\mu)$ and $p>0$, then $\E[|X|^p]<\infty$ if and only if $p \in (0,\alpha)$.
	\end{itemize}
\end{prop}
In case $\beta=\pm 1$, we say that $X \sim S_\alpha(\gamma,\beta,\mu)$ is \textbf{totally skewed}, while if $\beta=0$ we say that $X$ is \textbf{symmetric}. Let us now recall that a L\'evy process $X=\{X(t), \ t \ge 0\}$ is a c\'adl\'ag stochastic process whose increments are stationary and independent. We always assume that $X(0)=0$ a.s. 
\begin{defn}
	A L\'evy process $X=\{X(t), t \ge 0\}$ is called a \textbf{stable L\'evy motion} if there exist $\alpha \in (0,2]$, $\beta \in [-1,1]$ and $\gamma>0$ such that
	\begin{equation*}
		X(t+s)-X(s)\sim S_\alpha(\gamma t^{\frac{1}{\alpha}},\beta,0).
	\end{equation*}
	We will denote this as $X \sim \mathbf{S}_\alpha(\gamma,\beta)$. Notice that, if $\alpha=2$, $X \sim \mathbf{S}_\alpha(1,\beta)$ is a Brownian motion independently of $\beta$.
\end{defn}
Let us consider a special case. If $\beta=1$, $\alpha \in (0,1)$ and $\gamma=\left(\cos\left(\frac{\pi \alpha}{2}\right)\right)^{-\frac{1}{\alpha}}$, we say that $\sigma_\alpha \sim S_\alpha(\gamma,1)$ is an $\alpha$-\textbf{stable subordinator}. It is well-known that $\alpha$-stable subordinators are $\alpha^{-1}$-self similar, i.e., for any $z>0$, $\sigma_\alpha \circ \iota_z \overset{d}{=}z^{\frac{1}{\alpha}}\sigma_\alpha$. Furthermore, $\alpha$-stable subordinators are characterised in law by means of the Laplace transform $\E[e^{-z\sigma_\alpha(t)}]=e^{-tz^\alpha}$. Furthermore $\alpha$-stable subordinators are a.s. strictly increasing. Given an $\alpha$-stable subordinator $\sigma_\alpha$, we can define its first passage time process
\begin{equation*}
	L_\alpha(t):=\min\{s \ge 0: \ \sigma_\alpha(s) \ge t\},
\end{equation*}
that is called \textbf{inverse $\alpha$-stable subordinator}. Notice that $L_\alpha$ is not a Markov process, differently from $\sigma_\alpha$. Furthermore, $L_\alpha$ has a.s. continuous sample paths that are only non-decresing (that is to say that they can still admit intervals of constancy). Finally, while $\sigma_\alpha$ is $\alpha^{-1}$-self similar, $L_\alpha$ is $\alpha$-self similar, i.e., for any $z>0$, $L_\alpha \circ \iota_z \overset{d}{=}z^{\alpha}L_\alpha$. For further details on inverse $\alpha$-stable subordinators, see \cite{meerschaert2013inverse}. 

In the following, we will need some generalizations of the Central Limit Theorem for stochastic processes in $\mathbb{D}$. Consider any random variable $X$ and let $F_X$ be its cumulative distribution function. Define, for $t \in \R$,
\begin{equation*}
	\overline{F}_X(t)=1-F_X(t)=\bP(X>t) \quad \mbox{ and } \quad \overline{G}_X(t)=\bP(|X|>t).
\end{equation*}
For $\alpha \in (0,2]$, we say that $X$ belongs to the \textbf{normal domain of attraction of a stable law} if either $\alpha<2$ and there exist two constants $C>0$ and $\beta \in [-1,1]$ such that
\begin{equation*}
	\lim_{t \to +\infty}\frac{t^\alpha \overline{G}_X(t)}{C}=1 \quad \mbox{ and } \quad \lim_{t \to +\infty}\frac{\overline{F}_X(t)}{\overline{G}_X(t)}=\frac{1+\beta}{2}
\end{equation*} 
or $\alpha=2$ and $\E[X^2]<\infty$. The following theorem, that summarizes  \cite[Theorem 4.3.2]{whitt2002stochastic} and \cite[Theorem 4.5.3]{whitt2002stochastic} holds true.
\begin{thm}\label{thm:DTa}
	Let $\alpha \in (0,2]$. Assume that $X$ belongs to the normal domain of attraction of a stable law with parameters $C>0$ and $\beta \in [-1,1]$ if $\alpha<2$ or, if $\alpha=2$, set $C=\sqrt{{\sf Var}[X]}$. Let $(X_n)_{n \ge 1}$ be a sequence of i.i.d. random variables with the same distribution as $X$ and define the sequence of processes $(Y_n)_{n \ge 1}$ as  $Y_n(t)=\sum_{k=1}^{\lfloor nt \rfloor}X_k$. Then
	\begin{equation*}
		\left(\frac{C_\alpha}{C}\right)^{\frac{1}{\alpha}}n^{-\frac{1}{\alpha}}(Y_n-\iota_{m_n}) \overset{{\sf J}_1}{\Rightarrow} \mathbf{S},
	\end{equation*}
	where $\mathbf{S} \sim \mathbf{S}_\alpha(1,\beta)$ and
	\begin{equation}\label{eq:Ca}
		C_\alpha:=\begin{cases}
			\displaystyle \frac{1-\alpha}{\Gamma(2-\alpha)\cos\left(\frac{\pi \alpha}{2}\right)} & \displaystyle\substack{\alpha \in (0,2) \\ \alpha \not = 1} \\[7pt]
			\displaystyle 2/\pi & \alpha=1\\
			1 & \alpha=2
		\end{cases},
		\quad m_n=\begin{cases}
			0 & \alpha<1 \\[7pt]
			\displaystyle \frac{\pi Cn^{2}}{2}\E\left[\sin\left(\frac{2X}{\pi Cn}\right)\right] & \alpha=1 \\[7pt]
			n\E[X] & \alpha \in (1,2].
		\end{cases}
	\end{equation}
\end{thm}
\subsection{Mittag-Leffler and related distributions}\label{sec:ML}
Throughout this paper we will consider some special distributions which generalize the notion of exponential and Erlang distributions. First of all, we recall that the \textbf{Mittag-Leffler function} of order $\alpha \in (0,1]$ is the entire function $E_\alpha: \mathbb{C} \to \mathbb{C}$ defined as
\begin{equation*}
	E_\alpha(z)=\sum_{k=0}^{+\infty}\frac{z^k}{\Gamma(\alpha k+1)}.
\end{equation*}
It is clear that if $\alpha=1$, then $E_1(z)=e^z$. For general properties of the Mittag-Leffler functions, we mainly refer to \cite{gorenflo2020mittag}. 

In our case, we will use such a function to construct a generalization of the exponential distribution as follows.
\begin{defn}
	A random variable $X$ is said to be a \textbf{Mittag-Leffler random variable} of \textbf{order} $\alpha \in (0,1]$ and \textbf{generalized rate} $\lambda>0$ if, for $t \ge 0$,
	\begin{equation*}
		\bP(X>t)=E_\alpha(-\lambda t^\alpha)
	\end{equation*}
	and we denote such a property as $X \sim {\sf ML}_\alpha(\lambda)$. If $\alpha=1$, $X$ is an exponential random variable of rate $\lambda$ and we will denote this as $X \sim {\sf Exp}(\lambda)$.
\end{defn}
This class of distributions has been considered, for instance, in \cite{mainardi2004fractional,mainardi2005renewal}, in the context of renewal processes. Let us also stress that Mittag-Leffler distributions are strictly related to stable subordinators. Indeed, the following property, which is a consequence of \cite{bingham1971limit}, holds true.
\begin{prop}\label{eq:sigma}
	Let $\sigma_\alpha$ be an $\alpha$-stable subordinator and $T$ an independent exponential random variable of parameter $\lambda>0$. Then $\sigma_\alpha(T)\sim {\sf ML}_\alpha(\lambda)$.
\end{prop}
The latter has been shown, for instance, in \cite[Lemma 2]{ascione2022skorokhod}. A direct consequence of Proposition \ref{eq:sigma} (together with the $\alpha^{-1}$-self similarity of the $\alpha$-stable subordinator) is the following scaling property for the Mittag-Leffler distribution.
\begin{prop}\label{prop:scalingML}
	If $X \sim {\sf ML}_\alpha(\lambda)$ then $\lambda^{\frac{1}{\alpha}}X\sim {\sf ML}_\alpha(1)$
\end{prop}
Together with the Mittag-Leffler distribution, we need to consider a possible generalization obtained by summing a fixed number of independent Mittag-Leffler random variable. Before proceeding, let us recall that the \textbf{Prabhakar function} with parameters $\gamma,\beta>0$, $\alpha \in (0,1]$ is the entire function $E_{\alpha,\beta}^{\gamma}:\mathbb{C} \to \mathbb{C}$ defined as
\begin{equation*}
	E_{\alpha,\beta}^{\gamma}(z)=\sum_{k=0}^{+\infty}\frac{\Gamma(\gamma+k)z^k}{\Gamma(\gamma)k!\Gamma(\alpha k+\beta)}.
\end{equation*}
The Prabhakar function has been first introduced in \cite{prabhakar1971singular}. Let us stress that the following derivative formula holds:
\begin{equation}\label{eq:derPrab}
	\der{}{z}E^{\gamma}_{\alpha,\beta}(z)=\frac{\Gamma(\gamma+1)}{\Gamma(\gamma)}E^{\gamma+1}_{\alpha,\alpha+\beta}(z).
\end{equation}
\begin{defn}\label{def:GE}
	A random variable $X$ is said to be a \textbf{generalized Erlang random variable} of order $\alpha \in (0,1]$, rate $\lambda>0$ and shape parameter $k \in \N$ if, for $t \ge 0$,
	\begin{equation*}
		\bP(X>t)=\sum_{n=0}^{k-1}\sum_{j=0}^{+\infty}(-1)^j\binom{n+j}{n}\frac{(\lambda t^{\alpha})^{j+n}}{\Gamma(\alpha(j+n)+1)}=\sum_{n=0}^{k-1}(\lambda t^{\alpha})^n E_{\alpha,n\alpha+1}^{n+1}(-\lambda t^\alpha)
	\end{equation*}
	and we denote such a property as $X \sim {\sf GE}_\alpha(k,\lambda)$.
\end{defn}
This random variable has been introduced in \cite{mainardi2004fractional} and the following property holds true.
\begin{prop}
	Let $X_1,\dots,X_k$ be independent Mittag-Leffler random variables of order $\alpha \in (0,1]$ and rate $\lambda>0$. Then $X_1+\cdots+X_k \sim {\sf GE}_\alpha(k,\lambda)$.
\end{prop}
Furthermore, we can actually characterise the probability density function of a random variable $X \sim {\sf GE}_\alpha(k,\lambda)$ as:
\begin{equation}\label{eq:densGE}
	f_{{\sf GE}_\alpha(k,\lambda)}(t)=\alpha k \lambda^k t^{k\alpha-1}E_{\alpha,k\alpha+1}^{k+1}(-\lambda t^\alpha).
\end{equation}
This formula follows by using the derivative relation \eqref{eq:derPrab} to differentiate term by term the sum in $\bP(X>t)$ in Definition \ref{def:GE}. Furthermore, if $k=1$, we get the density of a Mittag-Leffler random variable $X \sim {\sf ML}_\alpha(\lambda)$, hence we denote it by $f_{{\sf ML}_\alpha(\lambda)}$.
Our aim consists in using these distributions in order to characterise inter-arrival and service times of some models of queue. For this reason, we need also to introduce a very special counting process.
\subsection{The fractional Poisson process}
Now we define the analogous of the Poisson process in the context of Mittag-Leffler random variables.
\begin{defn}
	Let $(T_k)_{k \ge 0}$ be a sequence of i.i.d. random variables such that $T_1 \sim {\sf ML}_\alpha(\lambda)$ for some $\alpha \in (0,1]$ and $\lambda>0$. For all $t \ge 0$ we consider the random variable
	\begin{equation*}
		N_\alpha(t):=\max\left\{k \ge 0: \  \sum_{j=1}^{k}T_j \le t\right\},
	\end{equation*}
	where $\sum_{j=1}^{0}T_j=0$. The process $N_\alpha$ is called a \textbf{fractional Poisson process} of order $\alpha$ and rate $\lambda$.
\end{defn}
While the definition of the fractional Poisson process as a renewal process has been first given in \cite{mainardi2004fractional}, its name refers to its first introduction in \cite{laskin2003fractional}. Indeed, one can show that the probability masses $p_k(t;\alpha,\lambda)=\bP(N_\alpha(t)=k)$ are the unique bounded solutions of the fractional Cauchy problem
\begin{equation*}
	\begin{cases}
		\displaystyle \dersup{}{t}{\alpha}p_0(t;\alpha,\lambda)=-\lambda p_0(t;\alpha,\lambda)\\[7pt]
		\displaystyle \dersup{}{t}{\alpha}p_k(t;\alpha,\lambda)=-\lambda (p_k(t;\alpha,\lambda)-p_{k-1}(t;\alpha,\lambda)) & k=1,2,\cdots \\[7pt]
		p_0(0;\alpha,\lambda)=\delta_{0,k} & k=0,1,2,\cdots 
	\end{cases}
\end{equation*}
where the operator $\dersup{}{t}{\alpha}$ is the \textbf{Caputo fractional derivative}, i.e.
\begin{equation*}
	\dersup{}{t}{\alpha}f(t)=\frac{1}{\Gamma(1-\alpha)}\der{}{t}\int_0^t (t-\tau)^{-\alpha}(f(\tau)-f(0))\, d\tau.
\end{equation*}
For further details on fractional operators and differential equations, we refer to \cite{ascione2023fractional,Samko,kilbas2006theory} and references therein. Here we will not focus on such aspect. In any case, let us observe, as given in \cite{laskin2003fractional}, that for $\alpha \in (0,1)$, $\lambda,t>0$ and $k=0,1,2,\cdots$
\begin{equation}\label{eq:pkfPP}
	p_k(t;\alpha,\lambda)=\sum_{j=0}^{+\infty}(-1)^j\binom{j+k}{j}\frac{(\lambda t^\alpha)^{j+k}}{\Gamma(\alpha(j+k)+1)}=(\lambda t^\alpha)^kE^{k+1}_{\alpha,k\alpha+1}(-\lambda t^\alpha).
\end{equation}

A further important property of the fractional Poisson process is that it can be defined by means of a time-change procedure. 
%To do this, we need to introduce the notion of inverse $\alpha$-stable subordinator. Consider an $\alpha$-stable subordinator $\sigma_\alpha$. Then we define the first passage time process
%\begin{equation*}
%	L_\alpha(t)=\min\{s \ge 0: \ \sigma_\alpha(s) \ge t\},
%\end{equation*}
%that is called \textbf{inverse $\alpha$-stable subordinator}. It is worth noticing that while $\sigma_\alpha$ is strictly increasing but only c\'adl\'ag, $L_\alpha$ is only non-decreasing, since we have intervals of constancy corresponding to the jumps of $\sigma_\alpha$, but it admits continuous sample paths. Furthermore, while $\sigma_\alpha$ is $\alpha^{-1}$-self similar, $L_\alpha$ is $\alpha$-self-similar. For further details on this process, we refer to \cite{meerschaert2013inverse}. 
Indeed, the following result has been proved in \cite{meerschaert2011fractional}.
\begin{prop}
	Let $N$ be a Poisson process with rate $\lambda>0$ and let $L_\alpha$ be an independent inverse $\alpha$-stable subordinator. Then $N_\alpha(t):=N(L_\alpha(t))$ is a fractional Poisson process of order $\alpha$ and rate $\lambda$.
\end{prop}
It is worth noticing that a fractional Poisson process $N_\alpha$ is not a Markov process. However, one can prove that $N_\alpha$ still satisfies a quite interesting property, which will play a prominent role. Indeed, if we define the soujourn time process
%To better understand this, for a given a.s. c\'adl\'ag process $X=\{X(t), \ t \ge 0\}$, let us introduce the \textbf{sojourn time process}
\begin{equation*}
	J_{N_\alpha}(t)=t-\sup\{s \le t: \ N_\alpha(s)\not = N_\alpha(t)\},
\end{equation*}
where we set, in this case, $\sup \emptyset =0$, then the following result holds true:
\begin{prop}
	The couple $(N_\alpha,J_{N_\alpha})$ is a strong Markov process. Furthermore, as a consequence, if we denote $\tau_n=\min\{t \ge 0: \ N_\alpha(t)=n\}$, it holds
	\begin{equation*}
		\bP(\{N_\alpha(t), \ t \ge \tau_n \} \in B \mid \cF_\tau)=\bP(\{N_\alpha(t)+n, \ t \ge 0 \} \in B)
	\end{equation*}
	for any $B \in \mathcal{B}(\mathbb{D})$. In particular, $N_\alpha$ is a semi-Markov process.
\end{prop}
The latter is a direct consequence of the fact that $N_\alpha$ is a renewal process, with the same arguments used in \cite[Chapter III.3]{gikhman2004theory}. For further details on semi-Markov process, we refer to \cite{harlamov2013continuous} and references therein.
Throughout the paper, we will also use another quite important property: two independent fractional Poisson processes cannot share any discontinuity point. This is actually common to any couple of independent stochastically continuous process, as recalled in the following proposition.
%, whose proof is given in Appendix \ref{app:B}
\begin{prop}\label{prop:noshareddisc}
	Let $X=\{X(t), \ t \ge 0\}$ and $Y=\{Y(t), \ t \ge 0\}$ be two independent c\'adl\'ag processes such that $X$ is stochastically continuous.
	%, i.e. such that for all $\varepsilon>0$ and $t \ge 0$ it holds
	%\begin{equation*}
	%	\lim_{h \to 0}\bP(|X(t+h)-X(t)|>\varepsilon)=0.
	%\end{equation*}
	Let also $F:\mathbb{D}^k \to [0,+\infty)$ be a measurable functional. Then $X(F(Y)-)=X(F(Y))$ almost surely.
	%	\begin{equation}\label{eq:measF}
		%		\bP(\not = X(F(Y)))=0.
		%	\end{equation}
	In particular, as a consequence ${\sf Disc}(X) \cap {\sf Disc}(Y)=\emptyset$ almost surely.
	%	\begin{equation}\label{eq:nocomdisc}
		%		\bP()=1.
		%	\end{equation}
\end{prop}
The proof is based on a simple conditioning argument and thus is omitted. 
As a consequence, together with the fact that, clearly, the composition of a stochastically continuous process and a continuous process is still stochastically continuous, we have the following statement.
%Let us finally show the following easy property that we will use in the following.
\begin{lem}\label{lem:nonsim}
	Let $N_\alpha$, $N_\beta$ be two independent fractional Poisson processes with orders respectively $\alpha,\beta \in (0,1]$. Then ${\sf Disc}(N_\alpha)\cap {\sf Disc}(N_\beta)=\emptyset$ almost surely. 
\end{lem}
%\begin{rmk}
%	As a direct consequence of the previous lemma we get that if $\tau$ is a stopping time with respect to the natural filtration of $N_\alpha$, then
%	\begin{equation*}
	%		\bP(N_\beta(\tau-)\not = N_\beta(\tau))=0.
	%	\end{equation*}
%	Indeed, if $\tau$ is a stopping time with respect to $N_\alpha$, then it is measurable with respect to the $\sigma$-algebra generated by $N_\alpha$ as a random variable with values in $\mathbb{D}$. Hence, by the Doob-Dynkin Lemma, $\tau=F(N_\alpha)$ for some measurable function $F:\mathbb{D} \to [0,+\infty)$ and then the previous lemma applies.
%\end{rmk}
%\begin{rmk}
%	Let us also stress that with a slight modification of the previous proof one can show that two stochastically continuous c\'adl\'ag independent processes have no common discontinuities with probability $1$.
%\end{rmk}
In the following we will also make use of the distribution of the \textit{last discontinuity time} of a fractional Poisson process in an interval. This distribution is given by the following result.
%, whose proof is left in Appendix \ref{app:C}
\begin{thm}\label{thm:lastjumptime}
	Let $\alpha \in (0,1]$, $\lambda>0$ and consider a fractional Poisson process $N_\alpha$ of order $\alpha$ and rate $\lambda$. For any $t>0$ consider
	\begin{equation}\label{eq:lastjumptime0}
		\tau[t]=\sup\{s \le t: \ N_\alpha(s)\not = N_\alpha(s-)\},
	\end{equation}
	where $\sup \emptyset =0$ and set also $\tau[0]=0$. Let $F_{\sf LJ}(s;\alpha,\lambda,t)=\bP(\tau[t] \le s)$. Then $F_{\sf LJ}(s;\alpha,\lambda,t)=1$ if $s \ge t$, while, for $0 \le s \le t$, it is given in \eqref{eq:lastjumptime}.
\end{thm}
\begin{proof}
	Let us denote by $(\tau_k)_{k \ge 1}$ the sequence of discontinuities of $N_\alpha$ and set $\tau_0=0$. We first show that, for $s \le t$,
	\begin{equation}\label{eq:condlastjump}
		\bP(\tau_n \le s \mid N_\alpha(t)=n)=\frac{\alpha n}{t^{\alpha n}E_{\alpha,n\alpha+1}^{n+1}(-\lambda t^\alpha)} \int_0^s E_\alpha(-\lambda (t-z)^\alpha)z^{n\alpha-1}E_{\alpha,n\alpha+1}^{n+1}(-\lambda z^\alpha)\, dz.
	\end{equation}
	To evaluate this, let us first consider the case $n=1$. Then we have
	\begin{equation}\label{eq:n10}
		\bP(\tau_1 \le s \mid N_\alpha(t)=1)=\frac{\bP(\tau_1 \le s, N_\alpha(t)=1)}{\bP(N_\alpha(t)=1)}.
	\end{equation}
	Now we let $T_2=\tau_2-\tau_1$, where we recall that $T_2 \sim {\sf ML}_\alpha(\lambda)$ is independent of $\tau_1$, and we observe that
	\begin{equation*}
		\bP(\tau_1 \le s, N_\alpha(t)=1)=\bP(\tau_1 \le s, T_1 > t-\tau_1).
	\end{equation*}
	By the tower property of conditional expectation, we have
	\begin{equation}\label{eq:n1}
		\bP(\tau_1 \le s, N_\alpha(t)=1)=\E\left[\mathbf{1}_{[0,s]}(\tau_1)\E\left[\mathbf{1}_{(t-\tau_1,+\infty)}(T_1)\mid \tau_1\right]\right].
	\end{equation}
	Concerning the inner conditional expectation, let
	\begin{equation}\label{eq:n11}
		F(s;t)=\E\left[\mathbf{1}_{(t-s,+\infty)}(T_1)\right]=E_\alpha(-\lambda(t-s)^\alpha)
	\end{equation}
	so that using \eqref{eq:n11} into \eqref{eq:n1}, we have
	\begin{align}\label{eq:n13}
		\begin{split}
			\bP(\tau_1 \le s, N_\alpha(t)=1)&=\E\left[\mathbf{1}_{[0,s]}(\tau_1)E_\alpha(-\lambda(t-\tau_1)^\alpha)\right]\\
			&=\lambda \alpha \int_{0}^{s}E_\alpha(-\lambda(t-z)^\alpha)z^{\alpha-1}E^2_{\alpha,\alpha+1}(-\lambda z^\alpha)\, dz,	
		\end{split}
	\end{align}
	where we used \eqref{eq:densGE} with $k=1$ since $\tau_1 \sim {\sf ML}_\alpha(\lambda)$.
	%Now recall that
	%\begin{equation*}
	%	\der{}{z}E_\alpha(z)=E_{\alpha,\alpha+1}^2(z)
	%\end{equation*}
	%hence, since $\tau_1 \sim {\sf ML}_\alpha(\lambda)$, we can use its density in \eqref{eq:n12} to get
	%\begin{equation}\label{eq:n13}
	%	\bP(\tau_1 \le s, N_\alpha(t)=1)=\lambda \alpha \int_{0}^{s}E_\alpha(-\lambda(t-z)^\alpha)z^{\alpha-1}E^2_{\alpha,\alpha+1}(-\lambda z^\alpha)\, dz.
	%\end{equation}
	Plugging \eqref{eq:n13} into \eqref{eq:n10} we have
	\begin{equation}\label{eq:n10fin}
		\bP(\tau_1 \le s \mid N_\alpha(t)=1)=\frac{\alpha}{t^\alpha E^{2}_{\alpha,k\alpha+1}(-\lambda t^\alpha)} \int_{0}^{s}E_\alpha(-\lambda(t-z)^\alpha)z^{\alpha-1}E^2_{\alpha,\alpha+1}(-\lambda z^\alpha)\, dz.
	\end{equation}
	Now let us consider the case in which $n \ge 2$. Then we have
	\begin{equation}\label{eq:ngen0}
		\bP(\tau_n \le s \mid N_\alpha(t)=n)=\frac{\bP(\tau_n \le s, N_\alpha(t)=n)}{\bP(N_\alpha(t)=n)}.
	\end{equation}
	Let, again, $T_{n+1}=\tau_{n+1}-\tau_n$, and rewrite the numerator in \eqref{eq:ngen0} as
	\begin{equation*}
		\bP(\tau_n \le s, N_\alpha(t)=n)=\bP(\tau_n \le s, T_{n+1}>t-\tau_n).
	\end{equation*}
	Again, since $T_{n+1}$ is independent of $\tau_n$, we can use the same argument as before to achieve
	\begin{align}\label{eq:ngen2}
		\begin{split}
			\bP(\tau_n \le s, N_\alpha(t)=n)&=\E\left[\mathbf{1}_{[0,s]}(\tau_n)E_\alpha(-\lambda(t-\tau_n)^{\alpha})\right]\\
			&=\alpha n \lambda^n \int_0^s E_\alpha(-\lambda (t-z)^\alpha)z^{n\alpha-1}E_{\alpha,n\alpha+1}^{n+1}(-\lambda z^\alpha)\, dz,	
		\end{split}
	\end{align}
	where, since $\tau_n \sim {\sf GE}_\alpha(n,\lambda)$, we used \eqref{eq:densGE}.
	% hence we need the density of a generalized Erlang random variable. To get it, we need to differentiate term by term the summation in Definition \ref{def:GE}. Indeed, taking the derivative in the $t$ variable, using \eqref{eq:derPrab} and changing sign we get
	%\begin{align*}
	%	-\der{}{t}\bP(\tau_n>t)&=\sum_{k=0}^{n-1}\alpha(k+1) \lambda^{k+1} t^{(k+1)\alpha-1}E_{\alpha,(k+1)\alpha+1}^{k+2}(-\lambda t^\alpha)\\
	%	&\qquad -\sum_{k=1}^{n-1}k\alpha \lambda^k t^{k\alpha-1}E_{\alpha,k\alpha+1}^{k+1}(-\lambda t^{\alpha})\\
	%	&=\sum_{k=0}^{n-1}\alpha(k+1) \lambda^{k+1} t^{(k+1)\alpha-1}E_{\alpha,(k+1)\alpha+1}^{k+2}(-\lambda t^\alpha)\\
	%	&\qquad -\sum_{k=0}^{n-2}(k+1)\alpha \lambda^{k+1} t^{(k+1)\alpha-1}E_{\alpha,(k+1)\alpha+1}^{k+2}(-\lambda t^{\alpha})\\
	%	&=\alpha n\lambda^{n} t^{n\alpha-1}E_{\alpha,n\alpha+1}^{n+1}(-\lambda t^\alpha)
	%\end{align*}
	%thus \eqref{eq:ngen1} becomes
	%\begin{equation}\label{eq:ngen2}
	%	\bP(\tau_n \le s, N_\alpha(t)=n)=\alpha n \lambda^n \int_0^s E_\alpha(-\lambda (t-z)^\alpha)z^{n\alpha-1}E_{\alpha,n\alpha+1}^{n+1}(-\lambda z^\alpha)\, dz.
	%\end{equation}
	Finally, plugging this equation into \eqref{eq:ngen0} and using \eqref{eq:pkfPP}, we get
	\begin{equation*}
		\bP(\tau_n \le s \mid N_\alpha(t)=n)=\frac{\alpha n}{t^{\alpha n}E_{\alpha,n\alpha+1}^{n+1}(-\lambda t^\alpha)} \int_0^s E_\alpha(-\lambda (t-z)^\alpha)z^{n\alpha-1}E_{\alpha,n\alpha+1}^{n+1}(-\lambda z^\alpha)\, dz.
	\end{equation*}
	Now that we have \eqref{eq:condlastjump}, we can prove \eqref{eq:lastjumptime}. Indeed, by using both \eqref{eq:condlastjump} and \eqref{eq:pkfPP}, we have
	\begin{align*}
		\bP(\tau[t] \le s)&=\sum_{n=0}^{+\infty}\bP(\tau[t] \le s \mid N_\alpha(t)=n)\bP(N_\alpha(t)=n)\\
		&=E_\alpha(-\lambda t^\alpha)+\sum_{n=1}^{+\infty}\alpha n \lambda^n \int_0^s E_\alpha(-\lambda (t-z)^\alpha)z^{n\alpha-1}E_{\alpha,n\alpha+1}^{n+1}(-\lambda z^\alpha)\, dz,
	\end{align*}
	ending the proof.
\end{proof}
%For $s \in [0,t]$, we denote by $F_{\sf LJ}(s;\alpha,\lambda,t)$ the right-hand side of \eqref{eq:lastjumptime}, while we set $F_{\sf LJ}(s;\alpha,\lambda,t)=1$ if $s \ge t$.

%For simplicity, in what follows, we will denote
%\begin{align}\label{lastjumpdist}
%	\begin{split}
	%	F_{\sf LJ}(s;\alpha,\lambda,t)&:=E_\alpha(-\lambda t^\alpha)\\
	%	&\quad +\sum_{n=1}^{+\infty}\alpha n \lambda^n \int_0^s E_\alpha(-\lambda (t-z)^\alpha)z^{n\alpha-1}E_{\alpha,n\alpha+1}^{n+1}(-\lambda z^\alpha)\, dz
	%	\end{split}
%\end{align}
%for $s \in [0,t]$, while $F_{\sf LJ}(s;\alpha,\lambda,t)=1$ if $s \ge t$. 
\section{Proof of the results in Section \ref{sec:models}}\label{app:B}
\subsection{Proof of Theorem \ref{thm:interarrival2}}\label{sec:ptinter12}
It is clear that $T_1 \sim {\sf mML}_{\alpha,\beta}(\lambda,\mu)$ as the first discontinuity time must be an arrival time, i.e. $T_1=Y_1$. Now let us consider $T_k$ with $k \ge 2$. Then notice that under ${}_2Q(A_{k-1})=n$ one can only have at most $n$ departures before the arrival of new customer. Let $N_S[t_1,t_2]$ be the number of departures in the time interval $[t_1,t_2]$ and let us write
	\begin{multline*}
		\bP(T_k>t \mid {}_2Q(A_{k-1})=n)\\
		=\sum_{j=0}^{n}\bP(T_k>t, \ \mid {}_2Q(A_{k-1})=n, N_S[A_{k-1},A_k]=j)\bP(N_S[A_{k-1},A_k]=j).
	\end{multline*}
	Let $\widetilde{k}$ be such that $\sum_{h=0}^{\widetilde{k}}Y_{h}=A_{k-1}$. Then, under $N_S[A_{k-1},A_k]=j$ we have that $A_k=\sum_{h=0}^{\widetilde{k}+j+1}Y_{h}$ and then
	\begin{align*}
		\bP&(T_k>t, \ \mid {}_2Q(A_{k-1})=n, N_S[A_{k-1},A_k]=j)\\
		&=\bP\left(\sum_{h=\widetilde{k}+1}^{\widetilde{k}+j+1}Y_{h}>t, \ \mid {}_2Q(A_{k-1})=n, N_S[A_{k-1},A_k]=j\right)\\
		&=1-\bP\left(\sum_{h=\widetilde{k}+1}^{\widetilde{k}+j+1}Y_{h} \le t\right)=1-(F_{\alpha,\beta}^{\lambda,\mu})^{\ast (j+1)}(t).
	\end{align*}
	On the other hand, notice that $N_S[A_{k-1},A_k]$ is geometrically distributed with parameter $p_{\alpha,\beta}^{\lambda,\mu}$ (here $j$ is the number of failures).
	%, i.e.,
	%\begin{equation*}
	%\bP(N_S[{}_2A_{k-1},{}_2A_k]=j)=(1-p_{\alpha,\beta}^{\lambda,\mu})^jp_{\alpha,\beta}^{\lambda,\mu}
	%\end{equation*}
	Hence we get
	\begin{equation*}
		\bP(T_k>t \mid {}_2Q(A_{k-1})=n)
		=\sum_{j=0}^{n}(1-(F_{\alpha,\beta}^{\lambda,\mu})^{\ast (j+1)}(t))(1-p_{\alpha,\beta}^{\lambda,\mu})^jp_{\alpha,\beta}^{\lambda,\mu}.
	\end{equation*}
	%ending the proof.
\hfill \qed
\subsection{Proof of Theorem \ref{thm:service2}}\label{sec:ptser12}
	Let us consider $S_k$ with $k \ge 1$. If $A_{k}>D_{k-1}$, then $S_k=D_k-A_k$. Let $N_{A}[t_1,t_2]$ be the number of arrivals in the interval $(t_1,t_2]$. Then
	\begin{multline*}
		\bP(S_k>t \mid A_k>D_{k-1})\\
		=\sum_{j=0}^{+\infty}\bP(S_k>t, \ \mid A_k>D_{k-1}, N_A[A_{k},D_k]=j)\bP(N_A[A_{k},D_k]=j).
	\end{multline*}
	Now, let $A_k=Y_{\widetilde{k}}$ for some $\widetilde{k}$. Then, under $N_A[A_{k},D_k]=j$, we have $D_k=Y_{\widetilde{k}+j+1}$ and then
	\begin{align*}
		\bP(&S_k>t, \ \mid A_k>D_{k-1}, N_A[A_{k},D_k]=j)\\
		&=\bP\left(\sum_{h=\widetilde{k}+1}^{\widetilde{k}+j+1}Y_h>t \mid {}_2Q(D_{k-1})=0, N_A[A_{k},D_k]=j\right)=1-(F_{\alpha,\beta}^{\lambda \mu})^{\ast (j+1)}(t).
		%		\\
		%		&=1-(F_{\alpha,\beta}^{\lambda \mu})^{\ast (j+1)}(t).
	\end{align*}
	On the other hand, $N_A[A_{k},D_k]$ is geometrically distributed with parameter $1-p_{\alpha,\beta}^{\lambda,\mu}$, hence
	%\begin{equation*}
	%	\bP(N_A[{}_2A_{k},{}_2D_k]=j)=(p_{\alpha,\beta}^{\lambda,\mu})^j(1-p_{\alpha,\beta}^{\lambda,\mu}),
	%\end{equation*}
	%hence
	\begin{equation*}
		\bP(S_k>t \mid A_k>D_{k-1})
		=\sum_{j=0}^{+\infty}(1-(F_{\alpha,\beta}^{\lambda \mu})^{\ast (j+1)}(t))(p_{\alpha,\beta}^{\lambda,\mu})^j(1-p_{\alpha,\beta}^{\lambda,\mu}).
	\end{equation*}
	If $A_k<D_{k-1}$, then $S_k=D_k-D_{k-1}$ and we apply the same argument using $N_A[D_{k-1},D_k]$ in place of $N_A[A_{k},D_k]$.
	% Then we have, with the same argument,
	%\begin{align*}
	%	\bP&({}_2S_k>t \mid {}_2A_k<{}_2D_{k-1})\\
	%	&=\sum_{j=0}^{+\infty}\bP({}_2S_k>t, \ \mid {}_2A_k<{}_2D_{k-1}, N_A[{}_2D_{k-1},{}_2D_k]=j)\bP(N_A[{}_2D_{k-1},{}_2D_k]=j)\\
	%	&=\sum_{j=0}^{+\infty}(1-(F_{\alpha,\beta}^{\lambda \mu})^{\ast (j+1)}(t))(p_{\alpha,\beta}^{\lambda,\mu})^j(1-p_{\alpha,\beta}^{\lambda,\mu}).
	%\end{align*}
\hfill \qed
\subsection{Proof of Proposition \ref{prop:comparison}}\label{proofcomp}
	The fact that $p_{\lambda,\mu}^{\alpha,\beta}=\frac{1}{2}$ implies that $\bP({T}_1>t)=\bP({S}_1>t)$, and thus equality in \eqref{eq:comparsurv}, is evident from \eqref{eq:serv2Q} and \eqref{eq:interarrival2}. The strict inequality  in \eqref{eq:comparsurv} in case $p_{\lambda,\mu}^{\alpha,\beta}>1/2$ still follows by \eqref{eq:serv2Q} and \eqref{eq:interarrival2} once we observe that for $j \ge 2$ it holds
	\begin{equation*}
		(1-p_{\alpha,\beta}^{\lambda,\mu})^jp_{\alpha,\beta}^{\lambda,\mu}<(p_{\alpha,\beta}^{\lambda,\mu})^j(1-p_{\alpha,\beta}^{\lambda,\mu}),
	\end{equation*}
	while the reverse one holds if $p_{\lambda,\mu}^{\alpha,\beta}<1/2$.
\hfill \qed 
\subsection{Proof of Proposition \ref{prop:27}}\label{proofprop27}
	Let $J_0=0$ and
	\begin{equation*}
		J_k=\inf\{t>J_{k-1}: \ {}_3{Q}(t)\not = {}_3{Q}(t-)\}
	\end{equation*}
	be the event times of ${}_3{Q}$ and define $q_k={}_3{Q}(J_k)$. It is clear that $q_0=0$. Furthermore, if $q_{k-1}>0$, then, by definition of Skorokhod reflection,
	\begin{equation*}
		q_k-q_{k-1}={}_3{Q}(J_k)-{}_3{Q}(J_{k-1})=R(J_k)-R(J_{k-1})=r_k-r_{k-1},
	\end{equation*}
	since on $[J_{k-1},J_{k})$ the regulator $\Psi(R)$ is constant. If $q_{k-1}=0$, then clearly $q_k=1$. This shows, in particular, that $q_k$ is a Markov chain whose transition probability is specified in \eqref{eq:mCren1}. Now let $\widetilde{Y}_k=J_k-J_{k-1}$ for $k \ge 1$. If $q_{k-1}>0$, then $\widetilde{Y}_k$ is an inter-jump time of $R$, i.e. it coincides with $Y_j$ for some $j \in \N$ and then
	\begin{equation*}
		\bP(\widetilde{Y}_k>t \mid q_{k-1})=1-F_{\alpha,\beta}^{\lambda,\mu}(t).
	\end{equation*}
	If instead $q_{k-1}=0$, then $J_k$ must be an upward jump time of $R$, hence the argument in the discussion above Theorem \ref{thm:interarrival22} applies.
\hfill \qed
\subsection{Proof of Proposition \ref{prop:distserv3}}\label{sec:restlessproof1}
	%	Let $\{d_n\}_{n \ge 0}$ the sequence of random times defined as $d_0=0$ and
	%	\begin{equation*}
		%		d_n=\min\{t>d_{n-1}: \ N_{\sf d}(t-)\not = N_{\sf d}(t)\}, \ n \ge 1.
		%	\end{equation*}
	%	Notice that $\{d_n\}_{n \ge 0}$ constitute a subsequence of $\tau_n$.
	%	
	Let us first show the first equality in \eqref{eq:service1}. If $A_n<D_{n-1}$, then clearly $n \ge 2$ and
	\begin{equation*}
		S_n=D_n-D_{n-1}.	
	\end{equation*}
	Furthermore for $t \in [D_{n-1},D_n)$, it must hold ${}_4Q(t) \ge 1$ a.s. Then, by definition of ${}_4Q$, $D_{n-1}$ and $D_n$ are two subsequent discontinuities of $N_{\sf d}$ and then, under $A_n<D_n$, we have $S_n=D_n-D_{n-1} \sim {\sf ML}_\beta(\mu)$. This shows the first equality in \eqref{eq:service1}.
	
	Concerning the second equality, if $A_n>D_{n-1}$, then we have
	\begin{equation*}
		S_n=D_n-A_n=D_n-\tau_{\sf d}[A_n]-(A_n-\tau_{\sf d}[A_n]),
	\end{equation*}
	where, for $t \ge 0$, $\tau_{\sf d}[t]$ is defined in \eqref{eq:lastjumptime0} with respect to the fractional Poisson process $N_{\sf d}$. To determine the distribution of $S_n$, let us condition further with respect to $\tau_{\sf d}[A_n]$, getting, under $A_n>D_{n-1}$, 
	\begin{align*}
		\bP&(S_n>t \mid A_n,D_{n-1},\tau_{\sf d}[A_n])\\
		&=\bP(D_n-\tau_{\sf d}[A_n]>A_n-\tau_{\sf d}[A_n]+t \mid A_n,D_{n-1},\tau_{\sf d}[A_n]).
	\end{align*}
	Once we condition with respect to $A_n$, $D_{n-1}$ and $\tau_{\sf d}[A_n]$, for $D_{n-1}<A_n$, the variable $D_n-\tau_{\sf d}[A_n]$ is a Mittag-Leffler random variable (since it is the time between two subsequent discontinuities of $N_{\sf d}$) conditioned on the event $D_n-\tau_{\sf d}[A_n]>A_n-\tau_{\sf d}[A_n]$ (since we know that $N_{\sf d}(A_n)=N_{\sf d}(\tau_{\sf d}[A_n])$), hence
	\begin{align*}
		\bP&(S_n>t \mid A_n,D_{n-1},\tau_{\sf d}[A_n])=\frac{E_\beta(-\mu(A_n-\tau_{\sf d}[A_n]+t)^\beta)}{E_\beta(-\mu(A_n-\tau_{\sf d}[A_n])^\beta)}\\
		&=\frac{E_\beta(-\mu(A_n-D_{n-1}-(\tau_{\sf d}[A_n]-D_{n-1})+t)^\beta)}{E_\beta(-\mu(A_n-D_{n-1}-(\tau_{\sf d}[A_n]-D_{n-1}))^\beta)}.
	\end{align*}
	Next, since $D_{n-1}$ is a regeneration time of $N_{\sf d}$, we have
	\begin{equation*}
		\tau_{\sf d}[A_n]-D_{n-1}\overset{d}{=}\tau_{\sf d}[A_n-D_{n-1}].
	\end{equation*}
	Now let $T>0$ and notice that
	\begin{equation*}
		F_t(T):=\E\left[\frac{E_\beta(-\mu(T-\tau_{\sf d}[T]+t)^\beta)}{E_\beta(-\mu(T-\tau_{\sf d}[T])^\beta)}\right]=\int_{[0,T]}\frac{E_\beta(-\mu(T-s+t)^\beta)}{E_\beta(-\mu(T-s)^\beta)}dF_{\sf LJ}(s;\beta,\mu,T).
	\end{equation*}
	Then, we have, by the tower property of the conditional expectation, under $A_n>D_{n-1}$,
	\begin{align*}
		\bP&(S_n>t \mid A_n,D_{n-1})=F_t(A_n-D_{n-1}),
	\end{align*}
	ending the proof.
\hfill \qed
%\section{Proof of Theorem \ref{thm:lastjumptime}}\label{app:C}
\section*{Acknowledgements}
The authors are partially supported by the PRIN 2022 project 2022XZSAFN: \textit{Anomalous Phenomena on Regular and Irregular Domains: Approximating Complexity for the Applied Sciences}. GA is partially supported by the GNAMPA group of the Istituto Nazionale di Alta Matematica. LC is partially supported by the GNCS group of the Istituto Nazionale di Alta Matematica.

\bibliographystyle{abbrv}
\bibliography{biblio}
\end{document}